\documentclass[10pt]{amsart}
      \usepackage{amsmath,amsfonts}
\def\mapright#1{\smash{
\mathop{\rg}\limits^{#1}}}
\def\mapdown#1{\bigg\downarrow
\rlap{$\vcenter{\hbox{$\scriptstyle#1$}}$}}

\def\rg{\hbox to 30pt{\rightarrowfill}}
\def\lg{\hbox to 30pt{\leftarrowfill}}

      \parskip\smallskipamount
          \newtheorem{theorem}{Theorem}[section]
      \newtheorem{definition}[theorem]{Definition}
      \newtheorem{proposition}[theorem]{Proposition}
      \newtheorem{corollary}[theorem]{Corollary}
      \newtheorem{lemma}[theorem]{Lemma}

      \newtheorem{remark}[theorem]{Remark}
      \makeatletter
      \@addtoreset{equation}{section}
      \makeatother

      \newcommand{\BB}{{\mathbb B}}
      \newcommand{\CC}{{\mathbb C}}
      \newcommand{\NN}{{\mathbb N}}
      
      \newcommand{\ZZ}{{\mathbb Z}}
      \newcommand{\DD}{{\mathbb D}}
      \newcommand{\RR}{{\mathbb R}}
      \newcommand{\FF}{{\mathbb F}}
      \newcommand{\TT}{{\mathbb T}}

      \newcommand{\cA}{{\mathcal A}}
      \newcommand{\cB}{{\mathcal B}}
      \newcommand{\cC}{{\mathcal C}}
      \newcommand{\cD}{{\mathcal D}}
      \newcommand{\cE}{{\mathcal E}}
      
      \newcommand{\cG}{{\mathcal G}}
      \newcommand{\cH}{{\mathcal H}}
      \newcommand{\cK}{{\mathcal K}}
      \newcommand{\cL}{{\mathcal L}}
      \newcommand{\cM}{{\mathcal M}}
      \newcommand{\cN}{{\mathcal N}}
      \newcommand{\cO}{{\mathcal O}}
      \newcommand{\cQ}{{\mathcal Q}}

      \newcommand{\cS}{{\mathcal S}}
      \newcommand{\cT}{{\mathcal T}}
      \newcommand{\cU}{{\mathcal U}}
      \newcommand{\cV}{{\mathcal V}}

      \newcommand{\supp}{\hbox{\rm{supp}}\,}
      \newcommand{\rank}{\hbox{\rm{rank}}\,}

      \newdimen\expt
      \def\boxit#1{\setbox0\hbox{$\displaystyle{#1}$}
            \hbox{\lower.4\expt
       \hbox{\lower3\expt\hbox{\lower\dp0
            \hbox{\vbox{\hrule height.4\expt
       \hbox{\vrule width.4\expt\hskip3\expt
            \vbox{\vskip3\expt\box0\vskip2\expt}%
       \hskip3\expt\vrule width.4\expt}\hrule height.4\expt}}}}}}
      
\begin{document}
       \pagestyle{myheadings}
      \markboth{ Gelu Popescu}{  Curvature invariant on noncommutative  polyballs  }

      \title [ Holomorphic automorphisms  of noncommutative polyballs   ]
      { Holomorphic automorphisms  of noncommutative polyballs }
        \author{Gelu Popescu}
\date{February 18, 2015}
      \thanks{Research supported in part by an NSF grant}

      \address{Department of Mathematics, The University of Texas
      at San Antonio \\ San Antonio, TX 78249, USA}
      \email{\tt gelu.popescu@utsa.edu}

\begin{abstract}
The  abstract regular polyball ${\bf B_n}$, ${\bf n}=(n_1,\ldots, n_k)\in \NN^k$,   is a noncommutative analogue of the scalar polyball \, $(\CC^{n_1})_1\times\cdots \times  (\CC^{n_1})_1$, which has been recently studied in connection with operator model theory, curvature invariant, and Euler characteristic. In this paper, we study free holomorphic functions on ${\bf B_n}$ and provide analogues of several classical results from complex analysis such as: Abel theorem, Hadamard formula, Cauchy inequality, Schwarz lemma, and maximum principle. These results are used together with a class of  noncommutative Berezin transforms to obtain a complete description of the group $\text{\rm Aut}({\bf B_n})$ of all free holomorphic automorphisms of the polyball ${\bf B_n}$, which is an analogue of Rudin's characterization of the holomorphic automorphisms of the polydisc,  and show that
$$
\text{\rm Aut}({\bf B_n}) \simeq\text{\rm Aut}((\CC^{n_1})_1\times\cdots \times  (\CC^{n_1})_1).
$$
If ${\bf m}=(m_1,\ldots, m_q)\in \NN^q$, we show that the polyballs
${\bf B_n}$ and
${\bf B_m}$ are free biholomorphic equivalent
if and only if    $k=q$ and there is a permutation $\sigma$ such that $m_{\sigma(i)}=n_i$ for any $i\in \{1,\ldots, k\}$. This extends  Poincar\' e's  result  that  the open unit ball of $\CC^n$ is not biholomorphic  equivalent to the polydisk $\DD^n$, to our noncommutative setting.

The abstract polyball  ${\bf B_n}$ has a universal model ${\bf S}:=\{{\bf S}_{i,j}\}$ consisting of left creation operators acting on the tensor product  $F^2(H_{n_1})\otimes \cdots \otimes F^2(H_{n_k})$ of full Fock spaces. The noncommutative Hardy algebra ${\bf F}_{\bf n}^\infty$ (resp.\,the polyball algebra $\boldsymbol\cA_{\bf n}$) is the weakly closed (resp.\,norm closed) non-selfadjoint  algebra generated by $\{{\bf S}_{i,j}\}$ and the identity.
We prove that
$$
\text{\rm Aut}_{\boldsymbol\cA_{\bf n}}(C^*({\bf S}))\simeq \text{\rm Aut}_u(\boldsymbol\cA_{\bf n})\simeq  \text{\rm Aut}_u({\bf F}_{\bf n}^\infty)\simeq \text{\rm Aut}({\bf B_n}),
$$
where $\text{\rm Aut}_{\boldsymbol\cA_{\bf n}}(C^*({\bf S}))$  is the group of automorphisms of  the Cuntz-Toeplitz  $C^*$-algebra $C^*({\bf S})$  which leaves invariant the   noncommutative polyball algebra $\boldsymbol\cA_{\bf n}$, and
$\text{\rm Aut}_u(\boldsymbol\cA_{\bf n})$ (resp.\,${\rm Aut}_u({\bf F}_{\bf n}^\infty)$) is the group of unitarily implemented  automorphisms of  the  algebra $\boldsymbol\cA_{\bf n}$ (resp.\,${\bf F}_{\bf n}^\infty)$). Moreover, we obtain  formulas for the elements of these automorphism groups in terms of noncommutative Berezin transforms. As a consequence, we obtain a concrete description for the group of automorphisms of  the tensor product $\cT_{n_1}\otimes\cdots \otimes\cT_{n_k}$ of Cuntz-Toeplitz algebras which leave invariant the tensor product $\cA_{n_1}\otimes_{min}\cdots \otimes_{min}\cA_{n_k}$ of noncommutative disc algebras, which extends Voiculescu's result when $k=1$.

We  prove that the free holomorphic automorphism group \
 $\text{\rm Aut}({\bf B_n})$ is a   $\sigma$-compact, locally
compact topological group with respect to the topology induced by the metric
$$
d_{\bf B_n}(\phi,\psi):=\|\phi -\psi\|_\infty +
\|\phi^{-1}(0)-\psi^{-1}(0)\|,\qquad \phi,\psi\in  \text{\rm Aut}({\bf B_n}).
$$
Finally, we obtain  a  concrete unitary projective representation of the topological group
  $\text{\rm Aut}({\bf B_n})$  in terms
      of noncommutative Berezin kernels associated with regular polyballs.

\end{abstract}

      \maketitle

\section*{Contents}
{\it

\quad Introduction

\begin{enumerate}
   \item[1.]     Noncommutative   polyballs and free holomorphic functions
   \item[2.]   Maximum principle and Schwarz type results
   \item[3.]     Holomorphic automorphisms of noncommutative polyballs
   \item[4.]    Automorphisms of Cuntz-Toeplitz algebras
       \item[5.] Automorphisms of the   polyball algebra
    $\cA ({\bf B_n})$ and  the Hardy algebra $H^\infty({\bf B_n})$
    \item[6.] The automorphism group $Aut({\bf B_n})$ and unitary projective representations
\end{enumerate}

\quad References

}

\bigskip

\section*{Introduction}

Recently (see \cite{Po-Berezin3}, \cite{Po-Berezin-poly}), we have tried to unify the multivariable operator model theory  for ball-like domains and commutative polydiscs, and extend it to a more general class of noncommutative polydomains (which includes the regular polyballs) and use it to  develop a theory of free holomorphic functions. What is remarkable for these polydomains is that they have universal models, in a certain sense, which are (weighted)  creation operators acting on tensor products of full Fock spaces. The model theory and  the free holomorphic function theory  on these polydomains are related, via noncommutative Berezin transforms, to the study of the operator algebras generated by the universal models, as well as to the theory of functions in several complex variable (\cite{Kr}, \cite{Ru1}, \cite{Ru2}). It is the interplay between these three fields that lead to a rich analytic function theory on these noncommutative polydomains. Our work on curvature invariant \cite{Po-curvature-polyball} and Euler characteristic \cite{Po-Euler-charact} on noncommutative regular  polyballs has led us to   study  the free holomorphic automorphisms  of these polyballs, which is the goal of the present paper and  continues   work  of Voiculescu \cite{Vo},  of  Davidson and Pitts \cite{DP2},  of Helton, Klep, McCullough and Singled \cite{HKMS}, of Benhida and Timotin \cite{BeTi2}, \cite{BeTi}, and of the author in \cite{Po-automorphism}, \cite{Po-classification}.  In a related context we mention the work of
Muhly and Solel  \cite{MuSo3}, and of
   Power and Solel \cite{PS}.

Throughout this paper, $B(\cH)$ stands for the algebra of all bounded linear operators on a Hilbert space $\cH$. We denote by  $B(\cH)^{n_1}\times_c\cdots \times_c B(\cH)^{n_k}$, where $n_i \in\NN:=\{1,2,\ldots\}$,
   the set of all tuples  ${\bf X}:=({ X}_1,\ldots, { X}_k)$ in $B(\cH)^{n_1}\times\cdots \times B(\cH)^{n_k}$
     with the property that the entries of ${X}_s:=(X_{s,1},\ldots, X_{s,n_s})$  are commuting with the entries of
      ${X}_t:=(X_{t,1},\ldots, X_{t,n_t})$  for any $s,t\in \{1,\ldots, k\}$, $s\neq t$.
  Note that  the operators $X_{s,1},\ldots, X_{s,n_s}$ are not necessarily commuting.
  Let ${\bf n}:=(n_1,\ldots, n_k)$ and define  the polyball
  $${\bf P_n}(\cH):=[B(\cH)^{n_1}]_1\times_c \cdots \times_c [B(\cH)^{n_k}]_1,
  $$
  where
    $$[B(\cH)^{n}]_1:=\{(X_1,\ldots, X_n)\in B(\cH)^{n}:\ \|  X_1X_1^*+\cdots +X_nX_n^*\|<1\}, \quad n\in \NN.
    $$
    If $A$ is  a positive invertible operator, we write $A>0$. The {\it regular polyball} on the Hilbert space $\cH$  is defined by
$$
{\bf B_n}(\cH):=\left\{ {\bf X}\in {\bf P_n}(\cH) : \ {\bf \Delta_{X}}(I)> 0  \right\},
$$
where
 the {\it defect mapping} ${\bf \Delta_{X}}:B(\cH)\to  B(\cH)$ is given by
$$
{\bf \Delta_{X}}:=\left(id -\Phi_{X_1}\right)\circ \cdots \circ\left(id -\Phi_{ X_k}\right),
$$
 and
$\Phi_{X_i}:B(\cH)\to B(\cH)$  is the completely positive linear map defined by
$$\Phi_{X_i}(Y):=\sum_{j=1}^{n_i}   X_{i,j} Y X_{i,j} ^*, \qquad Y\in B(\cH).
$$ We call  the operator ${\bf \Delta_{X}}(I)$ the defect of ${\bf X}$. Note that if $k=1$, then ${\bf B_n}(\cH)$ coincides with the noncommutative unit ball $[B(\cH)^{n_1}]_1$.
   We remark that the scalar representation of
  the the ({\it abstract})
 {\it regular polyball} ${\bf B}_{\bf n}:=\{{\bf B_n}(\cH):\ \cH \text{\ is a Hilbert space} \}$ is
   ${\bf B_n}(\CC)={\bf P_n}(\CC)= (\CC^{n_1})_1\times \cdots \times (\CC^{n_k})_1$.

Let $H_{n_i}$ be
an $n_i$-dimensional complex  Hilbert space with orthonormal basis $e^i_1,\ldots, e^i_{n_i}$.
  We consider the {\it full Fock space}  of $H_{n_i}$ defined by
$F^2(H_{n_i}):=\CC 1 \oplus\bigoplus_{p\geq 1} H_{n_i}^{\otimes p},$
where  $H_{n_i}^{\otimes p}$ is the
(Hilbert) tensor product of $p$ copies of $H_{n_i}$. Let $\FF_{n_i}^+$ be the unital free semigroup on $n_i$ generators
$g_{1}^i,\ldots, g_{n_i}^i$ and the identity $g_{0}^i$.
  Set $e_\alpha^i :=
e^i_{j_1}\otimes \cdots \otimes e^i_{j_p}$ if
$\alpha=g^i_{j_1}\cdots g^i_{j_p}\in \FF_{n_i}^+$
 and $e^i_{g^i_0}:= 1\in \CC$.
  The length of $\alpha\in
\FF_{n_i}^+$ is defined by $|\alpha|:=0$ if $\alpha=g_0^i$  and
$|\alpha|:=p$ if
 $\alpha=g_{j_1}^i\cdots g_{j_p}^i$, where $j_1,\ldots, j_p\in \{1,\ldots, n_i\}$.
 We  define
 the {\it left creation  operator} $S_{i,j}$ acting on the  Fock space $F^2(H_{n_i})$  by setting
$
S_{i,j} e_\alpha^i:=  e^i_{g_j^i \alpha}$, $\alpha\in \FF_{n_i}^+,
$
 and
 the operator ${\bf S}_{i,j}$ acting on the tensor  product
$F^2(H_{n_1})\otimes\cdots\otimes F^2(H_{n_k})$ by setting
$${\bf S}_{i,j}:=\underbrace{I\otimes\cdots\otimes I}_{\text{${i-1}$
times}}\otimes S_{i,j}\otimes \underbrace{I\otimes\cdots\otimes
I}_{\text{${k-i}$ times}},
$$
where  $i\in\{1,\ldots,k\}$ and  $j\in\{1,\ldots,n_i\}$. We introduce the noncommutative Hardy algebra ${\bf F}_{\bf n}^\infty$ (resp. the polyball algebra $\boldsymbol\cA_{\bf n}$) as the weakly closed (resp.~norm closed) non-selfadjoint  algebra generated by $\{{\bf S}_{i,j}\}$ and the identity.

 We proved in \cite{Po-Berezin-poly} (in a more general setting)   that
    ${\bf X}\in B(\cH)^{n_1}\times\cdots \times B(\cH)^{n_k}$
    is a {\it pure} element in the regular polyball  ${\bf B_n}(\cH)^-$, i.e.  $\lim_{q_i\to \infty}\Phi_{X_i}^{q_i}(I)=0$ in the weak operator topology,  if and only if
there is a Hilbert space $\cK$ and a subspace $\cM\subset F^2(H_{n_1})\otimes \cdots \otimes  F^2(H_{n_k})\otimes \cK$ invariant under each  operator ${\bf S}_{i,j}\otimes I$ such that
$X_{i,j}^*=({\bf S}_{i,j}^*\otimes I)|_{\cM^\perp}$ under an appropriate identification of $\cH$ with $\cM^\perp$.
The $k$-tuple ${\bf S}:=({\bf S}_1,\ldots, {\bf S}_k)$, where  ${\bf S}_i:=({\bf S}_{i,1},\ldots,{\bf S}_{i,n_i})$, is  an element  in the
regular polyball $ {\bf B_n}(\otimes_{i=1}^kF^2(H_{n_i}))^-$ and  plays the role of  {\it universal model} for
   the abstract
  polyball ${\bf B}_{\bf n}^-:=\{{\bf B_n}(\cH)^-:\ \cH \text{\ is a Hilbert space} \}$. The existence of the universal model  will play an important role in our paper, since it will make the connection between noncommutative function theory, operator algebras, and complex function theory in several variables. The latter is due to the fact that the joint eingenvectors for the universal model are parameterized by the scalar polyball $(\CC^{n_1})_1\times\cdots \times (\CC^{n_k})_1$ via the Berezin transforms (see \cite{Po-Berezin3}).

In Section 1, we  show that the regular polyball ${\bf B_n}$ is a logarithmically convex complete Reinhardt noncommutative domain, in an appropriate sense. We provide   characterizations for free holomorphic functions on polyballs in terms of their universal models, obtain an analogue of Abel theorem from complex analysis, Cauchy type inequalities for the coefficients of free holomorphic functions, and an analogue of Liouville's theorem for entire functions.  We prove that
the largest regular polyball $\gamma{\bf B_n}$, $\gamma\in[0,\infty]$, which is included in the universal domain of convergence of a formal power series  $\varphi$ in indeterminates $\{Z_{i,j}\}$ and representation
 $\varphi=\sum\limits_{(\alpha)} A_{(\alpha)}\otimes Z_{(\alpha)}$ with   $A_{(\alpha)}\in B(\cK)$,  is given by the relation

 $$
 \frac{1}{\gamma}:=\limsup_{(p_1,\ldots, p_k)\in \ZZ_+^k} \|\sum\limits_{{ \alpha_i\in \FF_{n_i}^+, |\alpha_i|=p_i}\atop{i\in \{1,\ldots, k\} }}A_{(\alpha)}^*A_{(\alpha)}\|^{\frac{1}{2(p_1+\cdots +p_k)}},
 $$
where  ${Z}_{(\alpha)}:= {Z}_{1,\alpha_1}\cdots {Z}_{k,\alpha_k}$ if  $(\alpha):=(\alpha_1,\ldots, \alpha_k)\in \FF_{n_1}^+\times \cdots \times\FF_{n_k}^+$ and   $Z_{i,\alpha_i}:=Z_{i,j_1}\cdots Z_{i,j_p}\in \FF_{n_i}^+$
  if   $\alpha_i=g_{j_1}^i\cdots g_{j_p}^i$.

 In Section 2, we prove a Schwarz type result (\cite{Ru2}) which states that  if $F:{\bf B_n}(\cH)\to B(\cH)^p$ is a bounded  free holomorphic function with $\|F\|_\infty\leq 1$ and  $F(0)=0$, then
 $$
 \|F({\bf X})\|\leq m_{\bf B_n}({\bf X})<1\quad \text{ and } \quad  m_{\bf B_n}({\bf X})\leq \|{\bf X}\|, \qquad {\bf X}\in {\bf B_n}(\cH),
 $$
 where $m_{\bf B}$ is the Minkovski functional associated with the regular polyball ${\bf B_n}$.
 This result is used to prove a maximum principle for bounded free holomorphic functions on polyballs  which states that if   $F:{\bf B_n}(\cH)\to B(\cH)$ is a bounded  free holomorphic function and there exists ${\bf X}_0\in {\bf B_n}(\cH)$  such that
 $$
 \|F({\bf X})\|\leq \|F({\bf X}_0)\|, \qquad {\bf X}\in {\bf B_n}(\cH),
 $$
  then $F$ must be a constant. The results of Section 2 will play an important role in the next sections.

 In Section 3, we give a complete description of the  free holomorphic  automorphisms of the polyball ${\bf B_n}$ (see Theorem \ref{structure}), which extends  Rudin's characterization of the  holomorphic automorphisms of  the polydisc \cite{Ru2},  and prove some of  their  basic properties (see Theorem  \ref{structure2}).  We also present an analogue   of  Poincar\' e's result  \cite{Kr}, that  the open unit ball of $\CC^n$ is not biholomorphic  equivalent to the polydisk $\DD^n$, for noncommutative regular polyballs.
More precisely,  if  ${\bf n}=(n_1,\ldots, n_k)\in \NN^k$ and  ${\bf m}=(m_1,\ldots, m_q)\in \NN^q$, we show that there is a biholomorphic map between the polyballs
${\bf B_n}$ and
${\bf B_m}
$
if and only if    $k=q$ and there is a permutation $\sigma $ of the set $\{1,\ldots, k\}$  such that $m_{\sigma(i)}=n_i$ for any $i\in \{1,\ldots, k\}$.  Moreover, any free biholomorphic function $F:{\bf B_n}\to
{\bf B_m}$ is up to a permutation of $(m_1,\ldots, m_k)$  an automorphism of the noncommutative regular polyball ${\bf B_n}$. This  resembles  the classical result of Ligocka \cite{L} and Tsyganov \cite{T} concerning biholomorphic automorphisms of product spaces with nice boundaries. The results of  this section are used to  show that
$$
{\bf Aut}({\bf B_n})\simeq \text{\rm Aut}((\CC^{n_1})_1\times\cdots \times (\CC^{n_k})_1).
$$
More precisely, we prove that the map
$\Lambda$
defined by
$$\Lambda({\bf \Psi})({\bf z}):=({\boldsymbol\cB_z} [\hat {\bf \Psi}_1],\ldots,{\boldsymbol\cB_z} [\hat {\bf \Psi}_k]) \qquad {\bf z}\in (\CC^{n_1})_1\times\cdots \times (\CC^{n_k})_1,
$$
is a group isomorphism, where $\hat {\bf \Psi}:=\text{\rm SOT-}\lim_{r\to 1} {\bf \Psi}(r{\bf S})$ is the boundary function of ${\bf \Psi}=(\Psi_1,\ldots, \Psi_k)\in \text{\rm Aut}({\bf B_n})$ with respect to the universal model ${\bf S}$, and ${\boldsymbol\cB_z}$ is the noncommutative Berezin transform at ${\bf z}$.

In Section 4, we  prove that
 any automorphism  $\Gamma$ of the   Cuntz-Toeplitz $C^*$-algebra $C^*({\bf S})$, generated by the universal model ${\bf S}=\{{\bf S}_{i,j}\}$,  which leaves invariant the   noncommutative polyball algebra $\boldsymbol\cA_{\bf n}$, i.e. $\Gamma(\boldsymbol\cA_{\bf n})=\boldsymbol\cA_{\bf n}$,  has the form

$$
\Gamma(g):=\boldsymbol\cB_{\hat \Psi}[g]={\bf K}_{\hat {\Psi}}[g\otimes I_{\cD_{\hat {\Psi}}}]{\bf K}_{\hat {\Psi}}^*,\qquad
g\in  C^*({\bf S}),
$$
where  $\Psi\in \text{\rm Aut}({\bf B_n})$ and $\boldsymbol\cB_{\hat \Psi}$ is the noncommutative Berezin transform at   the boundary function $\hat{\Psi}$.  In this case,
the noncommutative Berezin kernel ${\bf K}_{\hat {\Psi}}$ is a unitary operator and $\Gamma$ is a unitarily implemented automorphism of
  $C^*({\bf S})$.
Moreover, we have
$$
\text{\rm Aut}_{\boldsymbol\cA_{\bf n}}(C^*({\bf S}))\simeq  \text{\rm Aut}({\bf B_n}),
$$
where $\text{\rm Aut}_{\boldsymbol\cA_{\bf n}}(C^*({\bf S}))$  is the group of automorphisms of   $C^*({\bf S})$ which leave invariant the   noncommutative polyball algebra $\boldsymbol\cA_{\bf n}$.  As a consequence, we obtain a concrete description for the group of automorphisms of  the tensor product $\cT_{n_1}\otimes\cdots \otimes\cT_{n_k}$ of Cuntz-Toeplitz algebras which leave invariant the tensor product $\cA_{n_1}\otimes_{min}\cdots \otimes_{min}\cA_{n_k}$ of noncommutative disc algebras, which extends Voiculescu's result when $k=1$.
In particular, each holomorphic automorphism of the regular polyball ${\bf B_n}$  induces an automorphism of the  tensor product  of Cuntz algebras $\cO_{n_1}\otimes \cdots \otimes  \cO_{n_k}$ which leaves invariant the non-self-adjoint  subalgebra $\cA_{n_1}\otimes_{min}\cdots \otimes_{min} \cA_{n_k}$.

 In Section 5, we  prove that
 any unitarily implemented  automorphism of  the noncommutative polyball algebra  $\boldsymbol\cA_{\bf n} $  (resp.\,the noncommutative Hardy algebra     ${\bf F}_{\bf n}^\infty $)  is the Berezin transform  of a boundary function   $\hat\Psi$,  where  $\Psi\in \text{\rm Aut}({\bf B_n})$. Moreover, we have
$$
\text{\rm Aut}_u(\boldsymbol\cA_{\bf n})\simeq \text{\rm Aut}_u({\bf F}_{\bf n}^\infty)\simeq \text{\rm Aut}({\bf B_n}).
$$
 When $k=1$, we recover  some of the results obtained by Davidson and Pitts \cite{DP2} and the author \cite{Po-automorphism}.
Let $H^\infty({\bf B_n})$ be the Hardy algebra of all  bounded free holomorphic functions on the regular polyball.
If $\Lambda: H^\infty({\bf B_n})\to H^\infty({\bf B_n})$
is a unital algebraic  homomorphism, it induces a unique homomorphism $\tilde
\Lambda:{\bf F^\infty_n}\to {\bf F^\infty_n}$ such that   $\Lambda \boldsymbol\cB=\boldsymbol\cB\tilde \Lambda$, where $\boldsymbol\cB$ is the noncommutative  Berezin transform.
We prove that
    $\tilde \Lambda$ is a unitarily implemented automorphism of ${\bf F}^\infty_{\bf n}$ if and only if
 there is $\varphi\in \text{\rm Aut}({\bf B_n})$ such that
    $$\Lambda(f)=f\circ \varphi,\qquad f\in H^\infty({\bf B_n}).$$
 A similar result holds for the algebra $A({\bf B_n})$ of all bounded free holomorphic functions on
 ${\bf B_n}(\cH)$ with continuous extension to ${\bf B_n}(\cH)^-$.

 In Section 6, we prove that the free holomorphic automorphism group \
 $\text{\rm Aut}({\bf B_n})$ is a   $\sigma$-compact, locally
compact topological group with respect to the topology induced by the metric
$$
d_{\bf B_n}(\phi,\psi):=\|\phi -\psi\|_\infty +
\|\phi^{-1}(0)-\psi^{-1}(0)\|,\qquad \phi,\psi\in  \text{\rm Aut}({\bf B_n}).
$$
We also show that if  ${\bf n}=(n_1,\ldots, n_k)\in \NN^k$, then
 the free holomorphic automorphism group \
 $\text{\rm Aut}({\bf B_n})$  has $\text{\rm card} (\Sigma)$  path connected components, where
 $$
 \Sigma:=\{\sigma\in \cS_k: \ (n_{\sigma(1)},\ldots, n_{\sigma(k)})=(n_1,\ldots, n_k)\}
 $$
 and $\cS_k$ is the symmetric group on the set $\{1,\ldots, k\}$.
We mention that a   map
$\pi: \text{\rm Aut}({\bf B_n})\to \cU(\cK)$, where $\cU(\cK)$ is the unitary group on the Hilbert space $\cK$,   is called (unitary)
projective representation if
  $\pi(id)=I$,
 $$
 \pi({\bf \Phi}) \pi({\bf \Psi})=c_{({\bf \Phi},{\bf \Psi})} \pi({{\bf \Phi}\circ {\bf \Psi}}),\qquad {\bf \Phi}, {\bf \Psi}\in \text{\rm Aut}({\bf B_n}),
 $$
   where $c_{({\bf \Phi},{\bf \Psi})}$
  is a complex number  with
 $|c({\bf \Phi},{\bf \Psi})|=1$, and
 the map $\text{\rm Aut}({\bf B_n})\ni {\bf \Phi}\mapsto \left<\pi({\bf\Phi})\xi,\eta\right> \in \CC$
is continuous for each $\xi,\eta\in \cK$.
Using the structure of the free holomorphic automorphisms  of the regular polyball ${\bf B_n}$, we conclude Section 6 by providing   a  concrete unitary projective representation of the topological group
  $\text{\rm Aut}({\bf B_n})$, with respect to the metric $d_{\bf B_n}$, in terms
      of noncommutative Berezin kernels associated with regular polyballs.

We mention that the techniques of the present paper  will be used in a future one to study the structure of the automorphism groups  associated with certain classes of noncommutative varieties in polyballs, including the case of commutative operatorial polyballs. We also expect some of our results to extend to more general noncommutative polydomains ( \cite{Po-Berezin3}, \cite{Po-Berezin-poly}).

\bigskip

\section{ Noncommutative   polyballs and free holomorphic functions }

In this section, we show that the regular polyball ${\bf B_n}$ is a logarithmically convex complete Reinhardt noncommutative domain. We study free holomorphic functions on regular polyballs and provide analogues of several classical results from complex analysis such as: Abel theorem, Hadamard formula, Cauchy inequality,  and  Liouville theorem for entire functions.

First, we introduce a class of  noncommutative Berezin  transforms associated  with regular polyballs.
 Let  ${\bf X}=({ X}_1,\ldots, { X}_k)\in {\bf B_n}(\cH)^-$ with $X_i:=(X_{i,1},\ldots, X_{i,n_i})$.
We  use the notation $X_{i,\alpha_i}:=X_{i,j_1}\cdots X_{i,j_p}$
  if  $\alpha_i=g_{j_1}^i\cdots g_{j_p}^i\in \FF_{n_i}^+$ and
   $X_{i,g_0^i}:=I$.
The {\it noncommutative Berezin kernel} associated with any element
   ${\bf X}$ in the noncommutative polyball ${\bf B_n}(\cH)^-$ is the operator
   $${\bf K_{X}}: \cH \to F^2(H_{n_1})\otimes \cdots \otimes  F^2(H_{n_k}) \otimes  \overline{{\bf \Delta_{X}}(I) (\cH)}$$
   defined by
   $$
   {\bf K_{X}}h:=\sum_{\beta_i\in \FF_{n_i}^+, i=1,\ldots,k}
   e^1_{\beta_1}\otimes \cdots \otimes  e^k_{\beta_k}\otimes {\bf \Delta_{X}}(I)^{1/2} X_{1,\beta_1}^*\cdots X_{k,\beta_k}^*h.
   $$
A very  important property of the Berezin kernel is that
     ${\bf K_{X}} { X}^*_{i,j}= ({\bf S}_{i,j}^*\otimes I)  {\bf K_{X}}$ for any  $i\in \{1,\ldots, k\}$ and $ j\in \{1,\ldots, n_i\}.
    $
    The {\it Berezin transform at} $X\in {\bf B_n}(\cH)$ is the map $ \boldsymbol{\cB_{\bf X}}: B(\otimes_{i=1}^k F^2(H_{n_i}))\to B(\cH)$
 defined by
\begin{equation*}
 {\boldsymbol\cB_{\bf X}}[g]:= {\bf K^*_{\bf X}} (g\otimes I_\cH) {\bf K_{\bf X}},
 \qquad g\in B(\otimes_{i=1}^k F^2(H_{n_i})).
 \end{equation*}
  If $g$ is in the $C^*$-algebra generated by ${\bf S}_{i,1},\ldots,{\bf S}_{i,n_i}$, we  define the Berezin transform at  $X\in {\bf B_n}(\cH)^-$, by
  $${\boldsymbol\cB_{\bf X}}[g]:=\lim_{r\to 1} {\bf K^*_{r\bf X}} (g\otimes I_\cH) {\bf K_{r\bf X}},
 \qquad g\in  C^*({\bf S}),
 $$
 where the limit is in the operator norm topology.
In this case, the Berezin transform at $X$ is a unital  completely positive linear  map such that
 $${\boldsymbol\cB_{\bf X}}({\bf S}_{(\alpha)} {\bf S}_{(\beta)}^*)={\bf X}_{(\alpha)} {\bf X}_{(\beta)}^*, \qquad (\alpha), (\beta) \in \FF_{n_1}^+\times \cdots \times\FF_{n_k}^+,
 $$
 where  ${\bf S}_{(\alpha)}:= {\bf S}_{1,\alpha_1}\cdots {\bf S}_{k,\alpha_k}$ if  $(\alpha):=(\alpha_1,\ldots, \alpha_k)\in \FF_{n_1}^+\times \cdots \times\FF_{n_k}^+$.

The  Berezin transform will play an important role in this paper.
   More properties  concerning  noncommutative Berezin transforms and multivariable operator theory on noncommutative balls and  polydomains, can be found in \cite{Po-poisson}, \cite{Po-holomorphic}, \cite{Po-automorphism},    \cite{Po-Berezin3}, and \cite{Po-Berezin-poly}.
For basic results on completely positive (resp. bounded)  maps  we refer the reader to \cite{Pa-book} and \cite{Pi-book}.

In what follows, we present some properties of the regular polyballs. Our first observation is that, in general, the inclusion  ${\bf B_n}(\cH)\subset {\bf P_n}(\cH)$ is strict. Indeed, consider the particular case   $n_1=\cdots=n_k=1$. Let $\cM$ be a Hilbert space, $\cH=\cM\oplus \cM$,
and $T_i:=\left(\begin{matrix} 0&0\\A_i&0\end{matrix}\right)$, $i\in \{1,\ldots, k\}$, where $A_i\in B(\cM)$ and $\|A_i\|<1$.  It is clear that $T_iT_s=T_sT_i$ for $i,s\in \{1,\ldots,k\}$, and
${\bf \Delta_T}=\left(\begin{matrix} I&0\\0&I-A_1A_1^*-\cdots-A_kA_k^*\end{matrix}\right)$.
Consequently, ${\bf T}=(T_1,\ldots, T_k)\in {\bf B}_{(1,\ldots, 1)}(\cH)$ if and only if $\|A_1A_1^*+\cdots+A_kA_k^*\|<1$. This clearly proves our assertion.
On the other hand,  note that there is $r\in (0,1)$ such that $r{\bf P_n}(\cH)\subset
{\bf B_n}(\cH)$. Moreover, due to Proposition 1.3 from \cite{Po-Berezin-poly},
one can easily see  that  $[B(\cH)^{n_1+\cdots +n_k}]_1\subset {\bf B_n}(\cH)$.

If ${\bf z}=({\bf z}_1,\ldots, {\bf z}_k)$,  where ${\bf z}_i=(z_{i,1},\ldots, z_{i,n_i})\in \CC^{n_i}$,
and ${\bf X}:=({\bf X}_1,\ldots, {\bf X}_k)$ is in the cartesian product   $B(\cH)^{n_1}\times\cdots \times B(\cH)^{n_k}$ with
${\bf X}_i=(X_{i,1},\ldots, X_{i,n_i})$, we denote
${\bf z X}:=({\bf z}_1{\bf X}_1,\ldots, {\bf z}_k {\bf X}_k)$, where ${\bf z}_i{\bf X}_i:=(z_{i,1}X_{i,1},\ldots, z_{i,k}X_{i,n_i})$. If ${\bf r }:=(r_1,\ldots, r_k)$, $r_i>0$, we set ${\bf rX}:=(r_1{\bf X}_1,\ldots, r_k{\bf X}_k)$.
 When $r\in \RR^+$, the notation $r{\bf X}$ is clear.

\begin{lemma} \label{ineqS}
If $\lambda_i\in \overline{\DD}$, $i\in \{1,\ldots, k\}$,    and ${\bf S}=({\bf S}_1,\ldots, {\bf S}_k)$ is the universal model for the regular polyball ${\bf B_n^-}$,  then
$$
(id-\Phi_{\lambda_1 {\bf S}_1})^{p_1}\circ\cdots \circ (id-\Phi_{\lambda_k {\bf S}_k})^{p_k}(I)\geq \prod_{i=1}^k (1-|\lambda_i|^2)^{p_i}I.
$$
 If   ${\bf z}=({\bf z}_1,\ldots, {\bf z}_k)$,  where ${\bf z}_i=(z_{i,1},\ldots, z_{i,n_i})\in \overline{\DD}^{n_i}$, then
$$
(id-\Phi_{{\bf z}_1 {\bf S}_1})^{p_1}\circ\cdots \circ (id-\Phi_{{\bf z}_k {\bf S}_k})^{p_k}(I)
\geq
(id-\Phi_{{\bf S}_1})^{p_1}\circ\cdots \circ (id-\Phi_{{\bf S}_k})^{p_k}(I),\qquad p_i\in \{0,1\}.
$$

\end{lemma}
\begin{proof} We recall that two operators $A,B\in B(\cH)$ are called doubly commuting if $AB=BA$ and $AB^*=B^*A$. Since the entries of ${\bf S}_i$ are doubly commuting with the entries of ${\bf S}_t$, whenever $i,t\in\{1,\ldots, k\}$, $i\neq t$, we have
$$(id-\Phi_{\lambda_1 {\bf S}_1})^{p_1}\circ\cdots \circ (id-\Phi_{\lambda_k {\bf S}_k})^{p_k}(I)=
\prod_{i=1}^k (I-\Phi_{\lambda_i{\bf S}_i}(I))^{p_i}.
 $$
 Taking into account that
 $I-\Phi_{\lambda_i{\bf S}_i}(I)\geq (1-|\lambda_i|^2)I$, the first inequality follows. Similarly, using the inequality $I-\Phi_{{\bf z}_i{\bf S}_i}(I)\geq I-\Phi_{{\bf S}_i}(I)$, one can deduce the second inequality.
\end{proof}

 \begin{definition}
 Let $G$ be a subset of $ B(\cH)^{n_1}\times\cdots \times B(\cH)^{n_k}$.
 \begin{enumerate}
 \item[(i)]
   $G$ is a complete Reinhardt set  if \, ${\bf zX}\in G$ for any ${\bf X}\in G$ and  ${\bf z}\in  \overline{\DD}^{n_1+\cdots+n_k}$.
 \item[(ii)]
  $G$ is a  logarithmically convex  set if
  $$\{(\log\|X_1\|,\ldots,\log\|X_k\|):\ (X_1,\ldots, X_k)\in G, X_i\neq 0\}
  $$
  is a convex subset of $\RR^k$.
  \end{enumerate}
 \end{definition}

\begin{proposition}\label{reg-poly} The  following properties hold:
\begin{enumerate}
\item [(i)]  The regular polyball  ${\bf B_n}(\cH)$ is relatively  open in   $ B(\cH)^{n_1}\times_c\cdots \times_c B(\cH)^{n_k}$, and its closure in the operator norm topology satisfies the relation
$$
{\bf B_n}(\cH)^-=\left\{ {\bf X}\in B(\cH)^{n_1}\times_c\cdots \times_c B(\cH)^{n_k}: \ {\bf \Delta_{X}^p}(I)\geq 0 \ \text{ for }\ {\bf p}=(p_1,\ldots, p_k)\ \text{ with } p_i\in \{0,1\}\right\},
$$
where $
{\bf \Delta_{X}^p}:=\left(id -\Phi_{X_1}\right)^{p_1}\circ \cdots \circ\left(id -\Phi_{ X_k}\right)^{p_k}
$
and  $(id-\Phi_{X_i})^0:=id$.
\item [(ii)] ${\bf B_n}(\cH)$ is a complete Reinhardt domain such that
$$
{\bf B_n}(\cH)=\bigcup_{{\bf z}\in  \overline{\DD}^{n_1+\cdots+n_k}}{\bf z}{\bf B_n}(\cH)=\bigcup_{{\bf z}\in  {\DD}^{n_1+\cdots+n_k}}{\bf z}{\bf B_n}(\cH)^-=\bigcup_{{\bf z}\in  {\DD}^{n_1+\cdots+n_k}}{\bf z}{\bf B_n}(\cH).
$$
and $$ {\bf B_n}(\cH)=\bigcup_{0\leq r<1}r{\bf B_n}(\cH)=\bigcup_{0\leq r<1}r{\bf B_n}(\cH)^-.
$$
\item [(iii)] ${\bf B_n}(\cH)^-$ is a complete Reinhardt  set and
$$
{\bf B_n}(\cH)^-=\bigcup_{{\bf z}\in  \overline{\DD}^{n_1+\cdots+n_k}}{\bf z}{\bf B_n}(\cH)^- = \bigcup_{0\leq r\leq 1}r{\bf B_n}(\cH)^-.
$$
\end{enumerate}
\end{proposition}
\begin{proof}
If ${\bf X}=(X_1,\ldots, X_k)\in {\bf B_n}(\cH)$, then there is $c>0$ such that ${\bf \Delta_{X}}(I)>cI$.  Given $d\in (0,c)$, there is $\epsilon>0$  such that
$
-dI\leq {\bf \Delta_{Y}}(I)-{\bf \Delta_{X}}(I)\leq dI
$
for any ${\bf Y}=(Y_1,\ldots, Y_k)\in B(\cH)^{n_1}\times_c\cdots \times_c B(\cH)^{n_k}$ with $\max_{i\in \{1,\ldots, k\}}\|X_i-Y_i\|<\epsilon$.
Consequently, we have
$$
{\bf \Delta_{Y}}(I)=({\bf \Delta_{Y}}(I)-{\bf \Delta_{X}}(I))+{\bf \Delta_{X}}(I)\geq (c-d)I>0,
$$
which proves that ${\bf B_n}(\cH)$ is relatively  open in   $ B(\cH)^{n_1}\times_c\cdots \times_c B(\cH)^{n_k}$ with respect to the product topology. To prove the second part of item (i), set
$$
\cD:=\left\{ {\bf X}\in B(\cH)^{n_1}\times_c\cdots \times_c B(\cH)^{n_k}: \ {\bf \Delta_{X}^p}(I)\geq 0 \ \text{ for }\ {\bf p}=(p_1,\ldots, p_k)\ \text{ with } p_i\in \{0,1\}\right\}.
$$
We shall prove that ${\bf B_n}(\cH)^-=\cD$.
Since ${\bf B_n}(\cH)$ is open, if  ${\bf X}\in {\bf B_n}(\cH)$, then there is $r\in [0,1)$ such that $\frac{1}{r}{\bf X}\in {\bf B_n}(\cH)$. Applying the Berezin transform at $\frac{1}{r}{\bf X}$ to the first inequality of Lemma \ref{ineqS}, when $\lambda_i=r$, we deduce that
\begin{equation*}
{\bf \Delta_{X}^p}(I)=(id-\Phi_{X_1})^{p_1}\circ\cdots \circ (id-\Phi_{ {X}_k})^{p_k}(I)\geq \prod_{i=1}^k (1-r^2)^{p_i}I.
\end{equation*}
 Hence, if ${\bf Y}\in {\bf B_n}(\cH)^-$,  a limiting process implies that
 ${\bf \Delta_{Y}^p}(I)\geq 0$ for any ${\bf p}=(p_1,\ldots, p_k)$ with  $ p_i\in \{0,1\}$.
 Therefore, ${\bf B_n}(\cH)^-\subseteq\cD$. To prove the reverse inequality, let
${\bf Y}=(Y_1,\ldots, Y_k)\in \cD$. In particular, we have $\|rY_i\|<1$ for any $r\in [0,1)$. Due to Lemma \ref{ineqS} and using the Berezin transform at $Y$, we have ${\bf \Delta}_{r{\bf Y}}(I)\geq (1-r^2)^kI$, which shows that $r{\bf Y}\in {\bf B_n}(\cH)$. Since $ r{\bf Y}\to {\bf Y}$, as $r\to 1$, we conclude that $\cD\subseteq  {\bf B_n}(\cH)^-$, which proves item (i).

If ${\bf z}\in  \overline{\DD}^{n_1+\cdots+n_k}$ and ${\bf T}\in {\bf B_n}(\cH)$, then applying the Berezin transform at ${\bf T}$ to the second inequality of Lemma \ref{ineqS} we obtain
$
{\bf \Delta_{zT}^p}(I)\geq {\bf \Delta_T^p}(I)>0
$
for any ${\bf p}=(p_1,\ldots, p_k)$ with $ p_i\in \{0,1\}$. Consequently, we have
$${\bf z}{\bf B_n}(\cH)\subseteq {\bf B_n}(\cH), \qquad {\bf z}\in  \overline{\DD}^{n_1+\cdots+n_k},
$$
which shows that ${\bf B_n}(\cH)$ is a complete Reinhardt domain and    ${\bf B_n}(\cH)=\bigcup_{{\bf z}\in  \overline{\DD}^{n_1+\cdots+n_k}}{\bf z}{\bf B_n}(\cH)$.

Let ${\bf T}\in {\bf B_n}(\cH)^-$ and ${\bf z}\in  {\DD}^{n_1+\cdots+n_k}$. Then there is $r\in(0,1)$ such that $\frac{1}{r}{\bf z}\in \DD^{n_1+\cdots +n_k}$.
Applying the Berezin transform at $r{\bf T}$ to the first inequality of Lemma \ref{ineqS} when $\lambda_1=\cdots =\lambda_k=r$,  we deduce that $r{\bf T}\in {\bf B_n}(\cH)$.
 Therefore, ${\bf zT}\in\frac{1}{r}{\bf z} {\bf B_n}(\cH)\in {\bf B_n}(\cH)$, which shows that
\begin{equation}
\label{subset1}
{\bf z}{\bf B_n}(\cH)^-\subseteq {\bf B_n}(\cH), \qquad {\bf z}\in  {\DD}^{n_1+\cdots+n_k}.
\end{equation}

Since ${\bf B_n}(\cH)$ is open, for any ${\bf X}\in {\bf B_n}(\cH)$, there is $r\in (0,1)$ such that ${\bf X}\in r{\bf B_n}(\cH)$. Consequently,
\begin{equation}
\label{subset2}
{\bf B_n}(\cH)\subset \bigcup_{0\leq r<1}r{\bf B_n}(\cH)\subset
\bigcup_{{\bf z}\in  {\DD}^{n_1+\cdots+n_k}}{\bf z}{\bf B_n}(\cH)\subseteq\bigcup_{{\bf z}\in  {\DD}^{n_1+\cdots+n_k}}{\bf z}{\bf B_n}(\cH)^-
\end{equation}
and
\begin{equation}
\label{subset3}
{\bf B_n}(\cH)\subset \bigcup_{0\leq r<1}r{\bf B_n}(\cH)\subset
\bigcup_{0\leq r<1}r{\bf B_n}(\cH)^-.
 \end{equation}
The relations \eqref{subset1} and \eqref{subset2} show that the first sequence of equalities in (ii) holds.
 Due to relation \eqref{subset1}, for each $r\in [0,1)$, we have
$ r{\bf B_n}(\cH)^-\subseteq {\bf B_n}(\cH)$ which together with  relation and \eqref{subset3} show that the second sequence of equalities in item (ii) holds. Now, one can easily see that item (iii) follows immediately from (ii).
The proof is complete.
\end{proof}

We remark that if  ${\bf r }:=(r_1,\ldots, r_k)$, $r_i>0$, then we also have
 $ {\bf B_n}(\cH)=\bigcup_{0\leq r_i<1}{\bf r}{\bf B_n}(\cH)^-.
$
Note also that the regular polyball ${\bf B_n}(\cH)$ is a logarithmically convex  complete Reinhardt domain.

For each $i\in\{1,\ldots, k\}$, let $Z_i:=(Z_{i,1},\ldots, Z_{i,n_i})$ be
an  $n_i$-tuple of noncommuting indeterminates and assume that, for any
$p,q\in \{1,\ldots, k\}$, $p\neq q$, the entries in $Z_p$ are commuting
 with the entries in $Z_q$. We set $Z_{i,\alpha_i}:=Z_{i,j_1}\cdots Z_{i,j_p}$
  if $\alpha_i\in \FF_{n_i}^+$ and $\alpha_i=g_{j_1}^i\cdots g_{j_p}^i$, and
   $Z_{i,g_0^i}:=1$, where $g_0^i$ is the identity in $\FF_{n^i}^+$.
   Given $A_{(\alpha_1,\ldots, \alpha_k)}\in B(\cK)$ with  $(\alpha_1,\ldots, \alpha_k)\in \FF_{n_1}^+\times \cdots \times\FF_{n_k}^+$,
we consider formal power series
$$
\varphi=\sum_{\alpha_1\in \FF_{n_1}^+,\ldots, \alpha_k\in \FF_{n_k}^+} A_{(\alpha_1,\ldots, \alpha_k)} \otimes Z_{1,\alpha_1}\cdots Z_{k,\alpha_k},\qquad A_{(\alpha_1,\ldots, \alpha_k)}\in B(\cK),
$$
in ideterminates $Z_{i,j}$. Denoting  $(\alpha):=(\alpha_1,\ldots, \alpha_k)\in \FF_{n_1}^+\times \cdots \times\FF_{n_k}^+$, $Z_{(\alpha)}:= Z_{1,\alpha_1}\cdots Z_{k,\alpha_k}$, and $A_{(\alpha)}:=A_{(\alpha_1,\ldots, \alpha_k)}$, we can also use
   the abbreviation
$\varphi=\sum\limits_{(\alpha)} A_{(\alpha)}\otimes Z_{(\alpha)}$.

The next result is an analogue of Abel theorem from complex analysis in  our noncommutative  multivariable setting.

\begin{theorem} \label{Abel} If  $\varphi= \sum\limits_{(p_1,\ldots, p_k)\in \ZZ_+^k}\sum\limits_{{ \alpha_i\in \FF_{n_i}^+, |\alpha_i|=p_i}\atop{i\in \{1,\ldots, k\} }}A_{(\alpha)}\otimes Z_{(\alpha)}$ is a formal power series and ${\bf r}=(r_1,\ldots, r_k)$, $r_i>0$, then the following statements hold.
\begin{enumerate}
\item[(i)] If the set
$$
A:=\{ \|r_1^{2p_1}\cdots r_k^{2p_k}\sum\limits_{{ \alpha_i\in \FF_{n_i}^+, |\alpha_i|=p_i}\atop{i\in \{1,\ldots, k\} }}A_{(\alpha)}^*A_{(\alpha)}\|:\ (p_1,\ldots, p_k)\in \ZZ_+^k\}
$$
is bounded, then the series
$$\sum\limits_{(p_1,\ldots, p_k)\in \ZZ_+^k}\|\sum\limits_{{ \alpha_i\in \FF_{n_i}^+, |\alpha_i|=p_i}\atop{i\in \{1,\ldots, k\} }}A_{(\alpha)}\otimes X_{(\alpha)}\|
$$
is convergent in ${\bf r}{\bf B_n}(\cH)$,  the regular polyball of  polyradius ${\bf r}=(r_1,\ldots, r_k)$,  and uniformly convergent on ${\bf s}{\bf B_n}(\cH)^-$ for any ${\bf s}=(s_1,\ldots, s_k)$ with $0\leq s_i<r_i$.
\item[(ii)] If the set $A$ is unbounded, then the series
$$\sum\limits_{(p_1,\ldots, p_k)\in \ZZ_+^k}\|\sum\limits_{{ \alpha_i\in \FF_{n_i}^+, |\alpha_i|=p_i}\atop{i\in \{1,\ldots, k\} }}A_{(\alpha)}\otimes X_{(\alpha)}\|\quad \text{ and } \quad
\sum\limits_{(p_1,\ldots, p_k)\in \ZZ_+^k}\sum\limits_{{ \alpha_i\in \FF_{n_i}^+, |\alpha_i|=p_i}\atop{i\in \{1,\ldots, k\} }}A_{(\alpha)}\otimes X_{(\alpha)}
$$
are divergent for some ${\bf X}\in {\bf r}{\bf B_n}(\cH)^-$ and some Hilbert space  $\cH$.
\end{enumerate}
\end{theorem}

\begin{proof} Let $s_i<r_i$, $i\in \{1,\ldots, k\}$, and ${\bf X}\in {\bf r} {\bf B_n}(\cH)$, and assume that  there is $C>0$ such that
$$\|r_1^{2p_1}\cdots r_k^{2p_k}\sum\limits_{{ \alpha_i\in \FF_{n_i}^+, |\alpha_i|=p_i}\atop{i\in \{1,\ldots, k\} }}A_{(\alpha)}^*A_{(\alpha)}\|\leq C,\qquad (p_1,\ldots, p_k)\in \ZZ_+^k.
$$
Due to the von Neumann type  inequality \cite{Po-Berezin-poly}, we have
\begin{equation*}
\begin{split}
\|\sum\limits_{{ \alpha_i\in \FF_{n_i}^+, |\alpha_i|=p_i}\atop{i\in \{1,\ldots, k\} }}A_{(\alpha)}\otimes X_{(\alpha)}\|
&\leq
\|\sum\limits_{{ \alpha_i\in \FF_{n_i}^+, |\alpha_i|=p_i}\atop{i\in \{1,\ldots, k\} }}s_1^{p_1}\cdots s_k^{p_k}A_{(\alpha)}\otimes {\bf S}_{(\alpha)}\|\\
&=s_1^{p_1}\cdots s_k^{p_k}\|\sum\limits_{{ \alpha_i\in \FF_{n_i}^+, |\alpha_i|=p_i}\atop{i\in \{1,\ldots, k\} }}A_{(\alpha)}^*A_{(\alpha)} \|^{1/2}\\
&< \left(\frac{s_1}{r_1}\right)^{p_1}\cdots \left(\frac{s_k}{r_k}\right)^{p_k}C^{1/2}
\end{split}
\end{equation*}
for any ${\bf X}\in s{\bf B_n}(\cH)^-$. On the other hand, due to Proposition \ref{reg-poly}, we have
${\bf r}{\bf B_n}(\cH)= \bigcup_{0\leq s_i<r_i}{\bf s}{\bf B_n}(\cH)^-$.
Now, one can easily complete the proof of part (i).

To prove (ii), assume that the set $A$ is unbounded.
Then, using the fact that the isometries ${\bf S}_{(\alpha)}$, with
 $(\alpha)=(\alpha_1,\ldots, \alpha_k)\in \FF_{n_1}^+\times\cdots\times \FF_{n_k}^+$, $|\alpha_i|=p_i$, have orthogonal ranges, one can easily deduce that  the series
$$\sum\limits_{(p_1,\ldots, p_k)\in \ZZ_+^k}\|\sum\limits_{{ \alpha_i\in \FF_{n_i}^+, |\alpha_i|=p_i}\atop{i\in \{1,\ldots, k\} }}A_{(\alpha)}\otimes r_1^{p_1}\cdots r_k^{p_k}{\bf S}_{(\alpha)}\|
$$
and $$\sum\limits_{(p_1,\ldots, p_k)\in \ZZ_+^k}\sum\limits_{{ \alpha_i\in \FF_{n_i}^+, |\alpha_i|=p_i}\atop{i\in \{1,\ldots, k\} }}A_{(\alpha)}\otimes r_1^{p_1}\cdots r_k^{p_k}{\bf S}_{(\alpha)}
$$
are divergent,  and ${\bf rS}:=(r_1{\bf S}_1,\ldots, r_k {\bf S}_k)\in {\bf r}{\bf B_n}(\otimes_{i=1}^k F^2(H_{n_i}))^-$.
\end{proof}

 \begin{definition} A power series $\varphi=\sum\limits_{(\alpha)} A_{(\alpha)}\otimes Z_{(\alpha)}$ is called   {\it free holomorphic function} (with coefficients in $B(\cK)$) on the
{\it abstract  polyball}
$\boldsymbol\rho{\bf B_n}:=
\{\boldsymbol\rho{\bf  B_n}(\cH):\ \cH \text{ is a Hilbert space}\}$, $\boldsymbol\rho=(\rho_1,\ldots,\rho_k)$,  $\rho_i>0$, if the series
$$
\varphi({\bf X} ):=\sum\limits_{(p_1,\ldots, p_k)\in \ZZ_+^k}\sum\limits_{{ \alpha_i\in \FF_{n_i}^+, |\alpha_i|=p_i}\atop{i\in \{1,\ldots, k\} }} A_{(\alpha)}\otimes  X_{(\alpha)}
$$
is convergent in the operator norm topology for any ${\bf X}=\{X_{i,j}\}\in \boldsymbol\rho{\bf B_n}(\cH)$ with $i\in \{1,\ldots, k\}$ and  $j\in \{1,\ldots, n_i\}$,  and any Hilbert space $\cH$. We denote by $Hol({\bf \boldsymbol\rho B_n})$ the set of all free holomorphic functions on  ${\bf \boldsymbol\rho B_n}$ with scalar coefficients.
\end{definition}

Using Theorem \ref{Abel}, one can easily deduce the following characterization for free holomorphic functions on regular polyballs.

\begin{corollary} Let   ${\bf S} $
    be the universal model associated with the abstract regular polyball
  ${\bf B_n}$. A formal power series $\varphi=\sum\limits_{(\alpha)} A_{(\alpha)}\otimes Z_{(\alpha)}$ is a  free holomorphic function (with coefficients in $B(\cK)$) on the
abstract  polyball
$\boldsymbol\rho{\bf B_n}$, where
  $\boldsymbol\rho=(\rho_1,\ldots,\rho_k)$,  $\rho_i>0$, if and only if  the series
$$\sum\limits_{(p_1,\ldots, p_k)\in \ZZ_+^k}\|\sum\limits_{{ \alpha_i\in \FF_{n_i}^+, |\alpha_i|=p_i}\atop{i\in \{1,\ldots, k\} }}A_{(\alpha)}\otimes r_1^{p_1}\cdots r_k^{p_k}{\bf S}_{(\alpha)}\|
$$
converges for any $r_i\in [0, \rho_i)$, $i\in \{1,\ldots, k\}$.
\end{corollary}
Throughout the paper, we say that the abstract polyball ${\bf B_n}$ or a free holomorphic function $F$ on ${\bf B_n}$ has a certain property, if the property holds for any Hilbert space representation of ${\bf B_n}$ and $F$, respectively.
We remark that the coefficients of a free holomorphic function on a polyball are uniquely determined by its representation on an infinite dimensional Hilbert space. Indeed, assume that $F=\sum\limits_{(\alpha)} A_{(\alpha)}\otimes Z_{(\alpha)}$, $A_{(\alpha)}\in \cK$,  is a  free holomorphic function with $F(r{\bf S})=0$ for any $r\in [0,1)$. Then, for any $x,y\in \cK$, we have
$$
\left<F(r{\bf S})(x\otimes 1),(y\otimes {\bf S}_{(\alpha)}1\right>=r^{|\alpha_1|+\cdots +|\alpha_k|}\left< A_{(\alpha)}x,y\right>=0
$$
for any $(\alpha)=(\alpha_1,\ldots, \alpha_k)\in \FF_{n_1}^+\times \cdots \times \FF_{n_k}^+$. Hence $A_{(\alpha)}=0$, which proves our assertion.

\begin{corollary} \label{scalar} If $\varphi=\sum\limits_{(\alpha)} a_{(\alpha)}\otimes Z_{(\alpha)}$, $a_{(\alpha)}\in \CC$  is a  free holomorphic function  on the
abstract  polyball
$\boldsymbol\rho{\bf B_n}$,  $\rho=(\rho_1,\ldots, \rho_k)$, then its representation on $\CC$, i.e.
$$
\varphi(\lambda_1,\ldots, \lambda_k)=\sum\limits_{(\alpha)} a_{(\alpha)}\otimes \lambda_{(\alpha)},\quad \lambda_i=(\lambda_{i,1},\ldots, \lambda_{i, n_i}),
$$
is a holomorphic function on the scalar polyball $\boldsymbol\rho{\bf P_n}(\CC)= (\CC^{n_1})_{\rho_1}\times \cdots \times (\CC^{n_k})_{\rho_k}$
\end{corollary}

In what follows, we obtain Cauchy type inequalities for the coefficients of  free holomorphic functions on regular polyballs.
\begin{theorem}\label{Cauchy-ineq} Let $F:\boldsymbol\rho{\bf B_n}(\cH)\to B(\cK)\otimes_{min} B(\cH)$ be a free holomorphic function with representation
$$F({\bf X})=\sum\limits_{(p_1,\ldots, p_k)\in \ZZ_+^k}\sum\limits_{{ \alpha_i\in \FF_{n_i}^+, |\alpha_i|=p_i}\atop{i\in \{1,\ldots, k\} }}A_{(\alpha)}\otimes X_{(\alpha)}.
$$
Let ${\bf r}=(r_1,\ldots, r_k)$ be such that $0<r_i<\rho_i$ and define  $M({\bf r}):=\sup_{{\bf X}\in {\bf r}{\bf B_n}(\cH)^-}\|F({\bf X})\|$.
Then, for each $(p_1,\ldots, p_k)\in \ZZ_+^k$, we have
$$
\|\sum\limits_{{ \alpha_i\in \FF_{n_i}^+, |\alpha_i|=p_i}\atop{i\in \{1,\ldots, k\} }}A_{(\alpha)}^*A_{(\alpha)}\|^{1/2}\leq \frac{1}{r_1^{p_1}\cdots r_k^{p_k}}M({\bf r}).
$$
Moreover, $M({\bf r})=\|F({\bf rS})\|$, where ${\bf S}$ is the universal model of the regular polyball ${\bf B_n}$.
\end{theorem}
\begin{proof}
Using the fact that the isometries ${\bf S}_{(\alpha)}$, with
 $(\alpha)=(\alpha_1,\ldots, \alpha_k)\in \FF_{n_1}^+\times\cdots\times \FF_{n_k}^+$, $|\alpha_i|=p_i$, have orthogonal ranges, we deduce that
\begin{equation*}
\begin{split}
\left|\left< \left(\sum\limits_{{ \alpha_i\in \FF_{n_i}^+, |\alpha_i|=p_i}\atop{i\in \{1,\ldots, k\} }}A_{(\alpha)}^*\otimes  {\bf S}_{(\alpha)}^*\right)F({\bf r S})(h\otimes 1), h\otimes 1\right>\right|
&\leq \left\|\sum\limits_{{ \alpha_i\in \FF_{n_i}^+, |\alpha_i|=p_i}\atop{i\in \{1,\ldots, k\} }}A_{(\alpha)}^*\otimes  {\bf S}_{(\alpha)}^*\right\| M({\bf r}) \|h\|^2\\
&=
\|\sum\limits_{{ \alpha_i\in \FF_{n_i}^+, |\alpha_i|=p_i}\atop{i\in \{1,\ldots, k\} }}A_{(\alpha)}^*A_{(\alpha)}\|^{1/2}M({\bf r}) \|h\|^2
\end{split}
\end{equation*}
for any $h\in \cK$. On the other hand, we have
\begin{equation*}
\begin{split}
&\left< \left(\sum\limits_{{ \alpha_i\in \FF_{n_i}^+, |\alpha_i|=p_i}\atop{i\in \{1,\ldots, k\} }}A_{(\alpha)}^*\otimes  {\bf S}_{(\alpha)}^*\right)F({\bf r S})(h\otimes 1), h\otimes 1\right>\\
&=r_1^{p_1}\cdots r_k^{p_k}
\left< \left(\sum\limits_{{ \alpha_i\in \FF_{n_i}^+, |\alpha_i|=p_i}\atop{i\in \{1,\ldots, k\} }}A_{(\alpha)}^* { A}_{(\alpha)}\otimes I\right) (h\otimes 1), h\otimes 1\right>\\
&=r_1^{p_1}\cdots r_k^{p_k}\|\left(\sum\limits_{{ \alpha_i\in \FF_{n_i}^+, |\alpha_i|=p_i}\atop{i\in \{1,\ldots, k\} }}A_{(\alpha)}^*A_{(\alpha)}\right)^{1/2}h\|^2.
\end{split}
\end{equation*}
Hence, using the previous inequality, we deduce that

$$
r_1^{p_1}\cdots r_k^{p_k}\|\left(\sum\limits_{{ \alpha_i\in \FF_{n_i}^+, |\alpha_i|=p_i}\atop{i\in \{1,\ldots, k\} }}A_{(\alpha)}^*A_{(\alpha)}\right)^{1/2}h\|^2 \leq
\|\sum\limits_{{ \alpha_i\in \FF_{n_i}^+, |\alpha_i|=p_i}\atop{i\in \{1,\ldots, k\} }}A_{(\alpha)}^*A_{(\alpha)}\|^{1/2}M({\bf r}) \|h\|^2
$$
for any $h\in \cK$, and the inequality in the theorem follows.
The fact that $M({\bf r})=\|F({\bf rS})\|$ is due to von Neumann inequality \cite{Po-poisson}. The proof is complete.
\end{proof}

We remark that due to the fact that there is $r\in (0,1)$ such that $r{\bf P_n}(\cH)\subset
{\bf B_n}(\cH)$, we have
$$
 B(\cH)^{n_1}\times_c\cdots \times_c B(\cH)^{n_k}=\bigcup_{\rho>0} \rho{\bf B_n}(\cH).
$$
We say that $F$ is an {\it entire function} in $B(\cH)^{n_1}\times_c\cdots \times_c B(\cH)^{n_k}$ if $F$ is free holomorphic on every regular polyball $\rho{\bf B_n}(\cH)$, $\rho>0$.

Here is an analogue of Liouville's theorem for entire functions  on $B(\cH)^{n_1}\times_c\cdots \times_c B(\cH)^{n_k}$.

\begin{corollary} If $F:B(\cH)^{n_1}\times_c\cdots \times_c B(\cH)^{n_k}\to B(\cK)\otimes_{min} B(\cH)$ is an entire function  with the property that there is a constant $C>0$  and $(p_1,\ldots, p_k)\in \ZZ_+^k$ such that
$$
\|F({\bf X})\|\leq C \|\sum\limits_{{ \alpha_i\in \FF_{n_i}^+, |\alpha_i|=q_i}\atop{i\in \{1,\ldots, k\} }}{\bf X}_{(\alpha)}{\bf X}_{(\alpha)}^*\|^{1/2}
$$
for any ${\bf X}\in B(\cH)^{n_1}\times_c\cdots \times_c B(\cH)^{n_k}$, then $F$ is a polynomial of degree at most $q_1+\cdots+q_k$. In particular, a bounded free holomorphic function must be constant.
\end{corollary}
\begin{proof} Let $F$ have the representation
$$F({\bf X})=\sum\limits_{(p_1,\ldots, p_k)\in \ZZ_+^k}\sum\limits_{{ \alpha_i\in \FF_{n_i}^+, |\alpha_i|=p_i}\atop{i\in \{1,\ldots, k\} }}A_{(\alpha)}\otimes X_{(\alpha)}.
$$
Due to the hypothesis, we have
$$\|F({\bf rS})\|\leq C r_1^{q_1}\cdots r_k^{q_k}\|\sum\limits_{{ \alpha_i\in \FF_{n_i}^+, |\alpha_i|=q_i}\atop{i\in \{1,\ldots, k\} }}{\bf S}_{(\alpha)}{\bf S}_{(\alpha)}^*\|^{1/2}\leq C r_1^{q_1}\cdots r_k^{q_k}
$$
for any $r_i>0$.
Hence, and using  Theorem \ref{Cauchy-ineq}, we deduce that
\begin{equation*}
\begin{split}
\|\sum\limits_{{ \alpha_i\in \FF_{n_i}^+, |\alpha_i|=p_i}\atop{i\in \{1,\ldots, k\} }}A_{(\alpha)}^*A_{(\alpha)}\|^{1/2}
&\leq \frac{1}{r_1^{p_1}\cdots r_k^{p_k}}M({\bf r})
\leq  \frac{1}{r_1^{p_1}\cdots r_k^{p_k}} \|F({\bf rS})\|\\
&\leq  C\frac{1}{r_1^{p_1-q_1}\cdots r_k^{p_k-q_k}}
\end{split}
\end{equation*}
for any $r_i>0$ and $i\in\{1,\ldots, k\}$.
Consequently, if there is $s\in\{1,\ldots, k\}$ such that  $p_s>q_s$, then taking $r_s\to \infty$ we obtain $$\sum\limits_{{ \alpha_i\in \FF_{n_i}^+, |\alpha_i|=p_i}\atop{i\in \{1,\ldots, k\} }}A_{(\alpha)}^*A_{(\alpha)}=0,
$$
which implies $A_{(\alpha)}=0$ for any $(\alpha)=(\alpha_1,\ldots, \alpha_k)$ with $\alpha_i\in \FF_{n_i}^+$ and  $|\alpha_i|=p_i$ and any $p_i\in \ZZ^+$, $i\neq s$. Hence, we have
$$F({\bf X})=\sum\limits_{{(p_1,\ldots, p_k)\in \ZZ_+^k}\atop{p_i\leq q_i}}\sum\limits_{{ \alpha_i\in \FF_{n_i}^+, |\alpha_i|=p_i}\atop{i\in \{1,\ldots, k\} }}A_{(\alpha)}\otimes X_{(\alpha)}.
$$
The proof is complete.
\end{proof}

Define the set
$$
\Lambda:=\{{\bf r}=(r_1,\ldots, r_k)\in \RR_+^k: \ \{\|\sum\limits_{{ \alpha_i\in \FF_{n_i}^+, |\alpha_i|=p_i}\atop{i\in \{1,\ldots, k\} }}r_1^{2p_1}\cdots r_k^{2p_k}A_{(\alpha)}^*A_{(\alpha)}  \|\}_{(p_1,\ldots, p_k)\in \ZZ_+^k} \ \text{ is bounded}\}.
$$
  Given a  formal power series $\varphi=\sum\limits_{(\alpha)} A_{(\alpha)}\otimes Z_{(\alpha)}$, we define the  set
 $$
 {\bf D}_\varphi(\cH):=\bigcup_{{\bf r}\in \Lambda} {\bf r}{\bf B_n}(\cH).
 $$
We say that ${\bf D}_\varphi$ is logarithmically convex if  $\Lambda$ is log-convex, i.e.
the set $$\{(\log r_1,\ldots, \log r_k): \ (r_1,\ldots, r_k)\in \Lambda, r_i>0\}$$ is  convex.

\begin{proposition} Let $\varphi=\sum\limits_{(\alpha)} A_{(\alpha)}\otimes Z_{(\alpha)}$  be a formal power series. The following statements hold.
\begin{enumerate}
\item[(i)] $\varphi$ is free holomorphic on  ${\bf D}_\varphi$
and $$\varphi({\bf X})=\sum\limits_{(p_1,\ldots, p_k)\in \ZZ_+^k}\sum\limits_{{ \alpha_i\in \FF_{n_i}^+, |\alpha_i|=p_i}\atop{i\in \{1,\ldots, k\} }}A_{(\alpha)}\otimes X_{(\alpha)}, \qquad {\bf X}\in {\bf D}_\varphi,
$$
where the series is  convergent in the operator norm.
\item[(ii)] ${\bf D}_\varphi$  is  a logarithmically convex complete Reinhardt domain.
\end{enumerate}
\end{proposition}
\begin{proof}
According to Theorem \ref{Abel} and due to the uniqueness of the representation for free holomorphic functions on polyballs, $\varphi$ is a free holomorphic function on
$
 {\bf D}_\varphi(\cH):=\bigcup_{{\bf r}\in \Lambda} {\bf r}{\bf B_n}(\cH)
 $
and has the representation of item (i). To prove (ii), note first that, due to Proposition \ref{reg-poly}, ${\bf D}_\varphi$  is  a  complete Reinhardt domain. Now, let $(r_1,\ldots, r_k)$ and
$(s_1,\ldots, s_k)$ be in $\Lambda$. Then there is a constant $C>0$ such that
$$
\|r_1^{2p_1}\cdots r_k^{2p_k}Q_{\bf p} \|
\leq C\quad \text{ and }\quad
\|s_1^{2p_1}\cdots s_k^{2p_k}Q_{\bf p}  \|\leq C
$$
for any ${\bf p}=(p_1,\ldots, p_k)\in \ZZ_+^k$, where
 $Q_{\bf p}:=\sum\limits_{{ \alpha_i\in \FF_{n_i}^+, |\alpha_i|=p_i}\atop{i\in \{1,\ldots, k\} }}A_{(\alpha)}^*A_{(\alpha)}$.  Consequently, due to the spectral theorem for positive operators, we have
\begin{equation*}
\begin{split}
\left\|( r_1^t s_1^{1-t})^{2p_1}\cdots ( r_k^t s_k^{1-t})^{2p_k}Q_{\bf p}\right\|
&=
\left\|(r_1^{2p_1}\cdots r_k^{2p_k} Q_{\bf p})^t (s_1^{2p_1}\cdots s_k^{2p_k} Q_{\bf p})^{1-t}\right\|\\
&\leq
\left\|(r_1^{2p_1}\cdots r_k^{2p_k} Q_{\bf p})^t\right\| \left\|(s_1^{2p_1}\cdots s_k^{2p_k} Q_{\bf p})^{1-t}\right\|\\
&\leq
\left\|r_1^{2p_1}\cdots r_k^{2p_k} Q_{\bf p}\right\|^t \left\|s_1^{2p_1}\cdots s_k^{2p_k} Q_{\bf p}\right\|^{1-t}\\
&\leq C^t C^{1-t}=C.
\end{split}
\end{equation*}
Consequently,  $(r_1^t s_1^{1-t},\ldots,  r_k^t k_1^{1-t})\in \Lambda$, which proves that
 ${\bf D}_\varphi$  is   logarithmically convex.
The proof is complete.
\end{proof}

We remark that, due to Theorem \ref{Abel},  if $\boldsymbol\rho=(\rho_1,\ldots, \rho_k)\notin \Lambda$, then
$\sum\limits_{(p_1,\ldots, p_k)\in \ZZ_+^k}\sum\limits_{{ \alpha_i\in \FF_{n_i}^+, |\alpha_i|=p_i}\atop{i\in \{1,\ldots, k\} }}A_{(\alpha)}\otimes X_{(\alpha)}
$
is  divergent for some ${\bf X}\in {\boldsymbol\rho}{\bf B_n}(\cH)^-$ and some Hilbert space  $\cH$.
Indeed, take ${\bf X}=\boldsymbol \rho {\bf S}$ and use Theorem \ref{Abel}.
We call the set ${\bf D}_\varphi$  the {\it universal domain of convergence} of the power series $\varphi$.

 Our next task is to find the largest polyball $r{\bf B_n}(\cH)$, $r>0$, which is included in the universal domain of convergence of $\varphi$.

 \begin{theorem} \label{Hadamard}
 Let $\varphi=\sum\limits_{(\alpha)} A_{(\alpha)}\otimes Z_{(\alpha)}$  be a formal power series
 and define
 $\gamma\in[0,\infty]$ by setting
 $$
 \frac{1}{\gamma}:=\limsup_{(p_1,\ldots, p_k)\in \ZZ_+^k} \|\sum\limits_{{ \alpha_i\in \FF_{n_i}^+, |\alpha_i|=p_i}\atop{i\in \{1,\ldots, k\} }}A_{(\alpha)}^*A_{(\alpha)}\|^{\frac{1}{2(p_1+\cdots +p_k)}}.
 $$
 Then the following statements hold.
 \begin{enumerate}
 \item[(i)]  The series
 $$
 \sum\limits_{(p_1,\ldots, p_k)\in \ZZ_+^k}\|\sum\limits_{{ \alpha_i\in \FF_{n_i}^+, |\alpha_i|=p_i}\atop{i\in \{1,\ldots, k\} }}A_{(\alpha)}\otimes X_{(\alpha)}\|, \qquad {\bf X}\in \gamma{\bf B_n}(\cH),
 $$
 is convergent. Moreover, the convergence is uniform on $r{\bf B_n}(\cH)^-$ if $0\leq r<\gamma$.
 \item[(ii)]
 For any $s>\gamma$, there is a Hilbert space $\cH$ and ${\bf Y}\in s{\bf B_n}(\cH)^-$  such that the series
 $$
 \sum\limits_{(p_1,\ldots, p_k)\in \ZZ_+^k}\sum\limits_{{ \alpha_i\in \FF_{n_i}^+, |\alpha_i|=p_i}\atop{i\in \{1,\ldots, k\} }}A_{(\alpha)}\otimes Y_{(\alpha)}
 $$
 is divergent in the operator norm topology.
 \end{enumerate}
 \end{theorem}
 \begin{proof} Assume that $\gamma>0$ and let ${\bf X}\in r{\bf B_n}(\cH)^-$, where $0\leq r<\gamma$. Fix $\rho\in (r,\gamma)$ and note that
 $$\|\sum\limits_{{ \alpha_i\in \FF_{n_i}^+, |\alpha_i|=p_i}\atop{i\in \{1,\ldots, k\} }}A_{(\alpha)}^*A_{(\alpha)}\|^{\frac{1}{2(p_1+\cdots +p_k)}}<\frac{1}{\rho}
 $$
for all but finitely many $(p_1,\ldots, p_k)\in \ZZ_+^k$. Consequently, due to the von Neumann type inequality \cite{Po-poisson}, we have
\begin{equation*}
\begin{split}
\|\sum\limits_{{ \alpha_i\in \FF_{n_i}^+, |\alpha_i|=p_i}\atop{i\in \{1,\ldots, k\} }}A_{(\alpha)}\otimes X_{(\alpha)}\|
&\leq
\|\sum\limits_{{ \alpha_i\in \FF_{n_i}^+, |\alpha_i|=p_i}\atop{i\in \{1,\ldots, k\} }}A_{(\alpha)}\otimes r^{p_1+\cdots + p_k} {\bf S}_{(\alpha)}\|\\
&=r^{p_1+\cdots p_k}\|\sum\limits_{{ \alpha_i\in \FF_{n_i}^+, |\alpha_i|=p_i}\atop{i\in \{1,\ldots, k\} }}A_{(\alpha)}^*A_{(\alpha)}\|^{1/2}<\left(\frac{r}{\rho}\right)^{p_1+\cdots +p_k}
\end{split}
\end{equation*}
for all but finitely many $(p_1,\ldots, p_k)\in \ZZ_+^k$. Hence, item (i) holds and also implies that   the series
 $
 \sum\limits_{q=0}^\infty\|\sum\limits_{  {|\alpha_1|+\cdots +|\alpha_k|=q}\atop {\alpha_i\in \FF_{n_i}^+}}A_{(\alpha)}\otimes X_{(\alpha)}\|
 $
  is uniformly convergent  on $r{\bf B_n}(\cH)^-$. The case when $\gamma=\infty$ can be treated in a similar manner.
Now, assume that $\gamma<\rho<s$ and let ${\bf Y}:=s{\bf S}$, where ${\bf S}$ is the universal model of ${\bf B_n}^-$. It is clear that ${\bf Y}\in s{\bf B_n}(\cH)^-$ and
$$
\|\sum\limits_{{ \alpha_i\in \FF_{n_i}^+, |\alpha_i|=p_i}\atop{i\in \{1,\ldots, k\} }}A_{(\alpha)}\otimes Y_{(\alpha)}\|=
s^{p_1+\cdots +p_k}\|\sum\limits_{{ \alpha_i\in \FF_{n_i}^+, |\alpha_i|=p_i}\atop{i\in \{1,\ldots, k\} }}A_{(\alpha)}^*A_{(\alpha)}\|^{1/2}.
$$
Since $\frac{1}{\rho}<\frac{1}{\gamma}$, there are infinitely many tuples
$(p_1,\ldots, p_k)\in \ZZ_+^k$ such that
$$\|\sum\limits_{{ \alpha_i\in \FF_{n_i}^+, |\alpha_i|=p_i}\atop{i\in \{1,\ldots, k\} }}A_{(\alpha)}^*A_{(\alpha)}\|^{\frac{1}{2(p_1+\cdots +p_k)}}>\frac{1}{\rho}
 $$
and, consequently,
$\|\sum\limits_{{ \alpha_i\in \FF_{n_i}^+, |\alpha_i|=p_i}\atop{i\in \{1,\ldots, k\} }}A_{(\alpha)}\otimes Y_{(\alpha)}\|>\left(\frac{s}{\rho}\right)^{p_1+\cdots+p_k}$. This shows that item (ii) holds and, moreover,
that the series
 $$
 \sum\limits_{(p_1,\ldots, p_k)\in \ZZ_+^k}\left\|\sum\limits_{{ \alpha_i\in \FF_{n_i}^+, |\alpha_i|=p_i}\atop{i\in \{1,\ldots, k\} }}A_{(\alpha)}\otimes Y_{(\alpha)}\right\|
 $$
 is divergent.
  \end{proof}

 The number $\gamma$ satisfying properties (i) and (ii) in the theorem above is unique and is called the {\it polyball radius of convergence} for the power series $\varphi$.

 \begin{corollary} \label{Hadamard2} Under the conditions of Theorem \ref{Hadamard}, the following statements hold.
  \begin{enumerate}
 \item[(i)]  The series
 $$
 \sum\limits_{q=0}^\infty\|\sum\limits_{  {|\alpha_1|+\cdots +|\alpha_k|=q}\atop {\alpha_i\in \FF_{n_i}^+}}A_{(\alpha)}\otimes X_{(\alpha)}\|
 $$
  is uniformly convergent  on $r{\bf B_n}(\cH)^-$ if $0\leq r<\gamma$.
 \item[(ii)]
 For any $s>\gamma$, there is ${\bf Y}\in s{\bf B_n}(\cH)^-$ such that the series
 $$
 \sum\limits_{q=0}^\infty\sum\limits_{  {|\alpha_1|+\cdots +|\alpha_k|=q}\atop {\alpha_i\in \FF_{n_i}^+}}A_{(\alpha)}\otimes Y_{(\alpha)}
 $$
 is divergent in the operator norm topology.
 \end{enumerate}
 \end{corollary}

 \begin{proof} A closer look at the proof of Theorem \ref{Hadamard} reveals that  item (i) was already proved  and the only thing  that we need in order to complete the proof of  item (ii) is that, under the condition $\gamma<\rho<s$,
$$
 \sum\limits_{q=0}^\infty\sum\limits_{  {|\alpha_1|+\cdots |\alpha_k|=q}\atop {\alpha_i\in \FF_{n_i}^+}}A_{(\alpha)}\otimes s^q{\bf S}_{(\alpha)}
 $$
 is divergent in the operator norm topology. Assume the contrary and apply the convergent series above to the vector $x\otimes 1$, where  $x\in \cK$. We deduce that
$ \sum\limits_{q=0}^\infty\sum\limits_{  {|\alpha_1|+\cdots |\alpha_k|=q}\atop {\alpha_i\in \FF_{n_i}^+}}A_{(\alpha)}\otimes s^qe_{(\alpha)}
$
is in the Hilbert space $\cK\otimes\bigotimes_{i=1}^k F^2(H_{n_i})$. Since $\{e_{(\alpha)}\}_{(\alpha)\in \FF_{n_1}^+\times\cdots \times  \FF_{n_k}^+}$ is an orthonormal basis for $\bigotimes_{i=1}^k F^2(H_{n_i})$, we conclude that  the series
$\sum_{(\alpha) } A_{(\alpha)}^*A_{(\alpha)}$ is WOT-convergent.
Let $r\in [0,1)$ and note that

 \begin{equation*}
 \begin{split}
 &\sum_{p=0}^\infty r^p\sum_{{(p_1,\ldots, p_k)\in\ZZ_+^k}\atop{p_1+\cdots +p_k=p}}
 \|\sum\limits_{{ \beta_i\in \FF_{n_i}^+, |\beta_i|=p_i}\atop{i\in \{1,\ldots, k\} }}A_{(\beta)}\otimes s^{p_1+\cdots +p_k} {\bf S}_{(\beta)}\| \\
 &\leq
 \sum_{p=0}^\infty r^p\sum_{{(p_1,\ldots, p_k)\in\ZZ_+^k}\atop{p_1+\cdots +p_k=p}}
 \|\sum\limits_{{ \beta_i\in \FF_{n_i}^+, |\beta_i|=p_i}\atop{i\in \{1,\ldots, k\} }}s^{2(|\beta_1|+\cdots +|\beta_k|)}A_{(\beta)}^* A_{(\beta)} \|^{1/2}\\
 &\leq
 \sum_{p=0}^\infty r^p \left(\begin{matrix} p+k-1\\k-1\end{matrix}\right)
 \|\sum\limits_{{ \beta_i\in \FF_{n_i}^+ }\atop{i\in \{1,\ldots, k\} }}s^{2(|\beta_1|+\cdots +|\beta_k|)}A_{(\beta)}^* A_{(\beta)} \|^{1/2}.
\end{split}
 \end{equation*}
Since the latter series is convergent for any $r\in [0,1)$, we deduce that
$$
\sum_{p=0}^\infty  \sum_{{(p_1,\ldots, p_k)\in\ZZ_+^k}\atop{p_1+\cdots +p_k=p}}
 \|\sum\limits_{{ \beta_i\in \FF_{n_i}^+, |\beta_i|=p_i}\atop{i\in \{1,\ldots, k\} }}A_{(\beta)}\otimes (rs)^{p_1+\cdots +p_k} {\bf S}_{(\beta)}\|<\infty,
 $$
which implies  that
$$
\sum_{p=0}^\infty  \sum_{{(p_1,\ldots, p_k)\in\ZZ_+^k}\atop{p_1+\cdots +p_k=p}}
 \|\sum\limits_{{ \beta_i\in \FF_{n_i}^+, |\beta_i|=p_i}\atop{i\in \{1,\ldots, k\} }}A_{(\beta)}\otimes   {\bf X}_{(\beta)}\|<\infty
 $$
for any ${\bf X}\in \rho{\bf B_n}(\cH)^-$, where $\rho\in (\gamma, s)$, which contradicts  Theorem \ref{Hadamard} (see the end of its proof).
Therefore, item (ii) holds.
\end{proof}

A closer look at the proofs of  Theorem \ref{Hadamard} and Corollary \ref{Hadamard2} reveals the following result.

 \begin{corollary}  The  radius of convergence of a power series $\varphi=\sum\limits_{(\alpha)} A_{(\alpha)}\otimes Z_{(\alpha)}$  satisfies the relation
 \begin{equation*}\begin{split}
 \gamma &=\sup\left\{r\geq 0:\
 \sum\limits_{q=0}^\infty\sum\limits_{  {|\alpha_1|+\cdots +|\alpha_k|=q}\atop {\alpha_i\in \FF_{n_i}^+}}A_{(\alpha)}\otimes r^q{\bf S}_{(\alpha)} \ \text{is convergent in the operator norm}\right\}\\
 &=
 \sup\left\{r\geq 0:\
 \sum\limits_{(p_1,\ldots, p_k)\in \ZZ_+^k}\sum\limits_{{ \alpha_i\in \FF_{n_i}^+, |\alpha_i|=p_i}\atop{i\in \{1,\ldots, k\} }}A_{(\alpha)}\otimes  r^{p_1+\cdots +p_k}{\bf S}_{(\alpha)} \ \text{is convergent in the operator norm}\right\}.
 \end{split}
 \end{equation*}
 \end{corollary}
Moreover, we have the following characterization for free holomorphic functions on polyballs.

\begin{corollary} Let   ${\bf S} $
    be the universal model associated with the abstract regular polyball
  ${\bf B_n}$. A formal power series $\varphi=\sum\limits_{(\alpha)} A_{(\alpha)}\otimes Z_{(\alpha)}$ is a  free holomorphic function (with coefficients in $B(\cK)$) on the
abstract  polyball
$\boldsymbol\rho{\bf B_n}$, where
  $\boldsymbol\rho=(\rho_1,\ldots,\rho_k)$,  $\rho_i>0$, if and only if  the series
$$
\sum_{q=0}^\infty \sum_{{(\alpha)\in \FF_{n_1}^+\times \cdots \times\FF_{n_k}^+ }\atop {|\alpha_1|+\cdots +|\alpha_k|=q}} A_{(\alpha)} \otimes r^q \rho_1^{|\alpha_1|}\cdots \rho_k^{|\alpha_k|}{\bf S}_{(\alpha)}
$$
is convergent in the operator norm topology for any $r\in [0,1)$.
 Moreover, the set $Hol({\bf \boldsymbol\rho B_n})$  of all free holomorphic functions on  ${\bf \boldsymbol\rho B_n}$ is an algebra.
\end{corollary}

\smallskip

\section{Maximum principle and Schwarz type results}

In this section, we  present some results concening the composition of free holomorphic functions and study bounded free holomorphic functions on polyball.  We prove a Schwarz lemma, and  a maximum principle in this setting. The results play an important role in the next sections.

Let
$H^\infty({\bf B_n})$  denote the set of  all
elements $\varphi$ in $Hol({\bf B_n})$     such
that
$$\|\varphi\|_\infty:= \sup \|\varphi({\bf X})\|<\infty,
$$
where the supremum is taken over all   $ {\bf X}\in {\bf B_n}(\cH)$ and any Hilbert space
$\cH$. One can  show that $H^\infty({\bf B_n})$ is a Banach algebra under pointwise multiplication and the
norm $\|\cdot \|_\infty$.
For each $p\in \NN$, we define the norms $\|\cdot
\|_p:M_{p\times p}\left(H^\infty({\bf B_n})\right)\to
[0,\infty)$ by setting
$$
\|[\varphi_{st}]_{p\times p}\|_p:= \sup \|[\varphi_{st}({\bf X})]_{p\times p}\|,
$$
where the supremum is taken over all  $ {\bf X}\in {\bf B_n}(\cH)$ and any Hilbert space
$\cH$.
It is easy to see that the norms  $\|\cdot\|_p$, $p\in \NN$,
determine  an operator space structure  on $H^\infty({\bf B_n})$,
 in the sense of Ruan (\cite{Pa-book}, \cite{Pi-book}).

Given  $\varphi\in {\bf F}_{\bf n}^\infty$ and a Hilbert space $\cH$, the noncommutative Berezin
transform  associated with the abstract noncommutative polyball ${\bf B_n}$
generates a function whose representation on $\cH$ is
$$
\boldsymbol{\cB}[\varphi]:{\bf B_n}(\cH)\to B(\cH)
$$
defined by
$$
\boldsymbol{\cB}[\varphi]({\bf X}):=\boldsymbol{\cB}_{\bf X}[\varphi],\qquad
{\bf X}\in {\bf B_n}(\cH),
$$
where
 $ \boldsymbol{\cB_{\bf X}}: B(\otimes_{i=1}^k F^2(H_{n_i}))\to B(\cH)$ is the {\it Berezin transform at} $X$
 defined by

 \begin{equation*}
 {\boldsymbol\cB_{\bf X}}[g]:= {\bf K^*_{\bf X}} (g\otimes I_\cH) {\bf K_{\bf X}},
 \qquad g\in B(\otimes_{i=1}^k F^2(H_{n_i})),
 \end{equation*}
 where $F^2(H_{n_i})$ is the full Fock space on $n_i$ generators and
 $${\bf K_{\bf X}}: \cH \to F^2(H_{n_1})\otimes \cdots \otimes  F^2(H_{n_k})
  \otimes  \overline{{\bf \Delta_{X}}(I) (\cH)}$$
 is the {\it noncommutative Berezin kernel} associated with
   ${\bf X}$

We call ${\boldsymbol\cB}[\varphi]$  the {\it  Berezin transform of} $\varphi$. In \cite{Po-Berezin-poly},
we identified the  noncommutative algebra
${\bf F}_{\bf n}^\infty$ with the  Hardy subalgebra $H^\infty({\bf B_n})$ of   bounded free holomorphic functions on
${\bf B_n}$.
More precisely, we proved that  he map
$ \Phi:H^\infty({\bf B_n})\to {\bf F}_{n}^\infty $ defined by
$$
\Phi\left(\sum\limits_{(\alpha)} a_{(\alpha)} Z_{(\alpha)}\right):=\sum\limits_{(\alpha)} a_{(\alpha)} {\bf S}_{(\alpha)}
$$
is a completely isometric isomorphism of operator algebras.
Moreover, if  $g:=\sum\limits_{(\alpha) } a_{(\alpha)} Z_{(\alpha)}$
is  a free holomorphic function on the abstract  polyball ${\bf B_n}$, then the following statements are equivalent:
 \begin{enumerate}
 \item[(i)]$g\in H^\infty({\bf B_n})$;
\item[(ii)] $\sup\limits_{0\leq r<1}\|g(r{\bf S})\|<\infty$, where $g(r{\bf S}):=\sum_{q=0}^\infty \sum_{{(\alpha)\in \FF_{n_1}^+\times \cdots \times\FF_{n_k}^+ }\atop {|\alpha_1|+\cdots +|\alpha_k|=q}} r^q a_{(\alpha)} {\bf S}_{(\alpha)}$;
\item[(iii)]
there exists $\varphi\in {\bf F}_{\bf n}^\infty $ with $g={\boldsymbol\cB}[\varphi]$, where ${\boldsymbol\cB}$ is the  noncommutative Berezin
transform  associated with the abstract   polyball ${\bf B_n}$.
\end{enumerate}

In this case,
$$
\Phi(g)=\text{\rm SOT-}\lim_{r\to 1}g(r{\bf S}),   \quad
 \quad  \Phi^{-1}(\varphi)=\boldsymbol{\cB}[\varphi],\quad
\varphi\in {\bf F}_{\bf n}^\infty,
$$
and
\begin{equation*}
\|\Phi(g)\|=\sup_{0\leq r<1}\|g(r{\bf S})\|=
\lim_{r\to 1}\|g(r{\bf S})\|.
\end{equation*}
We use the notation $\hat g:=\Phi(g)$ and call $\hat g$ the ({\it model}) {\it boundary function} of $g$ with respect to the universal model ${\bf S}$.
 We  denote by  $A({\bf B_n})$   the set of all  elements $g$
  in $Hol({\bf B_n})$   such that the mapping
$${\bf B_n}(\cH)\ni {\bf X}\mapsto
g({\bf X})\in B(\cH)$$
 has a continuous extension to  $[{\bf B_n}(\cH)]^-$ for any Hilbert space $\cH$. One can show that  $A({\bf B_n})$ is a  Banach algebra under pointwise
multiplication and the norm $\|\cdot \|_\infty$, and it has an operator space structure under the norms $\|\cdot \|_p$, $p\in \NN$.  Moreover, we can
identify the polyball algebra $\boldsymbol\cA_{\bf n}$ with the subalgebra
 $A({\bf B_n})$.  We proved in \cite{Po-Berezin-poly} that  the map
$ \Phi:A({\bf B_n})\to \boldsymbol\cA_{\bf n} $
defined by
$$
\Phi\left(\sum\limits_{(\alpha)} a_{(\alpha)} Z_{(\alpha)}\right):=\sum\limits_{(\alpha)} a_{(\alpha)} {\bf S}_{(\alpha)}
$$
is a completely isometric isomorphism of operator algebras.
Moreover, if  $g:=\sum\limits_{(\alpha)} a_{(\alpha)} Z_{(\alpha)}$
is  a free holomorphic function on the abstract polyball ${\bf B_n}$,
then the following statements are equivalent:
 \begin{enumerate}
 \item[(i)]$g\in A({\bf B_n})$;
\item[(ii)] $g(r{\bf S}):=\sum_{q=0}^\infty \sum_{{(\alpha)\in \FF_{n_1}^+\times \cdots \times\FF_{n_k}^+ }\atop {|\alpha_1|+\cdots +|\alpha_k|=q}} r^q  a_{(\alpha)} {\bf S}_{(\alpha)}$ is convergent in the  norm topology  as $r\to 1$;
\item[(iii)]
there exists $\varphi\in \boldsymbol\cA_{n}$ with $g={\boldsymbol\cB}[\varphi]$, where ${\boldsymbol\cB}$ is the  noncommutative Berezin
transform  associated with the abstract  polyball ${\bf B_n}$.
\end{enumerate}

In this case,
$$
\Phi(g)=\lim_{r\to 1}g(r{\bf S})  \quad \text{ and } \quad  \Phi^{-1}(\varphi)=\boldsymbol{\cB}[\varphi],\quad \varphi\in \boldsymbol\cA_{\bf n},
$$
where the limit is in the operator norm topology.

\begin{lemma} \label{pure} Let $F:{\bf B_n}(\cH)^-\to B(\cH)^{m_1}\times\cdots \times B(\cH)^{m_q}$ be a free holomorphic function on ${\bf B_n}(\cH)$ and  continuous on ${\bf B_n}(\cH)^-$. If ${\bf X}\in {\bf B_n}(\cH)^-$  and $\hat F\in {\bf B_m}(\otimes_{i=1}^k F^2(H_{n_i}))^-$ are pure elements, then so is $F({\bf X})\in {\bf B_m}(\cH)^-$.
\end{lemma}
\begin{proof} Let $f:{\bf B_n}(\cH)^-\to [B(\cH)]_1^-$ be a free holomorphic function on ${\bf B_n}(\cH)$ and continuous on ${\bf B_n}(\cH)^-$.
 If ${\bf X}\in {\bf B_n}(\cH)^-$ is pure, we can apply the noncommutative Berezin transform and obtain
 $$
 f({\bf X})f({\bf X})^*=\lim_{r\to 1} \boldsymbol{\cB}_{r{\bf X}}[\hat f\hat f^*]=\lim_{r\to 1} \boldsymbol{\cB}_{{\bf X}}[\hat f_r\hat f_r^*].
 $$ Since $\lim_{r\to 1} \hat f_r=\hat f$ in norm  and $ \boldsymbol{\cB}_{{\bf X}}$ is continuous in norm, we deduce that
 $ f({\bf X})f({\bf X})^*= \boldsymbol{\cB}_{{\bf X}}[\hat f\hat f^*]$.
 In a similar manner, if $F=(F_1,\ldots, F_q)$ and
 $i\in \{1,\ldots, q\}$, we  obtain
$$
\sum_{\alpha\in \FF_{m_i}^+, |\alpha|=p} F_{i,\alpha}({\bf X})F_{i,\alpha}({\bf X})^*=\boldsymbol{\cB}_{\bf X}\left[\sum_{\alpha\in \FF_{m_i}^+, |\alpha|=p} \hat F_{i,\alpha} \hat F_{i,\alpha}^*\right].
$$
Since $$\left\|\sum_{\alpha\in \FF_{m_i}^+, |\alpha|=p} \hat F_{i,\alpha} \hat F_{i,\alpha}^*\right\|\leq 1\quad \text{ and } \quad \sum_{\alpha\in \FF_{m_i}^+, |\alpha|=p} \hat F_{i,\alpha} \hat F_{i,\alpha}^*\to 0 \  \text{ strongly as } p\to \infty,
$$
  we deduce that $F_i({\bf X})$ is pure and, therefore, so is $F({\bf X})$. The proof is complete.
\end{proof}

 \begin{proposition} \label{range}  Let $G:{\bf B_n}(\cH)\to B(\cH)^{m_1}\times\cdots \times B(\cH)^{m_q}$ be a free holomorphic function, where  ${\bf n}=(n_1,\ldots, n_k)\in \NN^k$ and  ${\bf m}=(m_1,\ldots, m_q)\in \NN^q$. Then  $\text{\rm range}\, G\subseteq {\bf B_m}(\cH)$ if and only if
 $$
 G(r{\bf S})\in {\bf B_m}(\otimes_{i=1}^k F^2(H_{n_i})),\qquad r\in [0,1),
 $$
 where ${\bf S}$ is the universal model of the regular  polyball ${\bf B_n}$.
 \end{proposition}

 \begin{proof}  Since $r{\bf S}\in {\bf B_n}(\otimes_{i=1}^k F^2(H_{n_i}))$ for any $r\in [0,1)$, the direct implication is obvious. To prove the converse, assume that $G=(G_1,\ldots, G_q)$ has the property that $G(r{\bf S})\in {\bf B_m}(\otimes_{i=1}^k F^2(H_{n_i}))$. Consequently, if $i,s\in \{1,\ldots, q\}$, $i\neq s$, then each entry of $G_i(r{\bf S})=(G_{i,1}(r{\bf S}),\ldots, G_{i,m_i}(r{\bf S}))$ commutes with each entry of
 $G_s(r{\bf S})=(G_{s,1}(r{\bf S}),\ldots, G_{s,m_s}(r{\bf S}))$.  Moreover,
 $G(r{\bf S})$ is a pure element with entries $\{G_{i,j}(r{\bf S})\}$ in the noncommutative polyball algebra $\boldsymbol\cA_{\bf n}$ and, for each $r\in[0,1)$,
 \begin{equation*}
 (id-\Phi_{G_1(r{\bf S})})\circ\cdots\circ (id-\Phi_{G_q(r{\bf S})})(I)>d_rI,
 \end{equation*}
 for some $d_r>0$. If ${\bf X}=(X_1,\ldots, X_k)\in {\bf B_n}(\cH)$, then there is $t\in (0,1)$ such that ${\bf X}\in t{\bf B_n}(\cH)$. Since $G$ is a free holomorphic function, it is continuous and $G(t{\bf S})$ has the entries in $\boldsymbol\cA_{\bf n}$.
 Applying the noncommutative Berezin transform at $\frac{1}{t}{\bf X}$ to the relations mentioned above, when $r=t$, we deduce that  the entries
    of $G_i({\bf X}) $ commute with the entries of
 $G_s({\bf X})$, if $i,s\in \{1,\ldots, q\}$, $i\neq s$, and
 \begin{equation}\label{inG}
 (id-\Phi_{G_1({\bf X})})\circ\cdots\circ (id-\Phi_{G_q({\bf X})})(I)>0.
 \end{equation}
On the other hand, since $G_i(t{\bf S})$ is pure, Lemma \ref{pure} implies that $G_i({\bf X})$ is pure. Hence, and using relation \eqref{inG}, we conclude that $G({\bf X})\in {\bf B_n}(\cH)$ for any ${\bf X}\in {\bf B_n}(\cH)$. The proof is complete.
\end{proof}

 Using Proposition \ref{range} and  the properties of the noncommutative Berezin transform, one can easily deduce the follow result.

 \begin{corollary} \label{Gt} Let $G=(G_1,\ldots, G_q)$, with $G_i:{\bf B_n}(\cH)\to B(\cH)^{m_i}$, be a free holomorphic function  such that, for each $r\in [0,1)$,
 \begin{enumerate}
 \item[(i)]$\|G_t(r{\bf S})\|<1$, $t\in \{1,\ldots, q\}$;
 \item[(ii)] the entries of
 $G_t(r{\bf S})$ are commuting with the entries of $G_s(r{\bf S})$ for any $s,t\in \{1,\ldots, q\}$ with $s\neq t$.
 \end{enumerate}
 Then $\text{\rm range}\, G\subseteq {\bf B_m}(\cH)$ if either one of the following conditions holds:
 \begin{enumerate}
 \item[(a)] ${\bf \Delta}_{G(r{\bf S})}(I)>0$ for any $r\in [0,1)$;
 \item[(b)]  the entries of
 $G_t(r{\bf S})$ are  doubly commuting with the entries of $G_s(r{\bf S})$ for any $s,t\in \{1,\ldots, q\}$ with $s\neq t$.
 \end{enumerate}
 \end{corollary}

 \begin{theorem} \label{homo-compo}  Let   ${\bf n}=(n_1,\ldots, n_k)\in \NN^k$ and  ${\bf m}=(m_1,\ldots, m_q)\in \NN^q$. If $G:{\bf B_n}(\cH)\to {\bf B_m}(\cH)$ and $F:{\bf B_m}(\cH)\to B(\cH)\bar\otimes_{min} B(\cE,\cG)$ are free holomorphic functions on regular polyballs, then
 $F\circ G$ is a free holomorphic function on  ${\bf B_n}(\cH)$.

 \end{theorem}

 \begin{proof}
 If $F$ has the Fourier representation
 $$
 F({\bf Y})=
 \sum\limits_{p=0}^\infty\sum\limits_{  {|\gamma_1|+\cdots +|\gamma_q|=p}\atop {\gamma_i\in \FF_{m_i}^+}}A_{(\gamma)}\otimes Y_{(\gamma)}, \qquad {\bf Y}\in {\bf B_m}(\cH),
 $$
 then we have
 $$
 (F\circ G)({\bf X})=
 \sum\limits_{p=0}^\infty\sum\limits_{  {|\gamma_1|+\cdots +|\gamma_q|=p}\atop {\gamma_i\in \FF_{m_i}^+}}A_{(\gamma)}\otimes G_{(\gamma)}({\bf X}), \qquad {\bf X}\in {\bf B_n}(\cH),
 $$
 where the convergence is in the operator norm topology.
 Due to Proposition \ref{range},
$$
 G(r{\bf S})=\{G_{s,t}({\bf S})\}\in {\bf B_m}(F^2(H_{n_1})\otimes \cdots \otimes F^2(H_{n_k})),\qquad r\in [0,1),
 $$
 where  $s\in\{1,\ldots, q\}$, $t\in \{1,\ldots, m_s\}$ and ${\bf S}$ is the universal model of the regular  polyball ${\bf B_n}$.
Since $F:{\bf B_m}(\cH)\to B(\cH)\bar\otimes_{min} B(\cE,\cG)$ is a free holomorphic function, for each $r\in[0,1)$,
\begin{equation}
\label{GAR}
\Lambda_r:= \sum\limits_{p=0}^\infty\sum\limits_{  {|\gamma_1|+\cdots +|\gamma_q|=p}\atop {\gamma_i\in \FF_{m_i}^+}}A_{(\gamma)}\otimes G_{(\gamma)}(r{\bf S}), \qquad {\bf X}\in {\bf B_n}(\cH),
 \end{equation}
 where the convergence is in the operator norm topology. Taking into account that
 $G_{s,t}({\bf S})$ is in the noncommutative polyball algebra $\boldsymbol\cA_{\bf n}$, we have
 $\Lambda_r\in B(\cE,\cG)\otimes \boldsymbol\cA_{\bf n}\subset B(\cE,\cG)\bar\otimes F_{\bf n}^\infty$. This implies that, for each $r\in [0,1)$, the operator $\Lambda_r$ has the Fourier representation
$
 \sum\limits_{q=0}^\infty\sum\limits_{  {|\alpha_1|+\cdots +|\alpha_k|=q}\atop {\alpha_i\in \FF_{n_i}^+}}C_{(\alpha)}(r)\otimes r^{|\alpha_1|+\cdots +|\alpha_k|} {\bf S}_{(\alpha)}
 $
 and
 \begin{equation}
 \label{La-SOT}
 \Lambda_r=\text{\rm SOT-}\lim_{\ell\to 1}\sum\limits_{q=0}^\infty\sum\limits_{  {|\alpha_1|+\cdots +|\alpha_k|=q}\atop {\alpha_i\in \FF_{n_i}^+}}C_{(\alpha)}(r)\otimes (r\ell)^{|\alpha_1|+\cdots +|\alpha_k|} {\bf S}_{(\alpha)}.
 \end{equation}
 where the series converges in the operator topology. The next step in our proof is to show that
 $C_{(\alpha)}(r)$ does not depend on $r\in [0,1)$. Using relations
 \eqref{GAR} and \eqref{La-SOT}, we have
 \begin{equation*}
 \begin{split}
 \left<C_{(\alpha)}(r)x,y\right>&=\left<(I\otimes  {\bf S}_{(\alpha)}^*)\frac{1}{r^{|\alpha_1|+\cdots +|\alpha_k|}} \Lambda_r(x\otimes 1),(y\otimes 1)\right>\\
 &=\lim_{d\to \infty} \sum\limits_{p=0}^d\sum\limits_{  {|\gamma_1|+\cdots +|\gamma_q|=p}\atop {\gamma_i\in \FF_{m_i}^+}}\left<A_{(\gamma)}x,y\right>\left<\frac{1}{r^{|\alpha_1|+\cdots +|\alpha_k|}}{\bf S}_{(\alpha)}^* G_{(\gamma)}(r{\bf S})1,1\right>
 \end{split}
 \end{equation*}
 for any   $(\alpha)\in \FF_{n_1}^+\times \cdots \times \FF_{n_k}^+$ and for any $x\in \cE$, $y\in \cG$.
 On the other hand, the product $G_{(\gamma)}$ is a free holomorphic function on ${\bf B_n}(\cH)$ and has a representation
 $$
 G_{(\gamma)}({\bf X})=\sum\limits_{p=0}^\infty\sum\limits_{  {|\beta_1|+\cdots +|\beta_q|=p}\atop {\beta_i\in \FF_{n_i}^+}}d^{(\gamma)}_{(\beta)} X_{(\beta)}, \qquad {\bf X}\in {\bf B_n}(\cH).
 $$
 Consequently,
 $$
 \left<\frac{1}{r^{|\alpha_1|+\cdots +|\alpha_k|}}{\bf S}_{(\alpha)}^* G_{(\gamma)}(r{\bf S})1,1\right>=d^{(\gamma)}_{(\alpha)}, \qquad r\in [0,1),
 $$
 for any  $(\alpha)\in \FF_{n_1}^+\times \cdots \times \FF_{n_k}^+$ and
 $(\gamma)\in \FF_{m_1}^+\times \cdots \times \FF_{m_k}^+$.
 Therefore, $C_{(\alpha)}(r)$ does not depend on $r\in[0,1)$. We set $C_{(\alpha)}(r)=C_{(\alpha)}$, and note that relation \eqref{La-SOT} implies that
 $$
 Q({\bf X}):=\sum\limits_{q=0}^\infty\sum\limits_{  {|\alpha_1|+\cdots +|\alpha_k|=q}\atop {\alpha_i\in \FF_{n_i}^+}}C_{(\alpha)}\otimes  { X}_{(\alpha)}, \qquad {\bf X}\in {\bf B_n}(\cH),
 $$
is a free holomorphic function on ${\bf B_n}(\cH)$. Moreover, since $Q$ is continuous in the operator norm  we deduce that
\begin{equation*}
\begin{split}
\Lambda_r:&= \sum\limits_{p=0}^\infty\sum\limits_{  {|\gamma_1|+\cdots +|\gamma_q|=p}\atop {\gamma_i\in \FF_{m_i}^+}}A_{(\gamma)}\otimes G_{(\gamma)}(r{\bf S}) =
\sum\limits_{q=0}^\infty\sum\limits_{  {|\alpha_1|+\cdots +|\alpha_k|=q}\atop {\alpha_i\in \FF_{n_i}^+}}C_{(\alpha)}\otimes r^{|\alpha_1|+\cdots +|\alpha_k|} {\bf S}_{(\alpha)}
\end{split}
\end{equation*}
for any $r\in [0,1)$.
Now, if ${\bf X}\in {\bf B_n}(\cH)$, then there is $r\in (0,1)$  such that ${\bf X}\in r{\bf B_n}(\cH)$. Applying the noncommutative Berezin transform at $\frac{1}{r}{\bf X}$ to the relation above, we deduce that
$$
(F\circ G)({\bf X})=
 \sum\limits_{p=0}^\infty\sum\limits_{  {|\gamma_1|+\cdots +|\gamma_q|=p}\atop {\gamma_i\in \FF_{m_i}^+}}A_{(\gamma)}\otimes G_{(\gamma)}({\bf X})=Q({\bf X})
$$
 for any ${\bf X}\in {\bf B_n}(\cH)$. The proof is complete.
\end{proof}

 \begin{proposition} \label{coef-ine}
 Let $F:{\bf B_n}(\cH)\to B(\cK)\bar \otimes_{min}B(\cH)$ be a bounded  free holomorphic function with coefficients in $B(\cK)$ and representation
 $$
 \sum\limits_{q=0}^\infty\sum\limits_{  {|\alpha_1|+\cdots +|\alpha_k|=q}\atop {\alpha_i\in \FF_{n_i}^+}}A_{(\alpha)}\otimes X_{(\alpha)}.
 $$
If  $\|F\|_\infty\leq 1$, then
 $$
 \sum\limits_{  {|\alpha_1|+\cdots +|\alpha_k|=q}\atop {\alpha_i\in \FF_{n_i}^+}}A_{(\alpha)}^*A_{(\alpha)}\leq I-A_{(0)}^*A_{(0)}
 $$
 for any $q\in \NN$, where $A_{(0)}:=F(0)$.
 \end{proposition}

  \begin{proof} Let $\cM$ be the subspace of $F^2(H_{n_1})\otimes \cdots \otimes F^2(H_{n_k})$ spanned by the vectors $1, e_{\alpha_1}^1\otimes \cdots \otimes e_{\alpha_k}^k$, where $\alpha_i\in \FF_{n_i}$ and $|\alpha_1|+\cdots +|\alpha_k|=q\in \NN$. Note that the operator $C:=P_{\cK\otimes \cM} F({\bf S})|_{\cK\otimes \cM}$ is a contraction and,  with respect to the decomposition
  $$
  \cK\otimes \cM=\cH\oplus\bigoplus_{  {|\alpha_1|+\cdots +|\alpha_k|=q}\atop {\alpha_i\in \FF_{n_i}^+}}e_{\alpha_1}^1\otimes \cdots \otimes e_{\alpha_k}^k\otimes \cH,
  $$
  has the operator matrix representation

  $$
  \left(
  \begin{matrix}
  A_{(0)} &[\begin{matrix}0&\cdots&\cdots&\cdots&0\end{matrix}]\\
  \left[\begin{matrix}
  A_{(\alpha)}\\
  \vdots\\
  {  {|\alpha_1|+\cdots +|\alpha_k|=q}\atop {\alpha_i\in \FF_{n_i}^+}}
  \end{matrix}\right]
  &
  \left[\begin{matrix}
  A_{(0)}&0&\cdots&0&0\\
  0&0&\cdots&0&0\\
  \vdots &\ddots&\vdots\\
  0&0&\cdots &0&A_{(0)}
  \end{matrix}\right]
  \end{matrix}
  \right).
  $$
  Indeed, we have
  $$
  \left<C(x\otimes 1),y\otimes 1\right>=\left<F({\bf S})(x\otimes 1),y\otimes 1\right>=\left<A_{(0)}x,y\right>
  $$
  and
  $$
  \left<C(x\otimes 1),y\otimes e_{\alpha_1}^1\otimes \cdots \otimes e_{\alpha_k}^k\right>=\left< A_{(\alpha)}x,y\right>
  $$
  for any $(\alpha)=(\alpha_1,\ldots, \alpha_k)\in \FF_{n_1}^+\times \cdots \times \FF_{n_k}^+$ with
  $|\alpha_1|+\cdots +|\alpha_k|=q$.
  If  $|\alpha_1|+\cdots +|\alpha_k|=|\beta_1|+\cdots +|\beta_k|=q$, then we have

  $$
  \left<C(x\otimes e_{\alpha_1}^1\otimes \cdots \otimes e_{\alpha_k}^k),y\otimes e_{\beta_1}^1\otimes \cdots \otimes e_{\beta_k}^k\right>=\delta_{\alpha_1\beta_1}\cdots \delta_{\alpha_k\beta_k} A_{(0)}
  $$
  for any $x,y\in \cK$. This proves our assertion. Consequently, the column operator matrix

$$
  \left(
  \begin{matrix}
  A_{(0)}  \\
  \left[\begin{matrix}
  A_{(\alpha)}\\
  \vdots\\
  {  {|\alpha_1|+\cdots +|\alpha_k|=q}\atop {\alpha_i\in \FF_{n_i}^+}}
  \end{matrix}\right]
  \end{matrix}
  \right)
  $$
is a contraction, which completes the proof.
\end{proof}

 We  recall that  ${\bf B_n}(\cH)$ is a complete Reinhardt domain and
 $$
 B(\cH)^{n_1}\times_c\cdots \times_c B(\cH)^{n_k}=\bigcup_{\rho>0} \rho{\bf B_n}(\cH).
$$
We define the Minkovski functional associated with the regular polyball ${\bf B_n}(\cH)$   to be   the function
 $m_{\bf B_n}:B(\cH)^{n_1}\times_c\cdots \times_c B(\cH)^{n_k}\to [0,\infty)$  given by
 $$
 m_{\bf B_n}({\bf X}):=\inf \left\{r>0: \ {\bf X}\in r{\bf B_n}(\cH)\right\}.
 $$

 \begin{proposition} The  Minkovski functional associated with the regular polyball ${\bf B_n}(\cH)$ has the following properties:
 \begin{enumerate}
 \item[(i)] $m_{\bf B_n}(\lambda{\bf X})=|\lambda| m_{\bf B}({\bf X})$ for $\lambda\in \CC$;

 \item[(ii)] $m_{\bf B_n}$ is upper semicontinuous;
 \item[(iii)]  ${\bf B_n}(\cH)=\{{\bf X}\in B(\cH)^{n_1}\times_c\cdots \times_c B(\cH)^{n_k}: \ m_{\bf B_n}({\bf X})<1\}$;
     \item[(iv)] ${\bf B_n}(\cH)^-=\{{\bf X}\in B(\cH)^{n_1}\times_c\cdots \times_c B(\cH)^{n_k}: \ m_{\bf B_n}({\bf X})\leq 1\}$;
         \item[(v)] There is a polyball $r{\bf P_n}(\cH)\subset{\bf B_n}(\cH)$ for some $r\in (0,1)$, where $m_{\bf B_n}$ is continuous.
 \end{enumerate}

 \end{proposition}

 \begin{proof} To prove (i), we may assume that ${\bf X}\neq 0$ and $\lambda\neq 0$. It is clear that
 $m_{\bf B_n}(\lambda{\bf X})=t>0$ if and only if
 $\lambda{\bf X}\in c{\bf B_n}(\cH)$ for any $c>t$, and
 $\lambda{\bf X}\notin d{\bf B_n}(\cH)$ if $0<d<t$. Taking into account that
 ${\bf B_n}(\cH)=e^{i\theta}{\bf B_n}(\cH)$ for any $\theta\in \RR$, we deduce that the latter conditions are equivalent to
 ${\bf X}\in \frac{c}{|\lambda|}{\bf B_n}(\cH)$ for any $c>t$ and
  ${\bf X}\notin \frac{d}{|\lambda|}{\bf B_n}(\cH)$  if $0<d<t$. Hence, we obtain that $m_{\bf B_n}({\bf X})=\frac{t}{|\lambda|}$, which shows that item (i). We skip the proof of item (ii), since it is due to (i) and  a  straightforward argument.

 According to Proposition \ref{reg-poly}, we have $ {\bf B_n}(\cH)=\bigcup_{0<r<1}r{\bf B_n}(\cH)$. Using this result, one can easily deduce item (iii). As we saw in the proof of the same proposition, for any $r\in (0,1)$, we have ${\bf B_n}(\cH)^-\subseteq \frac{1}{r}{\bf B_n}(\cH)$. Consequently, $m_{\bf B_n}({\bf X})\leq 1$ for any ${\bf X}\in  {\bf B_n}(\cH)^-$.  Now, assume that ${\bf X}\in B(\cH)^{n_1}\times_c\cdots \times_c B(\cH)^{n_k}$ is such that $m_{\bf B_n}({\bf X})= 1$. Then there is a sequence $\{t_m\}$ with $t_m>1$ and $t_m\to 1$ such that ${\bf X}\in t_m {\bf B_n}(\cH)$ for any $m\in \NN$. Taking $t_m\to 1$, we deduce that ${\bf X}\in {\bf B_n}(\cH)^-$. Hence, and using item (iii), one can see that item (iv) holds. To prove (v), note that the fact that $r{\bf P_n}(\cH)\subset{\bf B_n}(\cH)$ for some $r\in (0,1)$ is quite clear, while the continuity of
 $m_{\bf B_n}$ on $r{\bf P_n}(\cH)$ is due to the convexity of the latter polyball. The proof is complete.
\end{proof}

 Let $\CC\left<Z_{i,j}\right>$ be the algebra of all polynomials in indeterminates $Z_{i,j}$, where $i\in \{1,\ldots, k\}$ and $j\in \{1,\ldots, n_i\}$.
 We define the free  partial derivation  $\frac{\partial } {\partial Z_{i,j}}$  on $\CC\left<Z_{i,j}\right>$ as the unique linear operator  on this algebra, satisfying the conditions
$$
\frac{\partial I} {\partial Z_{i,j}}=0, \quad  \frac{\partial Z_{i,j}}
{\partial Z_{i,j}}=I, \quad  \frac{\partial Z_{i,j}} {\partial Z_{s,q}}=0\
\text{ if }  \ (i,j)\neq (s,q)
$$
and
$$
\frac{\partial (fg)} {\partial Z_{i,j}}=\frac{\partial f} {\partial Z_{i,j}}
g +f\frac{\partial g} {\partial Z_{i,j}}
$$
for any  $f,g\in \CC\left<Z_{i,j}\right>$.   The same definition extends to formal power series in the
noncommuting indeterminates $Z_{i,j}$.
  If $F:=\sum\limits_{(\alpha)\in \FF_{n_1}^+\times \cdots \times \FF_{n_k}^+}A_{(\alpha)}\otimes  Z_{(\alpha)}$ is a
  power series with  operator-valued coefficients, then   the    free partial derivative
    of $F$ with respect to $Z_{i,j} $ is the power series
$
\frac {\partial F}{\partial Z_{i,j}} := \sum_{\alpha\in \FF_n^+}A_{(\alpha)} \otimes
 \frac {\partial Z_{(\alpha)}}{\partial Z_{i,j} }.
$
One can prove  that if $F$ is a free
holomorphic function  on ${\bf B_n}(\cH)$ then so is
 $\frac {\partial F}{\partial Z_{i,j}}$. We leave the proof to the reader.

The next result is an analogue of Schwarz lemma from complex analysis.

 \begin{theorem} \label{Schwarz} Let $F:{\bf B_n}(\cH)\to B(\cH)^p$ be a bounded  free holomorphic function with $\|F\|_\infty\leq 1$. If $F(0)=0$, then
 $$
 \|F({\bf X})\|\leq m_{\bf B_n}({\bf X})<1 \quad \text{ and } \quad m_{\bf B_n}({\bf X})\leq \|{\bf X}\|,\qquad  {\bf X}\in {\bf B_n}(\cH),
 $$
 where $m_{\bf B_n}$ is the Minkovski functional associated with the regular polyball ${\bf B_n}(\cH)$.
In particular, if $p=1$, the free holomorphic function
 $$
 \psi({\bf X})=\sum_{i=1}^k\sum_{j=1}^{n_j}\frac {\partial F}{\partial Z_{i,j}}(0) X_{i,j},\qquad {\bf X}=(X_{i,j})\in {\bf B_n}(\cH),
 $$
  has the property that $\|\psi({\bf X})\|\leq m_{\bf B_n}({\bf X})<1$.
 \end{theorem}
  \begin{proof} Fix ${\bf X}\in {\bf B_n}(\cH)$ and let $t\in (0,1)$ be such that
  $m_{\bf B_n}({\bf X})<t<1$.
  Since $\frac{1}{t}{\bf X}\in {\bf B_n}(\cH)$,   Proposition \ref{reg-poly} implies   $\frac{\lambda}{t}{\bf X}\in {\bf B_n}(\cH)$ for any $\lambda\in \DD:=\{z\in \CC: \ |z|<1\}$.
  For each $x,y\in \cH^{(p)}$ with $\|x\|\leq 1$ and $\|y\|\leq 1$, define the function  $\varphi_{x,y}:\DD\to \CC$ by setting
  $$
  \varphi_{x,y}(\lambda):=\left<F\left(\frac{\lambda}{t}{\bf X}\right)x,y\right>,\qquad \lambda\in \DD.
  $$
Taking into account that  $F$ is free holomorphic on ${\bf B_n}(\cH)$ and $\|F\|_\infty\leq 1$, we deduce that
$\varphi_{x,y}$ is a holomorphic function on the unit disc and $|\varphi_{x,y}(\lambda)|\leq 1$.
Since $\varphi_{x,y}(0)=0$, an application of the classical Schwarz lemma to $\varphi_{x,y}$ implies
$|\varphi_{x,y}(\lambda)|\leq |\lambda|$ for any $\lambda\in \DD$.
Taking $\lambda=m_{\bf B_n}({\bf X})$, we obtain
$$
  \varphi_{x,y}(\lambda):=\left<F\left(\frac{m_{\bf B_n}({\bf X})}{t}{\bf X}\right)x,y\right>\leq m_{\bf B_n}({\bf X}),\qquad \lambda\in \DD.
  $$
for any $t\in (0,1)$  with
  $m_{\bf B_n}({\bf X})<t<1$. Since $F$ is continuous on ${\bf B_n}(\cH)$ and taking $t\to m_{\bf B_n}({\bf X})$, we obtain
  $
  |\left<F({\bf X})x,y\right>|\leq m_{\bf B_n}({\bf X})
$
for any  $x,y\in \cH^{(p)}$ with $\|x\|\leq 1$ and $\|y\|\leq 1$. Consequently,
$$\|F({\bf X})\|\leq m_{\bf B_n}({\bf X})<1, \qquad {\bf X}\in {\bf B_n}(\cH).
 $$
 According to Proposition 1.9 from \cite{Po-Berezin-poly}, if
 $\|{\bf X}\|:=\Phi_{X_1}(I)+\cdots +\Phi_{X_k}(I)\leq I$ then ${\bf X}=(X_1,\ldots, X_k)\in {\bf B_n}(\cH)^-$. Consequently, if
 ${\bf X}\in {B_n}(\cH)$, then $\frac{\bf X}{\|{\bf X}\|}\in {\bf B_n}(\cH)^-$, which implies $m_{\bf B_n}({\bf X})\leq t \|{\bf X}\|$ for any $t>1$. taking $t\to 1$, we deduce that
   $m_{\bf B_n}({\bf X})\leq \|{\bf X}\|$.

   Now, we consider the particular case when $p=1$.
Due to the classical Schwarz lemma,  we also have $|\varphi_{x,y}'(0)|\leq 1$.
Since $\varphi_{x,y}'(0)=\left<\frac{1}{t}\psi({\bf X}) x,y\right>$, we deduce that $\|\psi({\bf X})\|\leq t<1$. Taking $t\to m_{\bf B_n}({\bf X})$, we obtain
 $\|\psi({\bf X})\|\leq m_{\bf B_n}({\bf X})<1$.
 The proof is complete.
 \end{proof}

We have all the ingredients to prove the following maximum principle.

 \begin{theorem}\label{Max} Let $F:{\bf B_n}(\cH)\to B(\cH)$ be a bounded  free holomorphic function. If  there exists ${\bf X}_0\in {\bf B_n}(\cH)$  such that
 $$
 \|F({\bf X})\|\leq \|F({\bf X}_0)\|, \qquad {\bf X}\in {\bf B_n}(\cH),
 $$
  then $F$ must be a constant.
\end{theorem}
  \begin{proof} Assume that $\|F\|_\infty=1$  and there exists ${\bf X}_0\in {\bf B_n}(\cH)$  such that  $\|F({\bf X}_0)\|=1$. Let $F$ have the representation
   $$
 \sum\limits_{q=0}^\infty\sum\limits_{  {|\alpha_1|+\cdots +|\alpha_k|=q}\atop {\alpha_i\in \FF_{n_i}^+}}a_{(\alpha)} X_{(\alpha)}.
 $$
According to Theorem  \ref{coef-ine}, we have
 $$
 \sum\limits_{  {|\alpha_1|+\cdots +|\alpha_k|=q}\atop {\alpha_i\in \FF_{n_i}^+}}|a_{(\alpha)}|^2\leq 1-|F(0)|^2
 $$
 for any $q\in \NN$.  Hence,  if $|F(0)|=1$, then $a_{(\alpha)}=0$ for any $(\alpha)=(\alpha_1,\ldots, \alpha_k)\in \FF_{n_1}^+\times \cdots \times \FF_{n_k}^+$ with $|\alpha_1|+\cdots +|\alpha_k|\geq 1$, which implies $F=F(0)$.

 Now, we assume that $|F(0)|<1$ and set $\lambda:=F(0)$. Note that if $\Psi_\lambda$ is the corresponding automorphism of the open unit ball $[B(\cH)]_1$ (see the remarks preceding Theorem \ref{structure}), then, due to Theorem \ref{homo-compo}, $G:=\Psi_\lambda\circ F$ is a free holomorphic function on ${\bf B_n}(\cH)$ with the property that
 $G(0)=0$ and $\|G\|_\infty\leq 1$. Using  Theorem \ref{Schwarz}, we have $\|G({\bf X})\|<1$ for any
  ${\bf X}\in {\bf B_n}(\cH)$. Hence,
  $\|\Psi_\lambda(F({\bf X}_0))\|<1$. Since $\Psi_\lambda$ is an involutive automorphism of the open unit ball $[B(\cH)]_1$, we deduce that
  $$
  \|F({\bf X})\|=\|\Psi_\lambda(\Psi_\lambda(F({\bf X})))\|<1
  $$
  for any
  ${\bf X}\in {\bf B_n}(\cH)$, which contradicts our assumption that $\|F({\bf X}_0)\|=1$.
  The proof is complete.
\end{proof}

 \begin{corollary} Let $F:{\bf B_n}(\cH)\to B(\cH)$ be a nonconstant  bounded free holomorphic function. Then the following statements hold:
 \begin{enumerate}
 \item[(i)] $\|F({\bf X})\|<\|F\|_\infty$ for any ${\bf X}\in{\bf B_n}(\cH)$;

 \item[(ii)] the map $$[0,1)^k\ni {\bf r} \mapsto \|F_{\bf r}\|_\infty, \qquad {\bf r}=(r_1,\ldots, r_k)
     $$ is strictly increasing with respect to each $r_i$, where
     $$F_{\bf r}(X_1,\ldots, X_k):=F(r_1X_1,\ldots, r_k X_k), \qquad (X_1,\ldots, X_k) \in {\bf B_n}(\cH).
     $$
 \end{enumerate}

 \end{corollary}

  \begin{proof} Without loss of generality, we may assume that $\|F\|_\infty=1$. Part (i) is a consequence of Theorem \ref{Max}. To prove part (ii), let $0\leq r_1<t_1<1$ and set $r:=\frac{r_1}{t_1}\in [0,1)$. Since $F$ is a free holomorphic function on ${\bf B_n}(\cH)$, the operator
  $F(r{\bf S}_1, r_2{\bf S}_2,\ldots, r_k{\bf S}_k)$ is in the polyball algebra $\boldsymbol\cA_{\bf n}$ and
  $\|F_{(r,r_2,\ldots, r_k)}\|_\infty=\| F(r{\bf S}_1, r_2{\bf S}_2,\ldots, r_k{\bf S}_k)\|$. Applying part (i) to the bounded free holomorphic function $F_{(r,r_2,\ldots, r_k)}$ on ${\bf B_n}(\cH)$ and ${\bf X}=(r{\bf S}_1, r_2{\bf S}_2,\ldots, r_k{\bf S}_k)$, we obtain
  \begin{equation*}
  \begin{split}
  \|F_{(r_1,r_2,\ldots, r_k)}\|_\infty
  &=\|F_{(r_1,r_2,\ldots, r_k)}({\bf S})\|_\infty
  =\left\|F_{(t_1,r_2,\ldots, r_k)}\left(\frac{r_1}{t_1}{\bf S}_1, {\bf S}_2,\ldots, {\bf S}_k\right)\right\|\\&
  <
\|F_{(t_1,r_2,\ldots, r_k)}({\bf S}_1, {\bf S}_2,\ldots, {\bf S}_k)\|
=
\|F_{(t_1,r_2,\ldots, r_k)}\|_\infty.
\end{split}
\end{equation*}
The proof is complete.
 \end{proof}

The next version of the maximum principle is needed in the next sections.

 \begin{theorem} \label{max-mod} Let $F:{\bf B_n}(\cH)\to B(\cH)^p $ be a bounded  free holomorphic function with $\|F(0)\|<\|F\|_\infty$.  Then   there is  no  ${\bf X}_0\in {\bf B_n}(\cH)$  such that
 $
 \|F({\bf X}_0)\|=\|F\|_\infty.
 $
 \end{theorem}
  \begin{proof} Without loss of generality we may assume that $\|F\|_\infty=1$. If $F(0)=0$, Theorem \ref{Schwarz} implies $\|F({\bf X})\|<1$ for any ${\bf X}\in {\bf B_n}(\cH)$, which completes the proof.

   Now we consider that case when $0\neq \|F(0)\|<1$.
   Suppose that there is ${\bf X}_0\in {\bf B_n}(\cH)$ such that $\|F({\bf X}_0)\|=1$. Since $\|F(0)\|<\|F\|_\infty=1$, we have $\lambda:=F(0)\in (\CC^p)_1$. Let $\Psi_\lambda$ be the   automorphism of the open unit ball $[B(\cH)^p]_1$ (see the remarks preceding Theorem \ref{structure}). We recall that  $\Psi_\lambda$ is a free holomorphic function
on $[B(\cH)^p]_\gamma$, where $\gamma:=\frac{1}{\|\lambda\|_2}$,  and $\Psi_\lambda(\Psi_\lambda(X))=X$
for any $X\in [B(\cH)^p]_\gamma$, where $\gamma:=\frac{1}{\|\lambda\|_2}$.
  Using Theorem \ref{homo-compo}, we deduce that
  $G:=\Psi_\lambda \circ F:{\bf B_n}(\cH)\to B(\cH)^p$ is a free holomorphic function on ${\bf B_n}(\cH)$  such that $G(0)=0$ and $\|G\|_\infty\leq 1$. Due to the Schwarz type result  of Theorem \ref{Schwarz}, we have $\|G({\bf X})\|<1$ for any
  ${\bf X}\in {\bf B_n}(\cH)$. In particular, we have
  $\|\Psi_\lambda(F({\bf X}_0))\|<1$. Since $\Psi_\lambda$ is an involutive automorphism of the open unit ball $[B(\cH)^p]_1$, we deduce that
  $$
  \|F({\bf X}_0)\|=\|\Psi_\lambda(\Psi_\lambda(F({\bf X}_0)))\|<1,
  $$
  which is a contradiction.
  The proof is complete.
  \end{proof}

\smallskip

 \section{Holomorphic automorphisms of noncommutative polyballs}

In this section,   we use noncommutative Berezin transforms to obtain a complete description of the group $\text{\rm Aut}({\bf B_n})$ of all free holomorphic automorphisms of the polyball ${\bf B_n}$, which is an analogue of Rudin's characterization of the holomorphic automorphisms of the polydisc, and prove   some of  their  basic properties. We show that
$
\text{\rm Aut}({\bf B_n}) \simeq\text{\rm Aut}((\CC^{n_1})_1\times\cdots \times  (\CC^{n_1})_1)
$
and obtain an analogue   of  Poincar\' e's classical result that  the open unit ball of $\CC^n$ is not biholomorphic  equivalent to the polydisk $\DD^n$, for noncommutative regular polyballs.

 \begin{proposition}  \label{CSH} If  ${\bf n}=(n_1,\ldots, n_k)\in \NN^k$, then the following statements hold.
 \begin{enumerate}
 \item[(i)] If $C_i\in B(\CC^{n_i})$, $i\in \{1,\ldots, k\}$,  are contractions, then $g:{\bf B_n}(\cH)^-\to {\bf B_n}(\cH)^-$, defined by
 $$g({\bf X} )=(X_1C_1,\ldots, X_kC_k), \qquad {\bf X}:=(X_1,\ldots, X_k)\in {\bf B_n}(\cH)^-,$$
 is a free holomorphic function on ${\bf B_n}(\cH)$. In particular, if each $C_i$ is a unitary operator, then $g|_{{\bf B_n}(\cH)}\in Aut({\bf B_n})$ and $g$ is a homeomorphism of ${\bf B_n}(\cH)^-$.
 \item[(ii)] If $\sigma$ is a permutation of the set  $\{1,\ldots, k\}$ such that $n_{\sigma(i)}=n_i$, then $p_\sigma:{\bf B_n}(\cH)^-\to {\bf B_n}(\cH)^-$, defined by
 $$p_\sigma({\bf X} )=(X_{\sigma (1)},\ldots, X_{\sigma(k)}), \qquad {\bf X}:=(X_1,\ldots, X_k)\in {\bf B_n}(\cH)^-,$$
 is a homeomorphism of ${\bf B_n}(\cH)^-$ and $p_\sigma|_{{\bf B_n}(\cH)}$ a free holomorphic automorphism of  ${\bf B_n}(\cH)$.
 \item[(iii)] If  $\varphi_i:[B(\cH)^{n_i}]_1\to [B(\cH)^{n_i}]_1^-$, $i\in \{1,\ldots, k\}$,  is a free holomorphic function, then $G:{\bf B_n}(\cH)\to B(\cH)^{n_1}\times \cdots \times B(\cH)^{n_k}$ defined by
     $$
 G({\bf X}):=(\varphi_1(X_1),\ldots, \varphi_k(X_k)),\qquad {\bf X}=(X_1,\ldots, X_k)\in {\bf B_n}(\cH),
     $$
     is a free holomorphic function on the regular polyball and $\text{\rm range}\, G\subseteq {\bf B_n}(\cH)$. In particular, if each $\varphi_i$ is a free holomorphic automorphism of the unit ball $[B(\cH)^{n_i}]_1$, then $G\in Aut({\bf B_n})$.
 \end{enumerate}
  \end{proposition}
  \begin{proof} The results are immediate consequences of Theorem \ref{homo-compo} and Corollary \ref{Gt}.
\end{proof}

 Let  $F:{\bf B_n}(\cH)\to B(\cH)^{n_1+\cdots n_k}$  be a free holomorphic function  with $F:=(F_1,\ldots,
 F_k)$ and $F_i=(F_{i,1},\ldots, F_{i,n_i})$, where each $F_{i,j}$ is a free holomorphic function on ${\bf B_n}(\cH)$ with scalar coefficients.   We define $F'(0)$ as the
 linear operator on $\CC^{n_1+\cdots +n_k}$ having the matrix
 \begin{equation*}
 \left[
 \begin{matrix}
 \frac {\partial F_{1,1}}{\partial Z_{1,1}}(0)&\cdots \frac {\partial F_{1,1}}{\partial Z_{1,n_1}}(0)&\cdots &\frac {\partial F_{1,1}}{\partial Z_{k,1}}(0)&\cdots \frac {\partial F_{1,1}}{\partial Z_{k,n_k}}(0)\\
 \vdots&\cdots &\vdots &\cdots &\vdots\\
 \frac {\partial F_{1,n_1}}{\partial Z_{1,1}}(0)&\cdots \frac {\partial F_{1,n_1}}{\partial Z_{1,n_1}}(0)&\cdots &\frac {\partial F_{1,n_1}}{\partial Z_{k,1}}(0)&\cdots \frac {\partial F_{1,n_1}}{\partial Z_{k,n_k}}(0)\\
 \vdots&\cdots &\vdots &\cdots &\vdots\\
 \frac {\partial F_{k,1}}{\partial Z_{1,1}}(0)&\cdots \frac {\partial F_{k,1}}{\partial Z_{1,n_1}}(0)&\cdots &\frac {\partial F_{k,1}}{\partial Z_{k,1}}(0)&\cdots \frac {\partial F_{k,1}}{\partial Z_{k,n_k}}(0)\\
 \vdots&\cdots &\vdots &\cdots &\vdots\\
 \frac {\partial F_{k,n_k}}{\partial Z_{1,1}}(0)&\cdots \frac {\partial F_{k,n_k}}{\partial Z_{1,n_1}}(0)&\cdots &\frac {\partial F_{k,n_k}}{\partial Z_{k,1}}(0)&\cdots \frac {\partial F_{k,n_k}}{\partial Z_{k,n_k}}(0)
 \end{matrix}
 \right].
 \end{equation*}

 Now, we can prove the following    noncommutative version of
 Cartan's  uniqueness theorem \cite{Ca}, for free holomorphic functions on regular polyballs.

 \begin{theorem} \label{cartan1}
  Let $F:{\bf B_n}(\cH)\to {\bf B_n}(\cH)$ be a free
 holomorphic function such that $F(0)=0$ and $F'(0)=I$. Then
 $$F({\bf X})={\bf X},\qquad   {\bf X}\in {\bf B_n}(\cH).
 $$
\end{theorem}

\begin{proof} Let ${\bf X}=(X_{1,1},\ldots, X_{1,n_1},\ldots, X_{k,1},\ldots, X_{k,n_k})\in {\bf B_n}(\cH)$ and  let
$$
F({\bf X})=(F_{1,1}({\bf X}),\ldots, F_{1,n_1}({\bf X}),\ldots, F_{k,1}({\bf X}),\ldots, F_{k,n_k}({\bf X})),
$$
where
$F_{i,j}$ are free holomorphic functions on the regular polyball ${\bf B_n}(\cH)$, for any  $i\in \{1,\ldots, k\}$ and  $j\in \{1,\ldots, n_i\}$. We will also use the row matrix notation ${\bf X}=[X_{i,j}; \ i,j]$, where the indices $i,j$ are as above.
Since $F(0)=0$ and $F'(0)=I$, we must have
$$
F_{i,j}({\bf X})=X_{i,j}+\sum_{q=2}^\infty \sum\limits_{  {|\alpha_1|+\cdots +|\alpha_k|=q}\atop {\alpha_s\in \FF_{n_s}^+}}a_{\alpha_1,\ldots, \alpha_k}^{(ij)} X_{1,\alpha_1}\cdots X_{k,\alpha_k}, \qquad a_{\alpha_1,\ldots, \alpha_k}^{(ij)}\in \CC,
$$
for any  $i\in \{1,\ldots, k\}$ and  $j\in \{1,\ldots, n_i\}$.
Assume that at least one of the coefficients $a_{\alpha_1,\ldots, \alpha_k}^{(ij)}$ is different from zero. Let $m\geq 2$ be the smallest natural number  such that there exist $i_0\in \{1,\ldots, k\}$, $j_0\in\{1,\ldots, n_{i_0}\}$, and $\alpha^0_s\in \FF_{n_s}^+$ such that $|\alpha^0_1|+\cdots +|\alpha_k^0|=m$ and $a_{\alpha_1^0,\ldots, \alpha_k^0}^{(i_0j_0)}\neq 0$. Then we have
$F_{i,j}({\bf X})=X_{i,j}+\sum_{p=m}^\infty G_p^{(ij)}({\bf X}),
$
where
\begin{equation}
\label{Gp}
G_p^{(ij)}({\bf X}):=
\sum\limits_{  {|\alpha_1|+\cdots +|\alpha_k|=p}\atop {\alpha_s\in \FF_{n_s}^+}}a_{\alpha_1,\ldots, \alpha_k}^{(ij)} X_{1,\alpha_1}\cdots X_{k,\alpha_k}
\end{equation}
for any $p\geq m$, $i\in \{1,\ldots, k\}$, and  $j\in \{1,\ldots, n_i\}$.
Due to Theorem \ref{homo-compo},  $G_p^{(ij)}\circ F$, $p\geq m$,  is a free holomorphic function and
\begin{equation*}
\begin{split}
(G_m^{(ij)}\circ F)({\bf X})&=\sum\limits_{  {|\alpha_1|+\cdots +|\alpha_k|=m}\atop {\alpha_s\in \FF_{n_s}^+}}a_{\alpha_1,\ldots, \alpha_k}^{(ij)} F_{1,\alpha_1}({\bf X})\cdots F_{k,\alpha_k}({\bf X})=G_m^{(ij)}({\bf X})+K_{m+1}^{(ij)}({\bf X}),
\end{split}
\end{equation*}
where $K_{m+1}^{(ij)}$ is a free holomorphic function containing only monomials of degree greater  than or equal to $m+1$. Using now Theorem, we deduce that $F^{[2]}:=F\circ F$ is a free holomorphic function on the regular polyball ${\bf B_n}(\cH)$. Note that
$$\
F({\bf X})=[X_{i,j}:\ i, j]+ \left[\sum_{p=m}^\infty G_p^{(ij)}({\bf X}): \ i,j\right]
$$
and
\begin{equation*}
\begin{split}
F^{[2]}({\bf X})&=[F_{i,j}({\bf X}):\ i, j]+ \left[G_m^{(ij)}(F({\bf X}))+\sum_{p=m+1}^\infty G_p^{(ij)}(F({\bf X})): \ i,j\right]\\
&=\left[X_{i,j}+G_m^{(ij)}({\bf X})+\sum_{p=m+1}^\infty G_p^{(ij)}({\bf X}):\ i,j\right] + \left[G_m^{(ij)}({\bf X})+ \Omega_{m+1}^{(ij)}({\bf X}):\ i,j\right]\\
&=[X_{i,j}:\ i, j]+[2G_m^{(i,j)}({\bf X}): \ i,j]+ [\Gamma_{m+1}^{(ij)}({\bf X}):\ i,j],
\end{split}
\end{equation*}
where $\Omega_{m+1}^{(ij)}$ and $\Gamma_{m+1}^{(ij)}$ are free holomorphic functions containing only monomials of degree greater than or equal to $m+1$.
Continuing this process, we obtain

\begin{equation}\label{Fn}
F^{[n]}({\bf X})=[X_{i,j}:\ i, j]+[nG_m^{(i,j)}({\bf X}): \ i,j]+ [\Lambda_{m+1}^{(ij)}({\bf X}):\ i,j]\qquad n\in \NN,
\end{equation}
where $\Lambda_{m+1}^{(ij)}$ are free holomorphic functions containing only monomials of degree greater than or equal to $m+1$.

 Recall that $\alpha^0_s\in \FF_{n_s}^+$ and  $|\alpha^0_1|+\cdots +|\alpha_k^0|=m$. Consequently,
  if $\beta_i\in \FF_{n_i}^+$ with $|\beta_1|+\cdots + |\beta_k|=p\geq m$, then
  $S_{1,\beta_1}^*\otimes \cdots \otimes S_{k,\beta_k}^* (e^1_{\alpha_1^0}\otimes \cdots \otimes  e^k_{\alpha_k^0})\neq 0$ if and only if $p=m$ and  $\beta_i=\alpha_i^0$ for any $i\in \{1,\ldots, k\}$. In this case, we have
$S_{1,\alpha_1^0}^*\otimes \cdots \otimes S_{k,\alpha_k^0}^* (e^1_{\alpha_1^0}\otimes\cdots \otimes  e^k_{\alpha_k^0})=1$. Hence, and  using relation \eqref{Fn} when ${\bf X}={\bf S}$, we obtain
\begin{equation*}
F^{[n]}(r{\bf S})^*(e^1_{\alpha_1^0}\otimes \cdots \otimes  e^k_{\alpha_k^0})= r\left[{\bf S}_{i,j}(e^1_{\alpha_1^0}\otimes \cdots \otimes  e^k_{\alpha_k^0}):\ i,j\right]
+nr^m\left[G_m^{(ij)}({\bf S})^*(e^1_{\alpha_1^0}\otimes \cdots \otimes  e^k_{\alpha_k^0}):\ i,j\right],
\end{equation*}
where $G_m^{(ij)}$ are homogeneous polynomials of degree $m$ (see relation \eqref{Gp}). Taking into account the latter relation and the fact that
$$
\left\|\left[G_m^{(ij)}({\bf S})^*(e^1_{\alpha_1^0}\otimes \cdots \otimes  e^k_{\alpha_k^0}):\ i,j\right]\right\|\geq \left|a_{\alpha_1^0,\ldots, \alpha_k^0}^{(i_0j_0)}\right|>0,
$$
 we deduce that
$$
nr^m \left|a_{\alpha_1^0,\ldots, \alpha_k^0}^{(i_0j_0)}\right| \leq \|F^{[n]}(r{\bf S})^*(e^1_{\alpha_1^0}\otimes \cdots \otimes  e^k_{\alpha_k^0})\|+\left\| r\left[{\bf S}_{i,j}(e^1_{\alpha_1^0}\otimes \cdots \otimes  e^k_{\alpha_k^0}):\ i,j\right]\right\|
$$
for any $n\in \NN$.  Since $ F^{[n]}(r{\bf S})\in
{\bf B_n}(F^2(H_{n_1})\otimes \cdots\otimes F^2(H_{n_k}))$, taking $n\to \infty$ in the inequality above, we obtain a contradiction. Therefore, we must have $F({\bf X})={\bf X}$. The proof is complete.
\end{proof}

If $L:=[a_{ij}]_{n\times n}$ is a  bounded linear operator on
$\CC^n$, it generates a  function $\boldsymbol\Phi_L: B(\cH)^n\to B(\cH)^n$
 by setting
$$
\boldsymbol\Phi_{L}(X_1,\ldots, X_n):=[X_1,\ldots, X_n]{\bf
L}=\left[\sum_{i=1}^n a_{i1} X_i,\cdots,\sum_{i=1}^n a_{in}
X_i\right]
$$
where ${\bf L}:=[a_{ij}I_\cH]_{n\times n}$. By abuse of notation, we
also write  $\boldsymbol\Phi_L(X)=XL$.

A map $F:{\bf B_n}(\cH)\to {\bf B_n}(\cH)$ is called
free biholomorphic  if $F$  is  free homolorphic, one-to-one and
onto, and  has  free holomorphic inverse. The automorphism group of
${\bf B_n}(\cH)$, denoted by $Aut({\bf B_n}(\cH))$, consists
of all free biholomorphic functions  of ${\bf B_n}(\cH)$. It is
clear that $Aut({\bf B_n}(\cH))$ is  a group with respect to the
composition of free holomorphic functions.

In what follows, we characterize the free biholomorphic functions
with $F(0)=0$.

\begin{theorem}
\label{Cartan2} Let $F:{\bf B_n}(\cH)\to
{\bf B_n}(\cH)$  be a free biholomorphic function with
$F(0)=0$. Then there is an invertible bounded linear operator $L$ on
$\CC^{n_1+\cdots n_k}$ such that
$$ F({\bf X})=\boldsymbol\Phi_L ({\bf X}), \qquad {\bf X}\in {\bf B_n}(\cH).
$$
\end{theorem}

\begin{proof}
Consider the set $\Lambda_{\bf n}:=\{(i,j):\ i\in\{1,\ldots, k\}, j\in \{1,\ldots, n_i\}\}$ with the lexicographic order. Since $F({\bf X})=0$, we have
$F({\bf X})=[F_{s,t}({\bf X}):\ (s,t)\in \Lambda_{\bf n}]$
with
\begin{equation}
\label{Fst}
F_{s,t}({\bf X})=\sum_{(i,j)\in \Lambda_{\bf n}} a_{(s,t)}^{(i,j)} X_{i,j} +\Psi_{s,t}({\bf X}),
\end{equation}
where $\Psi_{s,t}$ is a free holomorphic function which contains only monomials of degree $\geq 2$.
Therefore, we have
\begin{equation}
\label{Fst2}
\Psi_{s,t}({\bf X})=\sum_{m=2}^\infty \sum\limits_{  {|\alpha_1|+\cdots +|\alpha_k|=m}\atop {\alpha_s\in \FF_{n_s}^+}}c_{\alpha_1,\ldots, \alpha_k}^{(s,t)} X_{1,\alpha_1}\cdots X_{k,\alpha_k}
\end{equation}
for some coefficients $c_{\alpha_1,\ldots, \alpha_k}^{(s,t)}\in \CC$.
Consider the matrix $L:=\left[a_{(s,t)}^{(i,j)}\right]_{((i,j),(s,t))\in \Lambda_{\bf n}\times \Lambda_{\bf n}}$ and note that
$$
F({\bf X})=[X_{i,j}:\ (i,j)\in \Lambda_{\bf n}]L+ [\Psi_{s,t}({\bf X}):\ (s,t)\in \Lambda_{\bf n}].
$$
Since $F$ is free biholomorphic function with
$F(0)=0$, its inverse $G:{\bf B_n}(\cH)\to
{\bf B_n}(\cH)$  is also a free holomorphic function with
$G(0)=0$.  As above, one can see that $G$ must have  a representation of the form
$$
G({\bf X})=[X_{s,t}:\ (s,t)\in \Lambda_{\bf n}]M+ [\Gamma_{i,j}({\bf X}):\ (i,j)\in \Lambda_{\bf n}],
$$
where  $M:=\left[b_{(i,j)}^{(s,t)}\right]_{((s,t),(i,j))\in \Lambda_{\bf n}\times \Lambda_{\bf n}}$
is a square matrix with complex coefficients and $\Gamma_{i,j}$ is  a free holomorphic function which contains only monomials of degree $\geq 2$. Now, one can easily see that
\begin{equation*}
\begin{split}
(G\circ F)({\bf X})&=[F_{s,t}({\bf X}):\ (s,t)\in \Lambda_{\bf n}]M+ [\Gamma_{i,j}(F({\bf X})):\ (i,j)\in \Lambda_{\bf n}]\\
&=[X_{i,j}:\ (i,j)\in \Lambda_{\bf n}]LM+ [\Psi_{s,t}({\bf X}):\ (s,t)\in \Lambda_{\bf n}]M+ [\Gamma_{i,j}(F({\bf X})):\ (i,j)\in \Lambda_{\bf n}]\\
&=[X_{i,j}:\ (i,j)\in \Lambda_{\bf n}]LM + [Q_{i,j}({\bf X}):\ (i,j)\in \Lambda_{\bf n}],
\end{split}
\end{equation*}
where  each $Q_{i,j}$  is a free holomorphic function which contains only monomials of degree $\geq 2$.
Since $(G\circ F)({\bf X})={\bf X}$ and
due to  the uniqueness of the representation of  free holomorphic functions, we deduce that $Q_{i,j}=0$ for any $ (i,j)\in \Lambda_{\bf n}$ and $LM=I_{n_1+\cdots +n_k}$. In a similar manner, one can prove that $LM=I_{n_1+\cdots +n_k}$. Therefore, $L$ is an invertible operator.

Since ${\bf B_n}(\cH)$ is a noncommutative Reinhardt domain (see Proposition \ref{reg-poly}), for each $\theta\in \RR$, the map
${\bf X}\mapsto e^{-i\theta}F(e^{i\theta}{\bf X})$ is a free holomorphic function on the regular polyball ${\bf B_n}(\cH)$. Consequently, Theorem \ref{homo-compo} implies that
$$H({\bf X}):=G(e^{-i\theta}F(e^{i\theta}{\bf X})),\qquad {\bf X}\in {\bf B_n}(\cH).
$$
is a free holomorphic function with $H(0)=0$ and
$$
H({\bf X})=[X_{i,j}:\ (i,j)\in \Lambda_{\bf n}]LM + [P_{i,j}({\bf X}):\ (i,j)\in \Lambda_{\bf n}],
$$
where  each $P_{i,j}$  is a free holomorphic function which contains only monomials of degree $\geq 2$.
Since $LM=I_{n_1+\cdots +n_k}$, we can apply Theorem \ref{cartan1} and deduce that $H({\bf X})={\bf X}$.
Due to the definition of $H$ and using the fact that $F\circ G=id$, we obtain
$
e^{i\theta}F({\bf X})=F(e^{i\theta}{\bf X})$ for  any ${\bf X}\in {\bf B_n}(\cH),
$
and $\theta \in \RR$. Using  relations \eqref{Fst}, \eqref{Fst2} and due to the uniqueness of the coefficients in the representation of free holomorphic functions, we deduce that
$$
c_{\alpha_1,\ldots, \alpha_k}^{(s,t)} e^{i\theta(|\alpha_1|+\cdots +|\alpha_k|)}=e^{i\theta}c_{\alpha_1,\ldots, \alpha_k}^{(s,t)},\qquad \theta\in \RR,
$$
for any $\alpha_i\in \FF_{n_i}^+$ with $|\alpha_1|+\cdots +|\alpha_k|\geq 2$, and $(s,t)\in \Lambda_{\bf n}$.
Hence, $c_{\alpha_1,\ldots, \alpha_k}^{(s,t)}=0$ and, therefore, $\Psi_{s,t}=0$. Now, relation \eqref{Fst} implies $F({\bf X})={\bf X}L$, and he proof is complete.
\end{proof}

\begin{theorem} \label{Cartan3}
 Let  ${\bf n}=(n_1,\ldots, n_k)\in \NN^k$ and  let $F:{\bf B_n}(\cH)\to
{\bf B_n}(\cH)$  be a free biholomorphic function with
$F(0)=0$. Then there  are unitary operators $U_i\in B(\CC^{n_i})$, $i\in \{1,\ldots, k\}$, and a permutation $\sigma\in \cS_k$ with the property that $n_{\sigma^{-1}(i)}=n_i$ for $i\in \{1,\ldots, k\}$   such that
$$ (p_{\sigma^{-1}}\circ F)({\bf X})= [X_1U_1, \ldots,X_kU_k] \qquad    {\bf X}=(X_1,\ldots, X_k)\in {\bf B_n}(\cH).
$$
Moreover, the converse is also true.
\end{theorem}

\begin{proof} According to Theorem \ref{Cartan2}, there  is an invertible bounded linear operator $L$ on
$\CC^{n_1+\cdots n_k}$ such that
$$ F({\bf X})=[X_1,\ldots, X_k]{\bf L}, \qquad {\bf X}\in {\bf B_n}(\cH).
$$
Since $F\in Aut({\bf B_n})$, its scalar representation
$f(\lambda_1,\ldots, \lambda_k):=[\lambda_1,\ldots, \lambda_k]{\bf L}$
is an automorphism of the scalar polyball $(\CC^{n_1})_1\times \cdots \times  (\CC^{n_k})_1$. Due to the classical result (see \cite{Ru1}, \cite{L},\cite{T}), there is a permutation $\sigma\in \cS_k$ such that $n_{{\sigma^{-1}}(i)}=n_i$ for $i\in \{1,\ldots, k\}$, such that
$$
(p_{\sigma^{-1}}\circ f)(\lambda_1,\ldots, \lambda_k)=(g_1(\lambda_1),\ldots, g_k(\lambda_k)),\qquad (\lambda_1,\ldots, \lambda_k)\in (\CC^{n_1})_1\times \cdots \times  (\CC^{n_k})_1,
$$
where $g_i\in Aut((\CC^{n_i})_1)$ with $g_i(0)=0$ for any $i\in \{1,\ldots, k\}$.  According to \cite{Ru2}, each $g_i\in Aut((\CC^{n_i})_1)$ with $g_i(0)=0$ has the form   $g_i(\lambda_i)=\lambda_i U_i$, where  $U_i\in B(\CC^{n_i})$  is a unitary operator. Consequently, we obtain
$$
(p_{\sigma^{-1}}\circ f)(\lambda_1,\ldots, \lambda_k)= [\lambda_1,\ldots, \lambda_k] {\bf U},\qquad (\lambda_1,\ldots, \lambda_k)\in (\CC^{n_1})_1\times \cdots \times  (\CC^{n_k})_1,
$$
where the unitary operator ${\bf U}\in B(C^{n_1+\cdots +n_k})$ is the direct sum ${\bf U}=U_1\oplus \cdots \oplus U_k$. Hence, we deduce that
$(p_{\sigma^{-1}}\circ F)(\lambda_1,\ldots, \lambda_k)= [\lambda_1,\ldots, \lambda_k] {\bf U}$, which, due to the linearity of each component of $F$, implies
$$(p_{\sigma^{-1}}\circ F)(X_1,\ldots, X_k)= [X_1,\ldots, X_k] {\bf U}
$$
for any $(X_1,\ldots, X_k)\in {\bf B_n}(\cH)$.

To prove the converse, let $U_i\in B(\CC^{n_i})$, $i\in \{1,\ldots, k\}$, be unitary operators.  Note that the map $g_i$ defined by
$
g_i(X_i):=X_iU_i$, $X_i\in [B(\cH)^{n_i}]_1$,  is a free   holomorphic automorphism of the noncommutative ball
$[B(\cH)^{n_i}]_1$. Hence, and using Proposition  \ref{CSH}, we deduce that
$g:=(g_1,\ldots, g_k)$ and $p_{\sigma}$ are  holomorphic automorphisms of the regular polyball ${\bf B_n}$. Consequently, $F:=p_{\sigma}\circ g\in Aut({\bf B_n})$ with $F(0)=0$. The proof is complete.
\end{proof}

Under the conditions of Theorem \ref{Cartan2}, we consider the unitary operator ${\bf U}\in B(C^{n_1+\cdots +n_k})$ defined as the direct sum ${\bf U}=U_1\oplus \cdots \oplus U_k$ and let $\boldsymbol\Phi_{\bf U}:{\bf B_n}(\cH)\to
{\bf B_n}(\cH)$  be the  free biholomorphic function defined by
$\boldsymbol\Phi_{\bf U}({\bf X}):={\bf X} {\bf U}$.  Then, we have $F=p_{\sigma}\circ \boldsymbol\Phi_{\bf U}$.

 \begin{theorem} Let $F:{\bf B_n}(\cH)\to {\bf B_n}(\cH)$ be a  free holomorphic function such that
 $F'(0)$ is a unitary operator on $\CC^{n_1+\cdots +{n_k}}$. Then  $F$ is a free holomorphic automorphism of ${\bf B_n}$ and
 $$
 F({\bf X})={\bf X} [F'(0)]^t,\qquad {\bf X}\in {\bf B_n}(\cH),
 $$
 where ${}^\tau$ denotes the transpose.
 \end{theorem}
  \begin{proof} Assume that $F$ has the representation
  $$
 F({\bf X}):=A_{(0)}+\sum\limits_{q=1}^\infty\sum\limits_{  {|\alpha_1|+\cdots +|\alpha_k|=q}\atop {\alpha_i\in \FF_{n_i}^+}}A_{(\alpha)}\otimes  { X}_{(\alpha)}, \qquad {\bf X}\in {\bf B_n}(\cH),
 $$
where $A_{(\alpha)}\in {\bf P_n}(\CC)$ is written as a row operator with entries in $\CC$. Note that
$$
F'(0)=\left[A_{(\alpha)}^\tau:\  |\alpha_1|+\cdots +|\alpha_k|=1, \alpha_i\in \FF_{n_i}^+\right].
$$
 Taking into account that $F'(0)$ is a co-isometry, we deduce that
$\sum\limits_{  {|\alpha_1|+\cdots +|\alpha_k|=1}\atop {\alpha_i\in \FF_{n_i}^+}}A_{(\alpha)}^*A_{(\alpha)}=I.
  $
  Since $F$ is a free holomorphic function with $\|F\|_\infty=1$, we can apply Proposition \ref{coef-ine}. Consequently, we have
$$
\sum\limits_{  {|\alpha_1|+\cdots +|\alpha_k|=1}\atop {\alpha_i\in \FF_{n_i}^+}}A_{(\alpha)}^*A_{(\alpha)}\leq I-F(0)^* F(0),
$$
which implies $F(0)=0$. Therefore,  since $[F'(0)]^\tau=\left[\begin{matrix}
A_{(\alpha)}\\
\vdots\\
|\alpha_1|+\cdots +|\alpha_k|=1\end{matrix}\right]$, we have
\begin{equation}
\label{FXX}
F({\bf X})={\bf X}[F'(0)]^\tau
+\sum\limits_{q=2}^\infty\sum\limits_{  {|\alpha_1|+\cdots +|\alpha_k|=q}\atop {\alpha_i\in \FF_{n_i}^+}}A_{(\alpha)}\otimes  { X}_{(\alpha)}.
\end{equation}
On the other hand, since $F'(0)$ is an isometry, we have $F'(0)^\tau [F'(0)^\tau]^*=I$. Multiplying relation \eqref{FXX} to the right by $([F'(0)]^\tau)^*$, we obtain
$$
H({\bf X}):=F({\bf X})([F'(0)]^\tau)^*={\bf X}+[G_{i,j}({\bf X}):\ (i,j)\in \Lambda_{\bf n}],
$$
where $\Lambda_{\bf n}:=\{(i,j):\ i\in\{1,\ldots, k\}, j\in \{1,\ldots, n_i\}\}$ and each $G_{i,j}$ is a free holomorphic function containing only monomials of degree $\geq 2$.  Since $H$ is a free holomorphic function on ${\bf B_n}(\cH)$ with $H(0)=0$ and $H'(0)=I_{n_1+\cdots + n_k}$, Theorem
\ref{cartan1} implies $H({\bf X})={\bf X}$.  Consequently, we have
$F({\bf X})([F'(0)]^\tau)^*={\bf X}$. Multiplying this relation to the right by $[F'(0)]^\tau$ and taking into account that $F'(0)$ is a co-isometry, we
deduce that
 $F({\bf X})={\bf X} [F'(0)]^\tau$ for any $ {\bf X}\in {\bf B_n}(\cH)$.
 This completes the proof.
\end{proof}

In \cite{Po-automorphism}, the  theory of noncommutative characteristic functions for row
contractions (see \cite{Po-charact}) was  used
to find all the involutive free holomorphic automorphisms of
$[B(\cH)^n]_1$. They  turned out to be of the form
\begin{equation*}
 \Psi_\lambda(Y_1,\ldots, Y_n)=- \Theta_\lambda(Y_1,\ldots, Y_n):={
\lambda}-\Delta_{ \lambda}\left(I_\cK-\sum_{i=1}^n \bar{{
\lambda}}_i Y_i\right)^{-1} [Y_1\cdots Y_n] \Delta_{{\lambda}^*},
\end{equation*}
for some $\lambda=(\lambda_1,\ldots, \lambda_n)\in \BB_n$, where
$\Theta_\lambda$ is the characteristic function  of the row
contraction $\lambda$, and $\Delta_{ \lambda}$,
$\Delta_{{\lambda}^*}$ are the defect operators defined by $\Delta_\lambda=(1-\|\lambda\|_2^2)^{1/2}$ and
$\Delta_{\lambda^*}=(I_{\CC^n}-\lambda^*\lambda)^{1/2}$.  Moreover,  we
determined the group $\text{\rm Aut}([B(\cH)^n]_1)$  of all the free
holomorphic automorphisms of the noncommutative ball $[B(\cH)^n]_1$
and showed that if $\Psi\in \text{\rm Aut}([B(\cH)^n]_1)$ and
$\lambda:=\Psi^{-1}(0)$, then there is a unitary operator $U$ on
$\CC^n$ such that
$$
\Psi=\Psi_U\circ \Psi_\lambda,
$$
where $\Psi_U (Y):= YU $ for any   $Y\in  [B(\cH)^n]_1$.
Let  $\lambda:=(\lambda_1,\ldots, \lambda_n)\in
\BB_n\backslash \{0\}$ and let $\gamma:=\frac{1}{\|\lambda\|_2}$.
Then $\Psi_\lambda:=-\Theta_\lambda$ is a free holomorphic function
on $[B(\cH)^n]_\gamma$ which has the following properties:
\begin{enumerate}
\item[(i)]
$\Psi_\lambda (0)=\lambda$ and $\Psi_\lambda(\lambda)=0$;
\item[(ii)] The identity
\begin{equation*}
I_{\cH}-\Psi_\lambda(X)\Psi_\lambda(X)^*= \Delta_{\lambda}(I- X
\lambda^*)^{-1}(I-  X   X^*)(I-  \lambda  X^*)^{-1} \Delta_{\lambda}
\end{equation*}

holds  for all  $X\in [B(\cH)^n]_\gamma$;

\item[(iii)] $\Psi_\lambda$ is an involution, i.e., $\Psi_\lambda(\Psi_\lambda(X))=X$
for any $X\in [B(\cH)^n]_\gamma$;
\item[(iv)] $\Psi_\lambda$ is a free holomorphic automorphism of the
noncommutative unit ball $[B(\cH)^n]_1$;
\item[(v)] $\Psi_\lambda$ is a homeomorphism of $[B(\cH)^n]_1^-$
onto $[B(\cH)^n]_1^-$.
\end{enumerate}

Now, we can prove a structure theorem for  holomorphic automorphisms
of regular polyballs.

\begin{theorem} \label{structure}
 Let  ${\bf n}=(n_1,\ldots, n_k)\in \NN^k$ and let $\boldsymbol\Psi\in \text{\rm Aut}({\bf B_n}(\cH))$. If $\boldsymbol\lambda=(\lambda_1,\ldots, \lambda_k)=\boldsymbol\Psi^{-1}(0)$, then there are  unique unitary operators $U_i\in B(\CC^{n_i})$, $i\in \{1,\ldots, k\}$, and a unique permutation $\sigma\in \cS_k$ with $n_{\sigma(i)}=n_i$  such that
 $$
 \boldsymbol\Psi=p_\sigma\circ \boldsymbol\Phi_{\bf U}\circ \boldsymbol\Psi_{\boldsymbol\lambda},
 $$
 where ${\bf U}:=U_1\oplus\cdots \oplus U_k$ and $\boldsymbol\Psi_{\boldsymbol\lambda}:=(\Psi_{\lambda_1},\ldots, \Psi_{\lambda_k})$.
\end{theorem}
\begin{proof}
Let $\boldsymbol\Psi\in \text{\rm Aut}({\bf B_n}(\cH))$ and let  $\boldsymbol\lambda=(\lambda_1,\ldots, \lambda_k)=\boldsymbol\Psi^{-1}(0)$.
For each $i\in \{1,\ldots, k\}$, $\lambda_i\in (\CC^{n_i})_1$, and $\Psi_{\lambda_i}$ is  a free holomorphic automorphism of the
noncommutative unit ball $[B(\cH)^{n_1}]_1$. Moreover, $\Psi_{\lambda_i}(\Psi_{\lambda_i}(X))=X$
for any $X\in [B(\cH)^{n_i}]_1$,
$\Psi_{\lambda_i} (0)=\lambda_i$. Consequently, using Proposition  \ref{CSH} and Theorem \ref{homo-compo}, we deduce that $\boldsymbol\Psi_{\boldsymbol\lambda}:=(\Psi_{\lambda_1},\ldots, \Psi_{\lambda_k})$ is a holomorphic automorphism of the regular polyball ${\bf B_n}$ with the property that
$$
\boldsymbol\Psi_{\boldsymbol\lambda}(\boldsymbol\Psi_{\boldsymbol\lambda}({\bf X}))={\bf X}, \qquad {\bf X}\in {\bf B_n}(\cH),
$$
and $\boldsymbol\Psi_{\boldsymbol\lambda}(0)=\boldsymbol\lambda$.
Hence, $\boldsymbol\Psi\circ \boldsymbol\Psi_{\boldsymbol\lambda}\in \text{\rm Aut}({\bf B_n}(\cH))$ and  $(\boldsymbol\Psi\circ \boldsymbol\Psi_{\boldsymbol\lambda})(0)=0$.
Applying Theorem \ref{Cartan2}, there  are unitary operators $U_i\in B(\CC^{n_i})$, $i\in \{1,\ldots, k\}$, and a permutation $\sigma\in \cS_k$ with the property that $n_{\sigma^{-1}(i)}=n_i$ for $i\in \{1,\ldots, k\}$   such that
$$ (p_{\sigma^{-1}}\circ (\boldsymbol\Psi\circ \boldsymbol\Psi_{\boldsymbol\lambda}))({\bf X})= [X_1U_1, \ldots,X_kU_k] \qquad    {\bf X}=(X_1,\ldots, X_k)\in {\bf B_n}(\cH).
$$
Hence, taking into account that $\boldsymbol\Psi_{\boldsymbol\lambda}(\boldsymbol\Psi_{\boldsymbol\lambda}({\bf X}))={\bf X}$, we obtain
$
 \boldsymbol\Psi=p_\sigma\circ \boldsymbol\Phi_{\bf U}\circ \boldsymbol\Psi_{\boldsymbol\lambda},
 $
which completes the proof.
\end{proof}

\begin{corollary} Let $F:{\bf B_n}(\cH)\to {\bf B_m}(\cH)$ be a bounded  free holomorphic function and  ${\bf a}\in {\bf B_n}(\CC)$.  Then
$$\|{\bf \Psi}_{F({\bf a})}(F({\bf X}))\|\leq m_{\bf B_n}({\bf \Psi}_{\bf a}({\bf X}))\leq
\|{\bf \Psi}_{\bf a}({\bf X})\|
$$
for any ${\bf X}\in {\bf B_n}(\cH)$, where $m_{\bf B_n}$  is the Minkovski functional.
\end{corollary}
\begin{proof}
Consider the automorphisms ${\bf \Psi_a}\in \text{\rm Aut}({\bf B_n})$  and
${\bf \Psi}_{F({\bf a})}\in \text{\rm Aut}({\bf B_m})$.   Due to Theorem \ref{homo-compo} and using the fact that ${\bf \Psi_a}(0)={\bf a}$ and ${\bf \Psi}_{F({\bf a})}(F({\bf a}))=0$, we deduce that
$G:={\bf \Psi}_{F({\bf a})}\circ F\circ {\bf \Psi_a}$ is a free holomorphic function from ${\bf B_n}(\cH)$ to ${\bf B_m}(\cH)$, and $G(0)=0$.
Applying Theorem \ref{Schwarz} to $G$, we obtain
$$
\|{\bf \Psi}_{F({\bf a})}\circ F\circ {\bf \Psi_a}({\bf Y})\| \leq m_{\bf B_m}({\bf Y})\leq \|{\bf Y}\|,\qquad {\bf Y}\in {\bf B_m}(\cH).
$$
Setting ${\bf Y}={\bf \Psi_a}({\bf Y})$ and using the fact that ${\bf \Psi_a}\circ {\bf \Psi_a}=id$, we  complete the proof.
\end{proof}

In what follows, we present an analogue of   Poincar\' e  result that  the open unit ball of $\CC^n$ is not biholomorphic  equivalent to the polydisk $\DD^n$, for noncommutative regular polyballs.

\begin{theorem}
\label{Poincare} Let  ${\bf n}=(n_1,\ldots, n_k)\in \NN^k$ and  ${\bf m}=(m_1,\ldots, m_q)\in \NN^q$. Then
$$
Bih({\bf B_n}(\cH),
{\bf B_m}(\cH))\neq \emptyset
$$
if and only if    $k=q$ and there is a permutation $\sigma\in \cS_k$ such that $m_{\sigma(i)}=n_i$ for any $i\in \{1,\ldots, k\}$.  Moreover, any free biholomorphic function $F:{\bf B_n}(\cH)\to
{\bf B_m}(\cH))$ is up to a permutation of $(m_1,\ldots, m_k)$  an automorphism of the noncommutative regular polyball ${\bf B_n}$.
\end{theorem}

\begin{proof} Let $F:{\bf B_n}(\cH)\to
{\bf B_m}(\cH))$  be a free biholomorphic  function. Then its scalar representation
$$f:(\CC^{n_1})_1\times \cdots \times  (\CC^{n_k})_1\to
(\CC^{m_1})_1\times \cdots \times  (\CC^{m_q})_1,
$$
defined by  $f({\bf z}):=F({\bf z})$, ${\bf z}=\{z_{i,j}\}\in {\bf B_n}(\CC)=(\CC^{n_1})_1\times\cdots \times (\CC^{n_k})_1$, is a scalar  biholomorphic function.
 Using Browder's invariance of domain theorem, we deduce that  $n_1+\cdots +n_k=m_1+\cdots + m_q$.  On the other hand, according to the classical result of Ligocka and Tsyganov (which is a generalization of Rudin's characterization of the  holomorphic automorphisms of  the polydisc \cite{Ru1}), we must have $k=q$ and  there is a permutation $\sigma\in \cS_k$ such that $m_{\sigma(i)}=n_i$ for any $i\in \{1,\ldots, k\}$.
Using Proposition   \ref{CSH} and Theorem \ref{homo-compo}, we deduce that
$p_\sigma\circ F\in \text{\rm Aut}({\bf B_{n}})$, which completes the proof.
\end{proof}

 Let $\lambda:=(\lambda_1,\ldots,\lambda_n)\in
\BB_n$, $\lambda\neq 0$,  and let  $\tilde \Theta_\lambda$ be the boundary function of the
characteristic function with respect to the right creation operators $R_1,\ldots, R_n$ on the Fock space $F^2(H_n)$, i.e., $\tilde\Theta_\lambda:=\text{\rm SOT-}\lim_{r\to 1}\Theta_\lambda(rR_1,\ldots, rR_n)$.
 We recall from \cite{Po-automorphism}, the following properties.
\begin{enumerate}
\item[(i)] the map $\Theta_\lambda$ is a free holomorphic function on the open
ball $[B(\cH)^n]_\gamma$, where $\gamma:=\frac{1}{\|\lambda\|_2}$;
\item[(ii)]
$\tilde \Theta_\lambda=\Theta_\lambda(R_1,\ldots, R_n)
 =-{ \lambda}+\Delta_{
\lambda}\left(I_{F^2(H_n)}-\sum_{i=1}^n \bar{{ \lambda}}_i
R_i\right)^{-1} [R_1,\ldots, R_n] \Delta_{{\lambda}^*}$;
\item[(iii)] $\tilde \Theta_\lambda$ is a pure row isometry
 with entries in the noncommutative disc algebra generated by $R_1,\ldots, R_n$ and the identity;

\item[(iv)] $\text{\rm rank}\,(I-\tilde \Theta_\lambda\tilde \Theta_\lambda^*)=1$
and $\tilde \Theta_\lambda$ is unitarily equivalent  to
$[R_1,\ldots, R_n]$.
\end{enumerate}

We define the right creation operators $R_{i,j}$ acting on  the  Fock space $F^2(H_{n_i})$ and the ampliations ${\bf R}_{i,j}$ acting on the tensor  product
$F^2(H_{n_1})\otimes\cdots\otimes F^2(H_{n_k})$.

\begin{theorem} \label{structure2}
 Let ${\bf \Psi}=(\Psi_1,\ldots,\Psi_k)\in \text{\rm Aut}({\bf B_n}(\cH))$, where   ${\bf n}=(n_1,\ldots, n_k)\in \NN^k$, and let $\hat{\bf \Psi}=(\hat{\Psi}_1,\ldots, \hat{\Psi}_k)$ be
the boundary function with respect to the universal model ${\bf S}=\{{\bf S}_{i,j}\}$. The following statements hold.
\begin{enumerate}
\item[(i)] ${\bf \Psi}$ is
 a free holomorphic function on the regular polyball $\gamma {\bf B_n}$ for some $\gamma>1$.
\item[(ii)] The boundary function $\hat{\bf \Psi}$ with respect to ${\bf S}$ is a pure element in the polyball $ {\bf B_n}(\otimes_{i=1}^k F^2(H_{n_i}))^-$ and $\hat{\bf \Psi}:=\lim_{r\to 1}\hat{\bf \Psi}(r{\bf S})=\hat{\bf \Psi}({\bf S})$. Each $\hat{\Psi}_i=(\hat{\Psi}_{i,1},\ldots, \hat{\Psi}_{i,n_i})$  is an  isometry with entries in the noncommutative disk algebra generated by ${\bf S}_{i,1},\ldots, {\bf S}_{i,n_i}$ and the identity.

\item[(iii)]${\bf \Psi}$ is a homeomorphism of ${\bf B_n}(\cH)^-$ onto ${\bf B_n}(\cH)^-$.
    \item[(iv)]  If  $\Psi\in Aut({\bf B_n}(\cH))$ and $\boldsymbol\lambda=(\lambda_1,\ldots, \lambda_k)=\boldsymbol\Psi^{-1}(0)$, then  the identity
 $$
 {\bf \Delta}_{\boldsymbol\Psi({\bf X})}(I)=\boldsymbol\Delta_{\boldsymbol\lambda}\left[\prod_{i=1}^k \left(I_\cH-\sum_{j=1}^k \bar{\lambda}_{i,j} X_{i,j}\right)^{-1}\right]{\bf \Delta_X}(I)
 \left[\prod_{i=1}^k \left(I_\cH-\sum_{j=1}^k {\lambda}_{i,j} X_{i,j}^*\right)^{-1}\right]
 $$
holds for any ${\bf X}=\{X_{i,j}\}\in {\bf B_n}(\cH)^-$, where  $\boldsymbol\Delta_{\boldsymbol\lambda}=\prod_{i=1}^k(1-\|\lambda_i\|^2_2)$.
\item[(v)] The defect of   the boundary function of ${\bf \Psi}$ with respect to the universal model  ${\bf R}=\{{\bf R}_{i,j}\}$ satisfies the relation
    $$
    {\bf \Delta_{\Psi(R)}}(I)={\bf K}_{{\bf \Psi}^{-1}(0)}{\bf K}_{{\bf \Psi}^{-1}(0)}^*,
    $$
    where ${\bf K}_{{\bf \Psi}^{-1}(0)}$ is the noncommutative Berezin kernel at ${\bf \Psi}^{-1}(0)\in {\bf B_n}(\CC)$.
\item[(vi)] $\text{\rm rank}\ {\bf \Delta_{\hat\Psi}}=1$ and $\hat{\bf \Psi}$ is unitarily equivalent to the universal model ${\bf S}$.

 \end{enumerate}

\end{theorem}
\begin{proof}  According to Theorem \ref{structure}, if $\boldsymbol\Psi\in \text{\rm Aut}({\bf B_n}(\cH))$ and  $\boldsymbol\lambda=(\lambda_1,\ldots, \lambda_k)=\boldsymbol\Psi^{-1}(0)$, then there are  unique unitary operators $U_i\in B(\CC^{n_i})$, $i\in \{1,\ldots, k\}$, and a unique permutation $\sigma\in \cS_k$ with $n_{\sigma(i)}=n_i$  such that
 $$
 \boldsymbol\Psi=p_\sigma\circ \boldsymbol\Phi_{\bf U}\circ \boldsymbol\Psi_{\boldsymbol\lambda},
 $$
 where ${\bf U}:=U_1\oplus\cdots \oplus U_k$ and $\boldsymbol\Psi_{\boldsymbol\lambda}:=(\Psi_{\lambda_1},\ldots, \Psi_{\lambda_k})$. Since $\Psi_{\lambda_i}:=-\Theta_{\lambda_i}$ is a free holomorphic function on the open ball $[B(\cH)^{n_i}]_{\gamma_i}$, where
 $\gamma_i:=\frac{1}{\|\lambda_i\|_2}$ if $\lambda_i\neq 0$ and $\gamma_i=\infty$, otherwise, Poposition \ref{CSH} part (iii) implies that
 $\boldsymbol\Psi_{\boldsymbol\lambda}$ is a free holomorphic function on the regular polyball $\gamma{\bf B_n}$ for $\gamma:=\min\{\gamma_i:\ i\in \{1,\ldots, k\}\}$. Using Theorem \ref{homo-compo} and Proposition \ref{CSH}, one can complete the proof of item (i).

 The first part of item (ii) follows from (i) and the continuity of the ${\bf \Psi}$ on $\gamma {\bf B_n}$. On the other hand, due to the remarks preceding the theorem, we know that
 $\hat\Psi_{\lambda_i}:=\lim_{r\to 1} \Psi_{\lambda_i}({ S}_i)=\Psi_{\lambda_i}({ S}_i)$ is a pure row isometry with entries in the noncommutative disc algebra generated by    ${ S}_{i,1},\ldots, { S}_{i,n_i}$ and the identity,  on the full Fock space $F^2(H_{n_i})$. If $U_i\in B(\CC^{n_i})$ are unitary operators, it is clear that
 the components of the boundary function
  $$\widehat {\boldsymbol\Phi_{\bf U}\circ \boldsymbol\Psi_{\boldsymbol\lambda}}
  =(\Psi_{\lambda_1}({\bf S})U_1,\ldots \Psi_{\lambda_k}({\bf S})U_k)
 $$
 are isometries. On the other hand,
 set $(\xi_{i,1},\ldots, \xi_{i,n_i}):= {\bf S}_iU_i$ and note that each $\xi_{i,j}$ is a linear combination of ${\bf S}_{i,1},\ldots, {\bf S}_{i, n_i}$. Note that
 $
 \sum_{\alpha\in \FF_{n_i}^+, |\alpha|=p}\xi_{i,\alpha}\xi_{i,\alpha}^* (e_\beta^i)=0
 $ for any $\beta\in \FF_{n_i}^+$ and $p>|\beta|$.
 Since $\sum_{\alpha\in \FF_{n_i}^+, |\alpha|=p}\xi_{i,\alpha}\xi_{i,\alpha}^*\leq I$, we deduce that
 $$
 \lim_{p\to\infty} \sum_{\alpha\in \FF_{n_i}^+, |\alpha|=p}\xi_{i,\alpha}\xi_{i,\alpha}^* x=0,\qquad x\in F^2(H_{n_i}),
 $$
 which proves that $\widehat {\boldsymbol\Phi_{\bf U}\circ \boldsymbol\Psi_{\boldsymbol\lambda}}$ is a pure element in $ {\bf B_n}(\otimes_{i=1}^k F^2(H_{n_i}))^-$.
 For any  permutation $\sigma\in \cS_k$ with $n_{\sigma(i)}=n_i$, the boundary function
 $\hat p_\sigma=({\bf S}_{\sigma(1)},\ldots, {\bf S}_{\sigma(k)})$ has the entries pure row isometries. Now, using Lemma \ref{pure}, we deduce that the boundary function of the composition
  $
 \boldsymbol\Psi=p_\sigma\circ \boldsymbol\Phi_{\bf U}\circ \boldsymbol\Psi_{\boldsymbol\lambda}
 $
satisfies the required properties of item (ii).

According to the remarks preceding Theorem \ref{structure}, each $\Psi_\lambda$ is a homeomorphism of $[B(\cH)^{n_i}]_1^-$ and
$\Psi_{\lambda_i}(\Psi_{\lambda_i}(X_i))=X_i$
for any $X_i\in [B(\cH)^{n_i}]_1^-$.   This implies that
$$
\boldsymbol\Psi_{\boldsymbol\lambda}(\boldsymbol\Psi_{\boldsymbol\lambda}({\bf X}))={\bf X}, \qquad {\bf X}\in {\bf B_n}(\cH)^-,
$$
which proves that $\boldsymbol\Psi_{\boldsymbol\lambda}$ is a homeomorphism of ${\bf B_n}(\cH)^-$.
 According to Proposition \ref{CSH}, $\boldsymbol{\Phi_\lambda}$ and $p_\sigma$ are also homeomorphisms of ${\bf B_n}(\cH)^-$. Since, due to Theorem \ref{structure}, each $\boldsymbol\Psi\in \text{\rm Aut}({\bf B_n}(\cH))$ has the representation
 $
 \boldsymbol\Psi=p_\sigma\circ \boldsymbol\Phi_{\bf U}\circ \boldsymbol\Psi_{\boldsymbol\lambda},
 $
we conclude that ${\bf \Psi}$ is a homeomorphism of ${\bf B_n}(\cH)^-$, which proves item (iii).

For each $i\in \{1,\ldots, k\}$, let $S_i=(S_{i,1},\ldots, S_{i,n_i})$ be the $n_i$-tuple of left creation operators on the full Fock space $F^2(H_{n_i})$.
According to the remarks preceding  Theorem \ref{structure}, we have
\begin{equation*}
\begin{split}
&\left(id-\Phi_{\psi_{\lambda_i}(S_i)}\right)(I_{F^2(H_{n_i})})\\
&=(1-\|\lambda_i\|_1^2)\left(I_{F^2(H_{n_i})}-\sum_{j=1}^{n_i} \bar\lambda_{i,j}S_{i,j}\right)^{-1}
\left(id- \Phi_{S_i}\right)(I_{F^2(H_{n_i})})
\left(I_{F^2(H_{n_i})}-\sum_{j=1}^{n_i} \lambda_{i,j}S_{i,j}^*\right)^{-1}.
\end{split}
\end{equation*}
Taking the tensor product of these relations when $i\in \{1,\ldots, k\}$, and using the definition of the universal model ${\bf S}$, we obtain
\begin{equation*}
\begin{split}
\left(id-\Phi_{\psi_{\lambda_1}({\bf S}_i)}\right)&\circ\cdots \circ\left(id-\Phi_{\psi_{\lambda_k}({\bf S}_k)}\right)(I_{\otimes_{i=1}^kF^2(H_{n_i})})\\
&=
\prod_{i=1}^k
(1-\|\lambda_i\|_1^2)\prod_{i=1}^k\left(I_{\otimes_{i=1}^kF^2(H_{n_i})}-\sum_{j=1}^{n_i} \bar\lambda_{i,j}{\bf S}_{i,j}\right)^{-1}\\
&\times
 \left(id- \Phi_{{\bf S}_1}\right)\circ\cdots \circ\left(id- \Phi_{{\bf S}_k}\right)(I_{\otimes_{i=1}^k F^2(H_{n_i})})
\prod_{i=1}^k\left(I_{\otimes_{i=1}^kF^2(H_{n_i})}-\sum_{j=1}^{n_i} \lambda_{i,j}{\bf S}_{i,j}^*\right)^{-1}.
\end{split}
\end{equation*}
Note that both side of the relation above,  as well as the factors involved, are in the noncommutative polyball algebra ${\boldsymbol\cA_n}$. Applying the Berezin transform at any element  ${\bf X}=(X_1,\ldots, X_k)\in {\bf B_n}(\cH)^-$, we obtain
\begin{equation}\label{DePsi}
 {\bf \Delta}_{\boldsymbol{\Psi_\lambda}({\bf X})}(I)=\boldsymbol\Delta_{\boldsymbol\lambda}\left[\prod_{i=1}^k \left(I_\cH-\sum_{j=1}^k \bar{\lambda}_{i,j} X_{i,j}\right)^{-1}\right]{\bf \Delta_X}(I)
 \left[\prod_{i=1}^k \left(I_\cH-\sum_{j=1}^k {\lambda}_{i,j} X_{i,j}^*\right)^{-1}\right],
 \end{equation}
where $\boldsymbol\Delta_{\boldsymbol\lambda}=\prod_{i=1}^k(1-\|\lambda_i\|^2_2)$.
Since each $\boldsymbol\Psi\in \text{\rm Aut}({\bf B_n}(\cH))$ has the representation
 $
 \boldsymbol\Psi=p_\sigma\circ \boldsymbol\Phi_{\bf U}\circ \boldsymbol\Psi_{\boldsymbol\lambda},
 $
one can easily see that $ {\bf \Delta}_{\boldsymbol{\Psi}({\bf X})}(I)= {\bf \Delta}_{\boldsymbol{\Psi_\lambda}({\bf X})}(I)$, which shows that item (iv) holds.

Now, we prove item (v). If $\boldsymbol\lambda=(\lambda_1,\ldots, \lambda_k)=\boldsymbol\Psi^{-1}(0)$, the Berezin kernel
${\bf K}_{\boldsymbol\lambda}:\CC\to \otimes _{i=1}^k F^2(H_{n_i})$ is defined by
$$
{\bf K}_{\boldsymbol\lambda}(1)=\sum_{(\alpha)\in F_{n_1}^+\times \cdots \times \FF_{n_k}^+}
\boldsymbol\Delta_{\boldsymbol\lambda}^{1/2} \bar{\boldsymbol\lambda}_{(\alpha)}\otimes e_{(\alpha)}.
$$
It is easy to see that ${\bf K}_{\boldsymbol\lambda}^*(e_{(\alpha)})
=\boldsymbol\Delta_{\boldsymbol\lambda}^{1/2}\boldsymbol\lambda_{(\alpha)}$
and
$$
{\bf K}_{\boldsymbol\lambda}{\bf K}_{\boldsymbol\lambda}^*={\bf K}_{\boldsymbol\lambda}(\boldsymbol\Delta_{\boldsymbol\lambda}^{1/2}\boldsymbol\lambda_{(\alpha)})
=\boldsymbol\Delta_{\boldsymbol\lambda}^{1/2}\boldsymbol\lambda_{(\alpha)}
\sum_{(\beta)\in  F_{n_1}^+\times \cdots \times \FF_{n_k}^+}
\bar{\boldsymbol\lambda}_{(\beta)}\otimes e_{(\beta)}.
$$
On the other hand, relation \eqref{DePsi} written for the universal model  ${\bf R}=\{{\bf R}_{i,j}\}$ implies
\begin{equation*}
\begin{split}
{\bf \Delta_{\Psi(R)}}(I)(e_{(\alpha)})
&=\boldsymbol\Delta_{\boldsymbol\lambda}\left[\prod_{i=1}^k \left(I_\cH-\sum_{j=1}^k \bar{\lambda}_{i,j} {\bf R}_{i,j}\right)^{-1}\right] P_\CC
 \left[\prod_{i=1}^k \left(I_\cH-\sum_{j=1}^k {\lambda}_{i,j} {\bf R}_{i,j}^*\right)^{-1}\right](e_{(\alpha)})\\
 &=\boldsymbol\Delta_{\boldsymbol\lambda}\left[\prod_{i=1}^k \left(I_\cH-\sum_{j=1}^k \bar{\lambda}_{i,j} {\bf R}_{i,j}\right)^{-1}\right](\boldsymbol\lambda_{(\alpha)})=\boldsymbol\Delta_{\boldsymbol\lambda}^{1/2}\boldsymbol\lambda_{(\alpha)}
\sum_{(\beta)\in  F_{n_1}^+\times \cdots \times \FF_{n_k}^+}
\bar{\boldsymbol\lambda}_{(\beta)}\otimes e_{(\beta)}.
\end{split}
\end{equation*}
 Therefore, item (v) follows. The fact that $\text{\rm rank}\ {\bf \Delta_{\hat\Psi}}=1$  is a simple consequence of item (iv) or (v).
 Since the boundary function $\hat{\bf \Psi}=(\hat{\Psi}_1,\ldots, \hat{\Psi}_k)$, with respect to the universal model ${\bf S}=\{{\bf S}_{i,j}\}$, is a pure element in the polyball $ {\bf B_n}(\otimes_{i=1}^k F^2(H_{n_i}))^-$ and  each $\hat{\Psi}_i=(\hat{\Psi}_{i,1},\ldots, \hat{\Psi}_{i,n_i})$  is an  isometry with entries in the noncommutative disk algebra generated by ${\bf S}_{i,1},\ldots, {\bf S}_{i,n_i}$ and the identity, we deduce that   $\hat{\bf \Psi}=(\hat{\Psi}_{i,1},\ldots, \hat{\Psi}_{i,n_i})$
 is a pure  doubly commuting tuple of isometries with  $\text{\rm rank}\ {\bf \Delta_{\hat\Psi}}=1$. Now, using the Wold type decomposition for nondegenerate $*$-representations of the $C^*$-algebra $C^*({\bf S})$ from \cite{Po-Berezin-poly} (see Corollary 7.3 and its consequences), we conclude that  $\hat{\bf \Psi}$ is unitarily equivalent to the universal model ${\bf S}$. The proof is complete.
 \end{proof}

\begin{theorem} The map
$
\Lambda:{\text\rm Aut}({\bf B_n})\to \text{\rm Aut}((\CC^{n_1})_1\times\cdots \times (\CC^{n_k})_1)
$
defined by
$$\Lambda({\bf \Psi})({\bf z}):=({\bf B_z} [\hat {\bf \Psi}_1],\ldots,{\bf B_z} [\hat {\bf \Psi}_k]) \qquad {\bf z}\in (\CC^{n_1})_1\times\cdots \times (\CC^{n_k})_1,
$$
is a group isomorphism, where $\hat {\bf \Psi}$ is the boundary function of ${\bf \Psi}=(\Psi_1,\ldots, \Psi_k)\in Aut({\bf B_n})$ with respect to the universal model ${\bf S}$ and ${\bf B_z}$ is the noncommutative Berezin transform at ${\bf z}$.

\end{theorem}
\begin{proof}
  Fix ${\bf \Psi}=(\Psi_1,\ldots, \Psi_k)\in \text{\rm Aut}({\bf B_n})$ and  $\boldsymbol\lambda=\{\lambda_{i,j}\}={\bf \Psi}^{-1}(0)\in {\bf B_n}(\CC)=(\CC^{n_1})_1\times\cdots \times (\CC^{n_k})_1$.
 Then, due to Theorem \ref{structure},   there are  unique unitary operators $U_i\in B(\CC^{n_i})$, $i\in \{1,\ldots, k\}$, and a unique permutation $\sigma\in \cS_k$ with $n_{\sigma(i)}=n_i$  such that
 \begin{equation}
 \label{PPP}
 \boldsymbol\Psi=p_\sigma\circ \boldsymbol\Phi_{\bf U}\circ \boldsymbol\Psi_{\boldsymbol\lambda},
 \end{equation}
 where ${\bf U}:=U_1\oplus\cdots \oplus U_k$. According  to  Theorem \ref{structure2}, Each $\hat{\Psi}_i=(\hat{\Psi}_{i,1},\ldots, \hat{\Psi}_{i,n_i})$  is a pure row isometry with entries in the noncommutative disk algebra generated by ${\bf S}_{i,1},\ldots, {\bf S}_{i,n_i}$ and the identity.
  Note
that if ${\bf z}=\{z_{i,j}\}\in {\bf B_n}(\CC)=(\CC^{n_1})_1\times\cdots \times (\CC^{n_k})_1$, then the Berezin kernel
${\bf K_z}:\CC\to F^2(H_{n_1})\otimes \cdots \otimes F^2(H_{n_k})$ is an isometry and $z_{i, j}={\bf K_z}^* {\bf S}_{i,j}{\bf  K_z}$ for
$i\in \{1,\ldots, k\}$ and $j\in \{1,\ldots, n_i\}$. Hence, using the continuity of the noncommutative
Berezin transform in the operator norm topology and relation \eqref{PPP}, we deduce that
\begin{equation*}
\begin{split}
[\Lambda({\bf \Psi})]({\bf z}):&= ({\bf B_z} [\hat {\bf \Psi}_1],\ldots,{\bf B_z} [\hat {\bf \Psi}_k]) = (p_\sigma\circ \boldsymbol\Phi_{\bf U}\circ \boldsymbol\Psi_{\boldsymbol\lambda})({\bf z})
\end{split}
\end{equation*}
for any ${\bf z}\in {\bf B_n}(\CC)$.  Due to \cite{Ru1}, \cite{L}, \cite{T}, each automorphism of the scalar polyball $(\CC^{n_1})_1\times\cdots \times (\CC^{n_k})_1$  has the form  ${\bf z}\mapsto  (p_\sigma\circ \boldsymbol\Phi_{\bf U}\circ \boldsymbol\Psi_{\boldsymbol\lambda})({\bf z})$.
 Therefore, $\Lambda({\bf \Psi})\in \text{\rm Aut}({\bf B_n}(\CC))$, which proves the surjectivity of $\Lambda$. Moreover, we
have $[\Lambda({\bf \Psi})]({\bf z})={\bf \Psi}({\bf z})$, ${\bf z}\in {\bf B_n}(\CC)$, which clearly
implies that $\Lambda$ is a homomorphism.  To prove injectivity of
$\Lambda$,  assume that $\Lambda({\bf \Psi})=\text{\rm
id}$, where ${\bf \Psi}=p_\sigma\circ \boldsymbol\Phi_{\bf U}\circ \boldsymbol\Psi_{\boldsymbol\lambda}$. Using the calculations
above, we have $p_\sigma\circ \boldsymbol\Phi_{\bf U}\circ \boldsymbol\Psi_{\boldsymbol\lambda}({\bf z})={\bf z}$ for any ${\bf z}\in {\bf B_n}(\CC)$. Hence,
 one can easily deduce that $\boldsymbol\lambda=0$, $U=-I$, and $\sigma=id$,  which implies $\Psi=\text{\rm
id}$. Therefore, $\Lambda$ is a group isomorphism. This completes the proof.
\end{proof}

\smallskip

\section{Automorphisms of   Cuntz-Toeplitz algebras}

In this section,   we determine the group of automorphisms of  the Cuntz-Toeplitz  $C^*$-algebra $C^*({\bf S})$  which leaves invariant the   noncommutative polyball algebra $\boldsymbol\cA_{\bf n}$, and
  the group of unitarily implemented  automorphisms of  the noncommutative polyball   algebra $\boldsymbol\cA_{\bf n}$ (resp. Hardy algebra \,${\bf F}_{\bf n}^\infty)$).   As a consequence, we obtain a concrete description for the group of automorphisms of  the tensor product $\cT_{n_1}\otimes\cdots \otimes\cT_{n_k}$ of Cuntz-Toeplitz algebras which leave invariant the tensor product $\cA_{n_1}\otimes_{min}\cdots \otimes_{min}\cA_{n_k}$ of noncommutative disc algebras, which extends Voiculescu's result when $k=1$.

\begin{proposition}
\label{Berezin-comp}  A free holomorphic function $F:{\bf B_n}(\cH)\to {\bf B_n}(\cH)^-$   has a  continuous extension (also denoted by F) to the closed polyball ${\bf B_n}(\cH)^-$  if and only if  the boundary function $\hat F$ has the entries  in the noncommutative polyball algebra $\boldsymbol\cA_{\bf n}$ and $\hat F\in {\bf B_n}(\otimes_{i=1}^k F^2(H_{n_i}))^-$.
Moreover, the noncommutative Berezin  transform has
the   property that
$$ \boldsymbol\cB_{F({\bf X})}[g]=\boldsymbol\cB_{\bf X}[\boldsymbol\cB_{\hat F}[g] ]
$$
for any
   ${\bf X} \in {\bf B_n}(\cH)^- $ and $g\in C^*({\bf S})$. If, in addition, $\hat F$ is a pure element of the polyball ${\bf B_n}(\otimes_{i=1}^k F^2(H_{n_i}))^-$, then the same relation holds
for any pure element
   ${\bf X} \in {\bf B_n}(\cH)^- $ and $g\in {\bf F^\infty_n}$.
\end{proposition}

\begin{proof}
 The first part of the proposition follows from \cite{Po-Berezin-poly} (Corollary 4.3).
 To prove the second part, let $F=(F_1,\ldots, F_k)$, with $F_i=(F_{i,1},\ldots, F_{i,n_i})$.  Note that the boundary function
 $\hat F=(\hat F_1,\ldots, \hat F_k)$, with $\hat F_i=(\hat F_{i,1},\ldots, \hat F_{i,n_i})$, is an element of the polyball ${\bf B_n}(\otimes_{i=1}^k F^2(H_{n_i}))^-$ and the entries $\hat F_{i,j}:=\lim_{r\to 1} F_{i,j}(r{\bf S})$ are in the noncommutative polyball algebra $\boldsymbol\cA_{\bf n}$.
 Let ${\bf X}\in {\bf B_n}(\cH)^-$ and set ${\bf A}:=(A_1,\ldots, A_k)$, with $A_i=(A_{i,1},\ldots, A_{i,n_i})$, where
 $$
 A_{i,j}:=F_{i,j}({\bf X})=\boldsymbol\cB_{\bf X}[\hat F_{i,j}]:=\lim_{r\to 1}\boldsymbol\cB_{r{\bf X}}[\hat F_{i,j}].
 $$
 We recall that the noncommutative Berezin transform
 $\boldsymbol\cB_{\bf X}:C^*({\bf S})\to B(\cH)$, which is defined by
 $\boldsymbol\cB_{\bf X}(f):=\lim_{r\to 1} \boldsymbol\cB_{r\bf X}[g]$, is a completely contractive  linear map such that
 $$
 \boldsymbol\cB_{\bf X}[fg^*]=\boldsymbol\cB_{\bf X}[f]\boldsymbol\cB_{\bf X}[g]^*,\qquad f,g\in \cA_{\bf n},
 $$
 and the restriction $\boldsymbol\cB_{\bf X}|_{\boldsymbol\cA_{\bf n}}$ is a unital contractive homomorphism from $\boldsymbol\cA_{\bf n}$  to $B(\cH)$.
 Now, note that
 $A_{(\alpha)}=F_{(\alpha)}({\bf X})=\boldsymbol\cB_{\bf X}[\hat F_{(\alpha)}]
 $
 and
 \begin{equation*}
 \begin{split}
 \boldsymbol\cB_{\bf X}[\boldsymbol\cB_{\hat  F}[{\bf S}_{(\alpha)} {\bf S}_{(\beta)}^*]&=\boldsymbol\cB_{\bf X}[\hat F_{(\alpha)} \hat F_{(\beta)}^*] =\boldsymbol\cB_{\bf X}[\hat F_{(\alpha)}] \boldsymbol\cB_{\bf X}[\hat F_{(\beta)}^*]\\
 &=F_{(\alpha)}({\bf X}) F_{(\beta)}({\bf X})^*=A_{(\alpha)} A_{(\beta)}^* =\boldsymbol\cB_{F(\bf X)}[{\bf S}_{(\alpha)} {\bf S}_{(\beta)}^*]
 \end{split}
 \end{equation*}
for any $(\alpha), (\beta)\in \FF_{n_1}^+\times \cdots \times \FF_{n_k}^+$.
Since the linear span of the monomials ${\bf S}_{(\alpha)} {\bf S}_{(\beta)}^*$ is dense in the $C^*$-algebra $C^*({\bf S})$ and the Berezin transform is continuous in the operator norm topology, we deduce that
$\boldsymbol\cB_{\hat F}[g]$ is in $C^*({\bf S})$ for any $g\in C^*({\bf S})$, and
$ \boldsymbol\cB_{F({\bf X})}[g]=\boldsymbol\cB_{\bf X}[\boldsymbol\cB_{\hat F}[g] ]
$
for any
   ${\bf X} \in {\bf B_n}(\cH)^- $ and $g\in C^*({\bf S})$.

    Now, we assume, in addition, that $\hat F$ is a pure element of the polyball ${\bf B_n}(\otimes_{i=1}^k F^2(H_{n_i}))^-$.
   Let $f\in {\bf F^\infty_n}$ have the Fourier representation
      $\sum_{(\alpha)} a_{(\alpha)}  {\bf S}_{(\alpha)}$  and set
$$
f_r({\bf S})= \sum_{q=0}^\infty \sum_{{(\alpha)\in \FF_{n_1}^+\times \cdots \times\FF_{n_k}^+ }\atop {|\alpha_1|+\cdots +|\alpha_k|=q}} r^q a_{(\alpha)}  {\bf S}_{(\alpha)},\qquad r\in [0,1),
$$
where the convergence is in the operator norm topology.
  Since $F({\bf X})$  is pure for any pure element
   ${\bf X} \in {\bf B_n}(\cH)^- $,  we can use the
${\bf F_n^\infty}$-functional calculus for pure elements in the regular polyball to deduce that
\begin{equation*}
\begin{split}
\boldsymbol{\cB}_{F({\bf X})}[f]&= \text{\rm SOT-}\lim_{r\to 1}\boldsymbol{\cB}_{rF({\bf X})}[f] =\text{\rm SOT-}\lim_{r\to 1}\sum_{q=0}^\infty \sum_{{(\alpha)\in \FF_{n_1}^+\times \cdots \times\FF_{n_k}^+ }\atop {|\alpha_1|+\cdots +|\alpha_k|=q}} r^q a_{(\alpha)}  {F}_{(\alpha)}({\bf X}).
\end{split}
\end{equation*}
 On the other hand, since the
boundary function $\hat F=(\hat F_1,\ldots, \hat F_n)$ is a
pure element in the polyball, we have
\begin{equation*}
\begin{split}
\boldsymbol{\cB}_{\hat F}[f]&=
\text{\rm SOT-}\lim_{r\to 1}\boldsymbol{\cB}_{\hat F}[f_r] =\text{\rm SOT-}\lim_{r\to 1}\sum_{q=0}^\infty \sum_{{(\alpha)\in \FF_{n_1}^+\times \cdots \times\FF_{n_k}^+ }\atop {|\alpha_1|+\cdots +|\alpha_k|=q}} r^q a_{(\alpha)}  \hat{F}_{(\alpha)}.
\end{split}
\end{equation*}
Now, since ${\bf X}$ is pure, the Berezin transform $\boldsymbol{\cB}_{\bf X}:{\bf F_n^\infty}\to
B(\cH)$ is SOT-continuous on bounded sets, and it coincides with the
${\bf F_n^\infty}$-functional calculus. Hence,   using the calculations
above and the fact that $\boldsymbol{\cB}_{\bf X}[\hat F_{(\alpha)}]=F_{(\alpha)}({\bf X})$ for any
$(\alpha)\in \FF_{n_1}^+\times \cdots \times\FF_{n_k}^+$,  we deduce that
\begin{equation*}
\begin{split}
\boldsymbol{\cB}_{\bf X}[\boldsymbol{\cB}_{\hat F}[f]]&=\text{\rm SOT-}\lim_{r\to
1}\boldsymbol{\cB}_{\bf X}\left[\sum_{q=0}^\infty \sum_{{(\alpha)\in \FF_{n_1}^+\times \cdots \times\FF_{n_k}^+ }\atop {|\alpha_1|+\cdots +|\alpha_k|=q}} r^q a_{(\alpha)}  \hat{F}_{(\alpha)}\right]\\
&=\text{\rm SOT-}\lim_{r\to 1}\sum_{q=0}^\infty \sum_{{(\alpha)\in \FF_{n_1}^+\times \cdots \times\FF_{n_k}^+ }\atop {|\alpha_1|+\cdots +|\alpha_k|=q}} r^q a_{(\alpha)}  \hat{F}_{(\alpha)}=\boldsymbol{\cB}_{F({\bf X})}[f]
\end{split}
\end{equation*}
for any $f\in {\bf F_n^\infty}$.
 This completes the proof.
 \end{proof}

A  consequence of Proposition \ref{Berezin-comp} is the
following.

\begin{corollary}\label{Berezin-comp2}
 If $\Psi, \Phi\in \text{\rm Aut}({\bf B_n})$, then
$
\boldsymbol\cB_{\widehat{\Psi\circ \Phi}}[g]=(\boldsymbol{\cB}_{\hat \Phi} \boldsymbol {\cB}_{\hat \Psi})[g]
$
for any $g$ in the Cuntz-Toeplitz algebra $ C^*({\bf S})$, or any  $g\in {\bf F^\infty_n}$.
\end{corollary}
\begin{proof} Note that $\widehat{\Psi\circ \Phi}=(\Psi\circ \Phi )({\bf S})=\Psi(\hat \Phi)$. Taking ${\bf X}=\hat \Phi$ in  Proposition
\ref{Berezin-comp}, the result follows.
\end{proof}

\begin{theorem} \label{unitar-shift} Let ${\bf T}=\{T_{i,j}\}\in {\bf B_n}(\cH)^-$  and let ${\bf S}=\{{\bf S}_{i,j}\}\in {\bf B_n}(\otimes_{i=1}^k F^2(H_{n_i}))^-$ be the universal model of the regular polyball. Then  ${\bf T}$ is unitarily equivalent to ${\bf S}\otimes I_\cK$, where $\cK$ is a Hilbert space,  if and only if $\dim \cD_{\bf T}=\dim \cK$, where $\cD_{\bf T}=\overline{{\bf \Delta_T}(I)(\cH)}$,  and  the noncommutative Berezin kernel ${\bf K_T}$ is a unitary operator.
Moreover, in this case,
$$
T_{i,j}={\bf K^*_T} ( {\bf S}_{i,j}\otimes I_{\cD_{\bf T}}){\bf K_T}={\bf K^*_T}(I\otimes W) ({\bf S}_{i,j}\otimes I_\cK) (I\otimes W^*) {\bf K_T}
$$
for any $i\in \{1,\ldots, k\}$ and $j\in \{1,\ldots, n_i\}$,
where $W:\cK\to \cD_{\bf T}$ is a unitary operator.
\end{theorem}
\begin{proof}
 First, we assume that ${\bf T}$ is unitarily equivalent to ${\bf S}\otimes I_\cK:=\{{\bf S}_{i,j}\otimes I_\cK\}$, i.e., there is a unitary operator $U:(\otimes_{i=1}^k F^2(H_{n_i}))\otimes \cK\to \cH$ such that
$$
T_{i,j}=U({\bf S}_{i,j}\otimes I_\cK)U^*, \qquad i\in \{1,\ldots, k\}, j\in \{1,\ldots, n_i\}.
$$
We show that the noncommutative Berezin kernel ${\bf K_T}$ satisfies the relation
$$
{\bf K_T}=(I\otimes W)U^*,
$$
where $W:\cK\to \cD_{\bf T}$ is a unitary operator.   Indeed, note that
${\bf \Delta}_{\bf T}(I)=U{\bf \Delta}_{{\bf S}\otimes I_\cK}(I)U^*=U(P_{\CC\otimes \cK})U^*$.
and
${\bf \Delta}_{\bf T}(I)^{1/2}=U{\bf \Delta}_{{\bf S}\otimes I_\cK}(I)^{1/2}U^*$. Consequently,
we have $\dim \overline{{\bf \Delta}_{\bf T}(I)(\cH)}=\dim \cK$ and
$U(1\otimes \cK)=\overline{{\bf \Delta}_{\bf T}(I)(\cH)}$. Using the definition of the noncommutative Berezin kernel, we deduce that
\begin{equation*}
\begin{split}
{\bf K_{T}}h:&=\sum_{\beta_i\in \FF_{n_i}^+, i=1,\ldots,k}
   e^1_{\beta_1}\otimes \cdots \otimes  e^k_{\beta_k}\otimes {\bf \Delta_{T}}(I)^{1/2} T_{1,\beta_1}^*\cdots T_{k,\beta_k}^*h\\
   &=\sum_{\beta_i\in \FF_{n_i}^+, i=1,\ldots,k}
   e^1_{\beta_1}\otimes \cdots \otimes  e^k_{\beta_k}\otimes U{\bf \Delta}_{{\bf S}\otimes I_\cK}(I)^{1/2}U^* U ({\bf S}_{1,\beta_1}^*\cdots {\bf S}_{k,\beta_k}^*\otimes I_\cK)U^*h\\
   &=\sum_{\beta_i\in \FF_{n_i}^+, i=1,\ldots,k}
   e^1_{\beta_1}\otimes \cdots \otimes  e^k_{\beta_k}\otimes U(P_{\CC\otimes \cK}) ({\bf S}_{1,\beta_1}^*\cdots {\bf S}_{k,\beta_k}^*\otimes I_\cK)U^*h
\end{split}
\end{equation*}
Consider the unitary operator $W:\cK\to \cD_{\bf T}$ defined by $Wy:=U(1\otimes y)$, $y\in \cK$.
For any vector $g=\sum_{\beta_i\in \FF_{n_i}^+, i=1,\ldots,k}
   e^1_{\beta_1}\otimes \cdots \otimes  e^k_{\beta_k}\otimes y_{(\beta)}$ in $(\otimes_{i=1}^k F^2(H_{n_i}))\otimes \cK$, the computations above imply
 \begin{equation*}
\begin{split}
{\bf K_{T}}Ug&=  \sum_{\beta_i\in \FF_{n_i}^+, i=1,\ldots,k}
   e^1_{\beta_1}\otimes \cdots \otimes  e^k_{\beta_k}\otimes U(P_{\CC\otimes \cK}) ({\bf S}_{1,\beta_1}^*\cdots {\bf S}_{k,\beta_k}^*\otimes I_\cK)g\\
   &=\sum_{\beta_i\in \FF_{n_i}^+, i=1,\ldots,k}
   e^1_{\beta_1}\otimes \cdots \otimes  e^k_{\beta_k}\otimes Wy_{(\beta)} =(I\otimes W)g.
\end{split}
\end{equation*}
Hence, ${\bf K_T}=(I\otimes W)U^*$ is a unitary operator. On the other hand, we have
$$
{\bf S}_{i,j}\otimes I_{\cD_{\bf T}}=(I\otimes W)({\bf S}_{i,j}\otimes I_\cK) (I\otimes W^*)
$$
for any $i\in \{1,\ldots, k\}$ and $ j\in \{1,\ldots, n_i\}$. Due to the properties of the noncommutative Berezin kernel, we have ${\bf K_T}T_{i,j}^*= ( {\bf S}^*_{i,j}\otimes I_{\cD_{\bf T}}){\bf K_T}$. Since  ${\bf K_T}$ is a unitary operator, we deduce that
$$
T_{i,j}= {\bf K^*_T}(I\otimes W) ({\bf S}_{i,j}\otimes I_\cK) (I\otimes W^*) {\bf K_T}.
$$

Conversely, if the noncommutative Berezin kernel ${\bf K_T}$ is a unitary operator, then,
 due to the fact that  ${\bf T}$ is a pure element in ${\bf B_n}(\cH)^-$  and
$
T_{i,j}={\bf K^*_T} ( {\bf S}_{i,j}\otimes I_{\cD_{\bf T}}){\bf K_T}
$
for any $i\in \{1,\ldots, k\}$ and $j\in \{1,\ldots, n_i\}$, we complete the proof.
\end{proof}

\begin{corollary} \label{unit-implement} Let ${\bf T}=\{T_{i,j}\}\in {\bf B_n}(\cH)^-$  and let ${\bf S}=\{{\bf S}_{i,j}\}\in {\bf B_n}(\otimes_{i=1}^k F^2(H_{n_i}))^-$ be the universal model of the regular polyball. Then  ${\bf T}$ is unitarily equivalent to ${\bf S}$ if and only if $\dim \cD_{\bf T}=1$ and  the noncommutative Berezin kernel ${\bf K_T}$ is a unitary operator. Moreover, in this case,   the   defect space $\cD_{\bf T}=\CC v_0$ for some vector $v_0\in \cH$ with $\|v_0\|=1$, and
$$
T_{i,j}={\bf K^*_T} ( {\bf S}_{i,j}\otimes I_{\cD_{\bf T}}){\bf K_T}={\bf K^*_T}\tilde W {\bf S}_{i,j} \tilde W^* {\bf K_T}, \qquad i\in \{1,\ldots, k\}, j\in \{1,\ldots, n_i\},
$$
where $\tilde W:\otimes_{i=1}^k F^2(H_{n_i})\to (\otimes_{i=1}^k F^2(H_{n_i}))\otimes \CC v_0$ is the unitary operator defined by
$$
\tilde Wg:= g\otimes v_0,\qquad g\in \otimes_{i=1}^k F^2(H_{n_i}).
$$
\end{corollary}

We denote by $\text{\rm Aut}_{\boldsymbol\cA_{\bf n}}(C^*({\bf S}))$ the group of automorphisms of the Cuntz-Toeplitz algebra  $C^*({\bf S})$ such that $\Gamma(\boldsymbol\cA_{\bf n})= \boldsymbol\cA_{\bf n}$.

\begin{theorem} \label{auto-C*} Any automorphism  $\Gamma$ of the   Cuntz-Toeplitz $C^*$-algebra $C^*({\bf S})$ which leaves invariant the   noncommutative polyball algebra $\boldsymbol\cA_{\bf n}$, i.e. $\Gamma(\boldsymbol\cA_{\bf n})=\boldsymbol\cA_{\bf n}$,  has the form

$$
\Gamma(g):=\boldsymbol\cB_{\hat \Psi}[g]={\bf K}_{\hat {\Psi}}[g\otimes I_{\cD_{\hat {\Psi}}}]{\bf K}_{\hat {\Psi}}^*,\qquad
g\in  C^*({\bf S}),
$$
where  $\Psi\in \text{\rm Aut}({\bf B_n})$ and $\boldsymbol\cB_{\hat \Psi}$ is the noncommutative Berezin transform at   the boundary function $\hat{\Psi}$.  In this case,
the noncommutative Berezin kernel ${\bf K}_{\hat {\Psi}}$ is a unitary operator and $\Gamma$ is a unitary implemented automorphism of
  $C^*({\bf S})$.
Moreover, we have
$$
\text{\rm Aut}_{\boldsymbol\cA_{\bf n}}(C^*({\bf S}))\simeq  \text{\rm Aut}({\bf B_n}).
$$
\end{theorem}

\begin{proof}    Let $\Gamma\in \text{\rm Aut}_{\boldsymbol\cA_{\bf n}}(C^*({\bf S}))$, i.e., $\Psi$ is an automorphism of the Cuntz-Toeplitz algebra  $C^*({\bf S})$ such that $\Gamma(\boldsymbol\cA_{\bf n})= \boldsymbol\cA_{\bf n}$.
 For each $i\in\{1,\ldots,k\}$ and  $j\in\{1,\ldots,n_i\}$, set $\tilde\varphi_{i,j}:=\Gamma({\bf S}_{i,j})$.  If $\tilde\varphi_i:=(\tilde\varphi_{i,1},\ldots, \tilde\varphi_{i,n_i})$, then, using the fact that $\Gamma$ is a morphism of $C^*$-algebras, we have
$$
(id-\Phi_{\tilde\varphi_i})(I)=\Gamma\left(I-\sum_{j=1}^{n_i} {\bf S}_{i,j} {\bf S}_{i,j}^*\right)\geq 0
$$
and, similarly,
$$
(id-\Phi_{\tilde\varphi_1})^{p_1}\circ \cdots \circ (id-\Phi_{\tilde\varphi_k})^{p_k}
=\Gamma\left[(id-\Phi_{{\bf S}_i})^{p_1}\circ \cdots \circ (id-\Phi_{{\bf S}_i})^{p_k} (I)\right]\geq 0
$$
for any $p_i\in \{0,1\}$.
On the other hand, if $s,t\in \{1,\ldots, k\}$, $s\neq t$, then
$$
\tilde\varphi_{s,j}\tilde \varphi_{t,p}=\Gamma({\bf S}_{s,j}{\bf S}_{t,p})=\Gamma({\bf S}_{t,p}{\bf S}_{s,j})=\tilde \varphi_{t,p}\tilde\varphi_{s,j}
$$
for any $j\in \{1,\ldots, n_s\}$ and $p\in \{1,\ldots, n_t\}$. Consequently, the $k$-tuple $\tilde\varphi:=(\tilde\varphi_1, \ldots, \tilde \varphi_k)$ is in the   closed  regular polyball
${\bf B_n}(\otimes_{i=1}^k F^2(H_{n_i}))^-$.
Now, using the noncommutative  Berezin transform, we define
$
\varphi_{i,j}({\bf X}):=\boldsymbol\cB_{\bf X}[\tilde \varphi_{i,j}]$ for  ${\bf X}\in {\bf B_n}(\cH),
$
and remark that, due to Proposition \ref{Berezin-comp}, the mapping  $\varphi: {\bf B_n}(\cH)\to {\bf B_n}(\cH)^-$ defined by
$\varphi({\bf X}):=(\varphi_1({\bf X}), \ldots, \varphi_k({\bf X}))$ and
$\varphi_i({\bf X}):=(\varphi_{i,1}({\bf X}), \ldots, \varphi_{i,n_i}({\bf X}))$  is a free holomorphic function on ${\bf B_n}(\cH)$ which has a continuous extension to the closed polyball ${\bf B_n}(\cH)^-$. This extension is also denoted by $\varphi$.

Now, note that $\Gamma^{-1}(\boldsymbol\cA_{\bf n})= \boldsymbol\cA_{\bf n}$. For each $i\in\{1,\ldots,k\}$ and  $j\in\{1,\ldots,n_i\}$, let $\tilde\xi_{i,j}:=\Gamma^{-1}({\bf S}_{i,j})$. As in the first part of the proof, one can show that
the $k$-tuple $\tilde\xi:=(\tilde\xi_1, \ldots, \tilde \xi_k)$, with $\tilde\xi_i:=(\tilde\xi_{i,1},\ldots, \tilde\xi_{i,n_i})$,  is in the   closed  regular polyball
${\bf B_n}(\otimes_{i=1}^k F^2(H_{n_i}))^-$.
Using the noncommutative  Berezin transform, we define
$
\xi_{i,j}({\bf X}):=\boldsymbol\cB_{\bf X}[\tilde \xi_{i,j}]$ for ${\bf X}\in {\bf B_n}(\cH),
$
and using again  Proposition \ref{Berezin-comp} we deduce that the map $\xi: {\bf B_n}(\cH)\to {\bf B_n}(\cH)^-$ defined by
$\xi({\bf X}):=(\xi_1({\bf X}), \ldots, \xi_k({\bf X}))$ and
$\xi_i({\bf X}):=(\xi_{i,1}({\bf X}), \ldots, \xi_{i,n_i}({\bf X}))$  is a free holomorphic function on ${\bf B_n}(\cH)$ which has a continuous extension to ${\bf B_n}(\cH)^-$, which is also denoted by $\xi$.

According to the results preceding Lemma \ref{pure}, each $\tilde\xi_{i,j}\in \boldsymbol\cA_{\bf n}$ has a unique formal  Fourier type representation  $\sum_{(\alpha)} a_{(\alpha)}^{(i,j)} {\bf S}_{(\alpha)}$ such that
$$
\tilde\xi_{i,j}=\lim_{r\to 1}\sum_{q=0}^\infty \sum_{{(\alpha)\in \FF_{n_1}^+\times \cdots \times\FF_{n_k}^+ }\atop {|\alpha_1|+\cdots +|\alpha_k|=q}} r^q a_{(\alpha)}^{(i,j)} {\bf S}_{(\alpha)},
$$
where the limit is in the operator norm topology. Using the continuity of $\Gamma$ in the norm topology,  we deduce that
\begin{equation*}
\begin{split}
{\bf S}_{i,j}=\Gamma(\tilde\xi_{i,j})&=
\lim_{r\to 1}\sum_{q=0}^\infty \sum_{{(\alpha)\in \FF_{n_1}^+\times \cdots \times\FF_{n_k}^+ }\atop {|\alpha_1|+\cdots +|\alpha_k|=q}} r^q a_{(\alpha)}^{(i,j)} \Gamma({\bf S}_{(\alpha)})=
\lim_{r\to 1}\sum_{q=0}^\infty \sum_{{(\alpha)\in \FF_{n_1}^+\times \cdots \times\FF_{n_k}^+ }\atop {|\alpha_1|+\cdots +|\alpha_k|=q}} r^q a_{(\alpha)}^{(i,j)} \tilde\varphi_{1,\alpha_1}\cdots \tilde\varphi_{k,\alpha_k}.
\end{split}
\end{equation*}
Due to the the continuity in norm of the Berezin transform  $\boldsymbol \cB_{\bf X}$, where ${\bf X}\in {\bf B_n}(\cH)$, we have
\begin{equation*}
\begin{split}
X_{i,j}&=\boldsymbol \cB_{\bf X}[{\bf S}_{i,j}] =
\lim_{r\to 1}\sum_{q=0}^\infty \sum_{{(\alpha)\in \FF_{n_1}^+\times \cdots \times\FF_{n_k}^+ }\atop {|\alpha_1|+\cdots +|\alpha_k|=q}} r^q a_{(\alpha)}^{(i,j)} \boldsymbol \cB_{\bf X}[\tilde\varphi_{1,\alpha_1}\cdots \tilde\varphi_{k,\alpha_k}]\\
&=\lim_{r\to 1}\sum_{q=0}^\infty \sum_{{(\alpha)\in \FF_{n_1}^+\times \cdots \times\FF_{n_k}^+ }\atop {|\alpha_1|+\cdots +|\alpha_k|=q}} r^q a_{(\alpha)}^{(i,j)}  \varphi_{1,\alpha_1}({\bf X})\cdots \varphi_{k,\alpha_k}({\bf X})\\
&=\lim_{r\to 1} \boldsymbol\cB_{\varphi({\bf X})}\left[\sum_{q=0}^\infty \sum_{{(\alpha)\in \FF_{n_1}^+\times \cdots \times\FF_{n_k}^+ }\atop {|\alpha_1|+\cdots +|\alpha_k|=q}} r^q a_{(\alpha)}^{(i,j)} {\bf S}_{(\alpha)}\right] =\boldsymbol\cB_{\varphi({\bf X})}\left[\tilde\xi_{i,j}\right]=\xi_{i,j}(\varphi({\bf X}))
\end{split}
\end{equation*}
for any ${\bf X}\in {\bf B_n}(\cH)$, $i\in \{1,\ldots, k\}$, and $j\in \{1,\ldots, n_i\}$. Consequently, using the continuity in norm of $\varphi$ and $\xi$ on the closed polyball ${\bf B_n}(\cH)^-$, we deduce that
 $(\xi\circ \varphi)({\bf X})={\bf X}$ for any ${\bf X}\in {\bf B_n}(\cH)^-$. Similarly, one can prove that $(\varphi\circ \xi)({\bf X})={\bf X}$ for any ${\bf X}\in {\bf B_n}(\cH)^-$.
Therefore, $\varphi:{\bf B_n}(\cH)^-\to {\bf B_n}(\cH)^-$ is a homeomorphism such that $\varphi$ and $\varphi^{-1}=\xi$  are free holomorphic functions on ${\bf B_n}(\cH)$.

    The next step is to prove that
$\varphi({\bf X})\in {\bf B_n}(\cH)$ for any ${\bf X}\in {\bf B_n}(\cH)$.
Indeed,  due to Corollary \ref{scalar}, the scalar representations of $\varphi$ and $\xi$ are holomorphic functions on
 ${\bf B_n}(\CC)$ with values in the closed polyball ${\bf B_n}(\CC)^-$. Applying the open mapping theorem from complex analysis to the scalar representations of $\varphi$ and $\xi$, we deduce that $\varphi({\bf B_n}(\CC))={\bf B_n}(\CC)$ and $\xi({\bf B_n}(\CC))={\bf B_n}(\CC)$. In particular, for each $i\in \{1,\ldots, k\}$, $\varphi_i:{\bf B_n}(\cH)\to B(\cH)^{n_i}$  is a free holomorphic function with the properties: $\|\varphi_i\|_\infty=1$ and $\|\varphi_i(0)\|<1$. Applying the maximum principle of  Theorem \ref{max-mod}, we conclude that $\|\varphi_i({\bf X})\|<1$ for any ${\bf X}\in {\bf B_n}(\cH)$.  Hence, and using Proposition 1.3 from \cite{Po-Berezin-poly}, we deduce that $\varphi({\bf X})\in {\bf B_n}(\cH)$, which proves our assertion. Similarly, one proves that
 $\xi({\bf X})\in {\bf B_n}(\cH)$ for any ${\bf X}\in {\bf B_n}(\cH)$. Therefore, $\varphi\in \text{\rm Aut} ({\bf B_n})$.

Now, we apply Theorem \ref{structure2} and deduce that  $\text{\rm rank}\ {\bf \Delta_{\tilde \varphi}}=1$ and $\tilde\varphi$ is unitarily equivalent to the universal model ${\bf S}$. Combining this with Theorem \ref{unitar-shift} and Corollary \ref{unit-implement}, we deduce that the noncommutative Berezin transform ${\bf K}_{\tilde \varphi}$ is a unitary operator and
\begin{equation*}
\begin{split}
\Gamma({\bf S}_{i,j})&=\tilde\varphi_{i,j}={\bf K}^*_{\tilde \varphi}({\bf S}_{i,j}\otimes I_{\cD_{\tilde\varphi}}){\bf K}_{\tilde \varphi}
={\bf K}^*_{\tilde \varphi}\tilde W{\bf S}_{i,j}\tilde W^*{\bf K}_{\tilde \varphi},
\end{split}
\end{equation*}
where $\tilde W:\otimes_{i=1}^k F^2(H_{n_i})\to (\otimes_{i=1}^k F^2(H_{n_i}))\otimes \CC v_0$ is the unitary operator defined by
$$
\tilde Wg:= g\otimes v_0,\qquad g\in \otimes_{i=1}^k F^2(H_{n_i}),
$$
where $\cD_{\tilde\varphi}=\CC v_0$ for some vector $v_0\in \otimes_{i=1}^k F^2(H_{n_i})$ with $\|v_0\|=1$.
Hence, we  also have
$$
\Gamma(g)= {\bf K}^*_{\tilde \varphi}(g\otimes I_{\cD_{\tilde\varphi}}){\bf K}_{\tilde \varphi}, \qquad g\in C^*({\bf S}).
$$

Conversely, assume that $\Gamma:C^*({\bf S})\to C^*({\bf S})$ is defined by

\begin{equation}
\label {GaK}
\Gamma(g):=\boldsymbol{\cB}_{\hat {\Psi}}[g]:={\bf K}_{\hat {\Psi}}[g\otimes I_{\cD_{\hat {\Psi}}}]{\bf K}_{\hat {\Psi}}^*,\qquad
g\in  C^*({\bf S}),
\end{equation}
where  $\Psi\in \text{\rm Aut}({\bf B_n})$ and $\boldsymbol{\cB}_{\hat {\Psi}}$ is the Berezin transform at    the boundary function $\hat{\Psi}$.  As above, due to Theorem \ref{structure2}, Theorem \ref{unitar-shift},  and Corollary \ref{unit-implement}, the noncommutative Berezin transform ${\bf K}_{\hat \Psi}$ is a unitary operator and $\Gamma$ is a unitarily implemented automorphism of $C^*({\bf S})$.

Now, note that each $\Gamma\in \text{\rm Aut}_{\boldsymbol\cA_{\bf n}}(C^*({\bf S}))$ corresponds to a unique $\Psi\in \text{\rm Aut}({\bf B_n})$ such that  relation \eqref{GaK} holds.
Indeed, if $\Psi_1, \Psi_2\in \text{\rm Aut}({\bf B_n})$ and $\boldsymbol{\cB}_{\hat {\Psi}_1}=\boldsymbol{\cB}_{\hat {\Psi}_2}$, then $\boldsymbol{\cB}_{\hat {\Psi}_1}[{\bf S}_{i,j}]=\boldsymbol{\cB}_{\hat {\Psi}_2}[{\bf S}_{i,j}]$, which shows that
$(\hat\Psi_1)_{i,j}=(\hat\Psi_2)_{i,j}$. Applying the Berezin transform at ${\bf X}\in {\bf B_n}(\cH)$, we obtain $(\Psi_1)_{i,j}({\bf X})=(\hat\Psi_2)_{i,j}({\bf X})$, which implies $\Psi_1=\Psi_2$.

 Define $\Lambda:\text{\rm Aut}_{\boldsymbol\cA_{\bf n}}(C^*({\bf S}))\to \text{\rm Aut}({\bf B_n})$ by setting $\Lambda(\Gamma)=\Psi$. As we have seen above, $\Lambda$ is a bijection. Let $ \Gamma_1,\Gamma_2\in \text{\rm Aut}_{\boldsymbol\cA_{\bf n}}(C^*({\bf S}))$ and $\Psi_1, \Psi_2\in \text{\rm Aut}({\bf B_n})$ be such that $\Lambda(\Gamma_j)=\Psi_j$, $j=1,2$.
Using Proposition \ref{Berezin-comp} and Corollary \ref{Berezin-comp2}, we deduce that
\begin{equation*}\begin{split}
\Gamma_1 [\Gamma_2(g))]&=\boldsymbol{\cB}_{\hat {\Psi}_1}[\Gamma_2(g)]
=\boldsymbol{\cB}_{\hat {\Psi}_1}[\boldsymbol{\cB}_{\hat {\Psi}_2}[g]]\\
&= \boldsymbol{\cB}_{\Psi_2(\hat {\Psi})}[g]=\boldsymbol{\cB}_{\widehat {\Psi_2\circ \Psi_1}}[g]=\Lambda^{-1}(\Psi_2\circ \Psi_1)(g)
\end{split}
\end{equation*}
for any $g\in C^*({\bf S})$. Hence, we obtain
$
\Lambda(\Gamma_1 \Gamma_2)=\Psi_2\circ \Psi_1=\Lambda(\Gamma_2)\circ \Lambda(\Gamma_1).
$
The proof is complete.
\end{proof}

 In \cite{Po-Berezin-poly}, we proved that the $C^*$-algebra $C^*({\bf S})$ is irreducible  and contains the compact operators in $B(\otimes_{i=1}^k F^2(H_{n_i}))$. Standard results in  representation  theory of $C^*$-algebras   (see e.g. \cite{Arv-book}), imply that any automorphism  of $C^*({\bf S})$ is a unitarily implemented automorphism. Having this result at hand, we remark that an alternative proof of the fact that $\varphi \in \text{\rm Aut}({\bf B_n})$ in Theorem \ref{auto-C*} can be obtained  using some ideas from the proof of Theorem \ref{auto-F} and avoiding the use of the open mapping theorem from complex analysis.

The Cuntz-Toeplitz algebra $\cT_n$ is the unique unital $C^*$-algebra generated  by $n\in \NN$ isometries $s_1,\ldots, s_n$ satisfying relations
$s_i^*s_j=\delta_{ij}1$ and $s_1s_1^*+\cdots +s_ns_n^*<1$. The noncommutative disc algebra $\cA_n$  (see \cite{Po-von}, \cite{Po-disc}) is the unique non-self-adjoint closed algebra generated $s_1,\ldots, s_n$ and the identity.
We also recall  \cite{Cu} that the Cuntz algebra $\cO_n$ is uniquely defined as the $C^*$-algebra generated by $n\geq 2$ isometries satisfying relations
$\sigma_i^*\sigma_j=\delta_{ij}1$ and $\sigma_1\sigma_1^*+\cdots +\sigma_n\sigma_n^*=1$.
In \cite{Cu}, Cuntz showed that if $\boldsymbol\cK\subset \cT_n$ denotes the algebra of compact operators, then
$$
0\to \boldsymbol \cK\to \cT_n\to \cO_n\to 0
$$
is a short exact sequence of $C^*$-algebras. Since the Cuntz algebra $\cO_n$ and the algebra of compact operators $\boldsymbol\cK$ are nuclear nuclear, so is the Cuntz-Toeplitz algebra $\cT_n$. This implies that the tensor products of  $C^*$-algebras $\cT_{n_1}\otimes \cdots \otimes  \cT_{n_k}$ and $\cO_{n_1}\otimes \cdots \otimes  \cO_{n_k}$ have a unique $C^*$-norm. The $C^*$-algebra $C^*({\bf S})$ generated by the universal model ${\bf S}=\{{\bf S}_{i,j}\}$ is $*$-isomorphic to $\cT_{n_1}\otimes \cdots \otimes  \cT_{n_k}$ (see  \cite{Po-poisson}). According to the definition of the min norm on tensor products of operator algebras \cite{Pa-book} and since $\cA_{n_i}$ can be seen as a subalgebra of $\cT_{n_i}$ (see \cite{Po-disc}), we also have  that
$\boldsymbol\cA_{\bf n}\simeq \cA_{n_1}\otimes_{min}\cdots \otimes_{min} \cA_{n_k}$.

Using the short exact sequence obtained by Cuntz \cite{Cu}, one can deduce that there is a a surjective $*$-representation
$\chi:C^*({\bf S})\to \cO_{n_1}\otimes \cdots \otimes  \cO_{n_k}$
such that $\chi({\bf S}_{i,j})=\sigma_{i,j}$, where
$$\boldsymbol\sigma_{i,j}:=\underbrace{I\otimes\cdots\otimes I}_{\text{${i-1}$
times}}\otimes \,\sigma_{i,j}\otimes \underbrace{I\otimes\cdots\otimes
I}_{\text{${k-i}$ times}},
$$
for  $i\in\{1,\ldots,k\}$ and  $j\in\{1,\ldots,n_i\}$, where $\{\sigma_{i,j}\}_{j=1}^{n_i}$ is a set of generators of the Cuntz algebra $\cO_{n_i}$.
We also remark (see \cite{Po-disc}) that  the closed non-seladjoint algebra $Alg(1,\sigma_i)$  generated by $\{\sigma_{i,j}\}_{j=1}^{n_i}$ and the identity is completely isometric isomorphic to the noncommutative disc algebra $\cA_{n_i}$. Consequently, one can see  $\boldsymbol\cA_{\bf n}\simeq \cA_{n_1}\otimes_{min}\cdots \otimes_{min} \cA_{n_k}$ as a subalgebra of $\cO_{n_1}\otimes \cdots \otimes  \cO_{n_k}$.

\begin{corollary} Let  ${\bf n}=(n_1,\ldots n_k)\in \NN^k$. Each holomorphic automorphism of the regular polyball ${\bf B_n}$  induces an automorphism of the $C^*$-algebra $\cO_{n_1}\otimes \cdots \otimes  \cO_{n_k}$ which leaves invariant the non-self-adjoint  subalgebra $\cA_{n_1}\otimes_{min}\cdots \otimes_{min} \cA_{n_k}$.
\end{corollary}

\smallskip

\section{Automorphisms of the   polyball algebra
    $\cA ({\bf B_n})$ and  the Hardy algebra $H^\infty({\bf B_n})$}

In this section,   we determine the group   of unitarily implemented  automorphisms of  the noncommutative polyball   algebra $\boldsymbol\cA_{\bf n}$  and Hardy algebra \,${\bf F}_{\bf n}^\infty$ and show that they are isomorphic to the group $ \text{\rm Aut}({\bf B_n})$. We also present the corresponding results for
  the Hardy algebra of all  bounded free holomorphic functions on the regular polyball $H^\infty({\bf B_n})$  and the   polyball algebra
    $\cA ({\bf B_n})$.

\begin{proposition} Let $f: {\bf B_m}(\cH)\to B(\cH)$ and $g:{\bf B_n}(\cH)\to {\bf B_m}(\cH)$ be free holomorphic functions. Then the following statements hold.
\begin{enumerate}
\item[(i)] If $f$ and $g$ have continuous extensions to the closed polyballs ${\bf B_m}(\cH)^-$ and ${\bf B_n}(\cH)^-$, respectively, then $f\circ g\in A({\bf B_n})$.
    \item[(ii)] If $f\in H^\infty({\bf B_m})$ then $f\circ g\in H^\infty({\bf B_n})$ and
    $\|f\circ g\|_\infty\leq \|f\|_\infty.$
    \item[(iii)] If $f\in H^\infty({\bf B_m})$  and $\hat g=(\hat g_1,\ldots, \hat g_m)$ is a pure element of the polyball ${\bf B_m}(\otimes_{i=1}^k F^2(H_{n_i}))$ with entries $\hat g_j\in \boldsymbol\cA_{\bf n}$, then
        $
        (f\circ g)({\bf X})=\boldsymbol \cB_{\bf X}[\boldsymbol \cB_{\hat g}[\hat f]]$ for  ${\bf X}\in {\bf B_n}(\cH).$

\end{enumerate}

\end{proposition}
\begin{proof}
Using Theorem \ref{homo-compo}, part (i) and (ii) are obvious.
Since $\text{\rm range}\,g\subset {\bf B_m}(\cH)$, Proposition \ref{range} implies  $g(r{\bf S})\in {\bf B_m}(\otimes_{i=1}^k F^2(H_{n_i}))$, $r\in [0,1)$, where
${\bf S}$  is the universal model of ${\bf B_n}$. Since $f\circ g\in H^\infty({\bf B_n})$ its boundary function  $\widehat{f\circ g}$ exists and
\begin{equation}\label{fcircg}
\widehat{f\circ g}=\text{\rm SOT-}\lim_{r\to 1}f(g(r{\bf S})).
\end{equation}
According to the second part of Proposition \ref{Berezin-comp}, we  have
\begin{equation}
\label{fBB}
f(g(r{\bf S}))=\boldsymbol{\cB}_{g(r{\bf S})}[\hat f)=\boldsymbol\cB_{r\bf{ S}}[\boldsymbol\cB_{\hat g}[\hat f]].
\end{equation}
Due to Theorem 3.5 and Lemma 3.3 from \cite{Po-Berezin-poly}, if $\Psi\in {\bf F_n^\infty}$, then $\psi=\text{\rm SOT-}\lim_{r\to 1}\boldsymbol\cB_{r\bf{S}}[\psi]$. Applying this result to $\boldsymbol\cB_{\hat g}[\hat f]\in  {\bf F_n^\infty}$ and using relations
\eqref{fcircg} and \eqref{fBB}, we deduce that $\widehat{f\circ g}=\boldsymbol\cB_{\hat g}[\hat f]$.
Since $(f\circ g)({\bf X})=\boldsymbol \cB_{\bf X}[\widehat{f\circ g} ]$  for ${\bf X}\in {\bf B_n}(\cH)$, we complete the proof.
\end{proof}

\begin{corollary} Let $f\in Hol({\bf B_n})$ and $\Psi\in Aut({\bf B_n})$. Then the following statements hold.
\begin{enumerate}
\item[(i)] $f\circ \Psi\in A({\bf B_n})$ for any $f \in A({\bf B_n})$.
\item[(ii)]  $f\circ \Psi\in H^\infty({\bf B_n})$ for any $f \in H^\infty({\bf B_n})$.
\item[(iii)] If $f \in H^\infty({\bf B_n})$, then $\|f\circ \Psi\|_\infty=\|f\|_\infty$ and
$$
        (f\circ \Psi)({\bf X})=\boldsymbol \cB_{\bf X}[\boldsymbol \cB_{\hat \Psi}[\hat f]],\qquad {\bf X}\in {\bf B_n}(\cH).$$
\end{enumerate}

\end{corollary}

We remark that there are operator-valued coefficient versions of the previous two results and the proofs are similar.

\begin{theorem} \label{auto-polyball} Any unitarily implemented  automorphism of  the noncommutative polyball algebra  $\boldsymbol\cA_{\bf n} $   is the Berezin transform $\boldsymbol\cB_{\hat \Psi}|_{\boldsymbol\cA_{\bf n}}$ of a boundary function   $\hat\Psi$,  where  $\Psi\in \text{\rm Aut}({\bf B_n})$. Moreover, we have
$$
\text{\rm Aut}_u(\boldsymbol\cA_{\bf n})\simeq \text{\rm Aut}({\bf B_n}).
$$
\end{theorem}
\begin{proof} First, assume that $\Psi\in \text{\rm Aut}({\bf B_n})$. Due to Theorem \ref{auto-C*}, the noncommutative Berezin transform $\boldsymbol\cB_{\hat \Psi}$ is a unitarily implemented automorphism of the Cuntz-Toeplitz  algebra $C^*({\bf S})$ such that  $\boldsymbol\cB_{\hat \Psi}(\boldsymbol\cA_{\bf n})=\boldsymbol\cA_{\bf n}$. Consequently, $\boldsymbol\cB_{\hat \Psi}|_{\boldsymbol\cA_{\bf n}}$ is a unitarily implemented automorphism of the noncommutative polyball algebra $\boldsymbol\cA_{\bf n}$.

Now, we assume that $\Gamma$ is a unitarily implemented automorphism of $\boldsymbol\cA_{\bf n}$, i.e.,
there exists a unitary operator $U\in B(\otimes_{i=1}^k F^2(H_{n_i}))$ such that
$\Gamma(Y)=U^* YU$ for any $Y\in \boldsymbol\cA_{\bf n}$.
  As in the proof of  Theorem \ref{auto-C*}, we deduce that there is $\Psi\in \text{\rm Aut}({\bf B_n})$ such that
$\Gamma=\boldsymbol{\cB_{\hat\Psi}}|_{\boldsymbol\cA_{\bf n}}$ and
$
\text{\rm Aut}({\bf B_n}) \simeq \text{\rm Aut}_u(\boldsymbol\cA_{\bf n}).
$
The proof is complete.
\end{proof}

 We remark that  Theorem \ref{auto-C*}  and Theorem \ref{auto-polyball} reveal that each unitarily implemented automorphism of $\boldsymbol \cA_{\bf n}$ has a unique extension  to an automorphism of the $C^*$-algebra $C^*({\bf S})$. Moreover,   the mappings
$\boldsymbol\cB_{\hat \Psi}|_{\boldsymbol\cA_{\bf n}}\mapsto \boldsymbol\cB_{\hat \Psi}\mapsto \Psi $ are group isomorphisms, showing that
$$
\text{\rm Aut}_u(\boldsymbol\cA_{\bf n})\simeq \text{\rm Aut}_{\boldsymbol\cA_{\bf n}}(C^*({\bf S}))\simeq \text{\rm Aut}({\bf B_n}).
$$

If $\Lambda: A({\bf B_n})\to A({\bf B_n})$
is an algebraic  homomorphism, it induces a unique homomorphism $\tilde
\Lambda:\boldsymbol{\cA_n}\to \boldsymbol{\cA_n}$ such that the  diagram
\begin{equation*}
\begin{matrix}
\boldsymbol{\cA_n}& \mapright{\tilde\Lambda}&
\boldsymbol{\cA_n}\\
 \mapdown{\boldsymbol\cB}& & \mapdown{\boldsymbol\cB}\\
A({\bf B_n})&  \mapright{\Lambda}& A({\bf B_n})
\end{matrix}
\end{equation*}
is commutative, i.e., $\Lambda \boldsymbol\cB=\boldsymbol\cB\tilde \Lambda$. The homomorphisms $\Lambda$
and $\tilde \Lambda$ uniquely determine each other by the formulas:
\begin{equation*}
\begin{split}
(\Lambda f)(X)&=\boldsymbol\cB_X[\tilde \Lambda(\hat f)], \qquad  f\in A({\bf B_n}),\
X\in {\bf B_n}(\cH), \quad \text{ and }\\
\tilde \Lambda(\hat f)&=\widehat{\Lambda(f)}, \qquad \hat f\in \boldsymbol\cA_{\bf n}.
\end{split}
\end{equation*}

We say that  a unital completely contractive  homomorphism $\tilde\Lambda:\boldsymbol\cA_{\bf n}\to \boldsymbol\cA_{\bf n}$ has  a completely contractive hereditary linear extension  to $C^*({\bf S})$ if  the linear maps defined by
$$
{\bf S}_{(\alpha)} {\bf S}_{(\beta)}^* \mapsto \tilde\Lambda({\bf S}_{(\alpha)}) \tilde\Lambda({\bf S}_{(\beta)})^*, \qquad (\alpha), (\beta)\in \FF_{n_1}^+\times \cdots \times \FF_{n_k}^+,
$$
and
$$
{\bf S}_{(\alpha)} {\bf S}_{(\beta)}^* \mapsto \tilde\Lambda^{-1}({\bf S}_{(\alpha)}) \tilde\Lambda^{-1}({\bf S}_{(\beta)})^*, \qquad (\alpha), (\beta)\in \FF_{n_1}^+\times \cdots \times \FF_{n_k}^+,
$$
are completely contractive.

\begin{theorem} \label{auto-A} Let $\Lambda:A({\bf B_n})\to A({\bf B_n})$ be a unital algebraic automorphism. Then the following statements are equivalent.
\begin{enumerate}
 \item[(i)]
    $\tilde \Lambda$ is a unitarily implemented automorphism of $\boldsymbol\cA_{\bf n}$.

    \item[(ii)] There is $\varphi\in \text{\rm Aut}({\bf B_n})$ such that
    $$\Lambda(f)=f\circ \varphi,\qquad f\in A({\bf B_n}).$$
   \item[(iii)] $\tilde\Lambda$ is a completely contractive automorphism   of $ \boldsymbol\cA_{\bf n}$ with completely contractive hereditary linear extension to $C^*({\bf S})$.
    \item[(iv)] $\tilde\Lambda$ is continuous and $\{\tilde\Lambda({\bf S}_{i,j})\}$ and $\{\tilde\Lambda^{-1}({\bf S}_{i,j})\}$ are  in the polyball ${\bf B_n}(\otimes_{i=1}^k F^2(H_{n_i}))^-$, where ${\bf S}=\{{\bf S}_{i,j}\}$ is the universal model of the regular polyball ${\bf B_n}$.
\end{enumerate}

\end{theorem}
\begin{proof}
Assume that $(i)$ holds. According to Theorem \ref{auto-polyball}, there is $\varphi\in \text{\rm Aut}({\bf B_n})$ such that $\tilde\Lambda = \boldsymbol\cB_{\hat \varphi}|_{\boldsymbol\cA_{\bf n}}$. Consequently,  using Proposition \ref{Berezin-comp} we obtain
\begin{equation*}
\begin{split}
\Lambda(f)({\bf X})&=\boldsymbol\cB_{\bf X} [\tilde\Lambda(\hat f)]=\boldsymbol\cB_{\bf X}[\boldsymbol\cB_{\hat \varphi}[\hat f)]]\\
&=\boldsymbol\cB_{\varphi({\bf X})}[\hat f]=f(\varphi({\bf X}))=(f\circ \varphi)({\bf X})
\end{split}
\end{equation*}
for any  $f\in A({\bf B_n})$, therefore item $(ii)$ holds.  Now, we prove that $(ii)\implies (iii)$.
Note that we have
$$
\tilde\Lambda(\hat f)=\widehat{\Lambda(f)}=\widehat{f\circ \varphi}=\boldsymbol\cB_{\hat \varphi}[\hat f]$$
for any $f\in A({\bf B_n})$. Hence $\tilde\Lambda=\boldsymbol\cB_{\hat \Psi}|_{\boldsymbol\cA_{\bf n}}$, which is a completely contractive  automorphism  and $\boldsymbol\cB_{\hat \Psi}$ is a completely contractive hereditary linear extension to $C^*({\bf S})$ (see Theorem \ref{auto-C*}).
Let us prove that $(iii)\implies (iv)$. Assume that $(iii)$ holds.
For each $i\in\{1,\ldots,k\}$ and  $j\in\{1,\ldots,n_i\}$, set $\hat\varphi_{i,j}:=\tilde\Lambda({\bf S}_{i,j})\in \boldsymbol\cA_{\bf n}$.
We need to show that
$\hat\varphi:=(\hat\varphi_1,\ldots, \hat\varphi_k)$, with $\hat\varphi_{i,1},\ldots, \hat\varphi_{i,n_i})$, is in the noncommutative polyball ${\bf B_n}(\otimes_{i=1}^k F^2(H_{n_i})^-$.
Since $\Phi_{{\bf S}_i}(I)\leq I$ and $\tilde \Lambda$ is completely contractive, we deduce that
$\Phi_{{\hat\varphi}_i}(I)\leq I$   for $i\in \{1,\ldots, k\}$. Let $1\leq p\leq k$ and $i_1<\cdots < i_p$ with $i_1,\ldots, i_p\in \{1,\ldots, k\}$. We have
$$
0\leq (id-\Phi_{{\bf S}_{i_1}})\circ \cdots \circ  (id-\Phi_{{\bf S}_{i_p}})(I)
=I-\sum_{q_j\in \{0,1\},\, q_1+\cdots +q_p>0} (-1)^{q_1+\cdots q_k+1} \Phi_{{\bf S}_{i_1}}\cdots \Phi_{{\bf S}_{i_p}},
$$
which is equivalent to
$$
\left\|\sum_{q_j\in \{0,1\},\, q_1+\cdots +q_p>0} (-1)^{q_1+\cdots q_k+1} \Phi_{{\bf S}_{i_1}}\cdots \Phi_{{\bf S}_{i_p}}\right\|\leq 1.
$$
Since  $\tilde \Lambda$ has completely contractive hereditary linear extension, we deduce that
$$
\left\|\sum_{q_j\in \{0,1\},\, q_1+\cdots +q_p>0} (-1)^{q_1+\cdots +q_k+1} \Phi_{{\hat \varphi}_{i_1}}\cdots \Phi_{{\hat \varphi}_{i_p}}\right\|\leq 1.
$$
Taking into account that the operator under the norm is self-adjoint, we deduce that
$$\sum_{q_j\in \{0,1\},\, q_1+\cdots +q_p>0} (-1)^{q_1+\cdots +q_k+1} \Phi_{{\hat \varphi}_{i_1}}\cdots \Phi_{{\hat \varphi}_{i_p}}\leq I,
$$
 which is equivalent  to
$$
(id-\Phi_{{\hat\varphi}_{i_1}})\circ \cdots \circ  (id-\Phi_{{\hat\varphi}_{i_p}})(I)\geq 0.
$$
This shows that $\hat\varphi:=(\hat\varphi_1,\ldots, \hat\varphi_k)$ is in the noncommutative polyball ${\bf B_n}(\otimes_{i=1}^k F^2(H_{n_i})^-$. Similarly, we can show that
$\{\tilde\Lambda^{-1}({\bf S}_{i,j})\}$ is  in the polyball ${\bf B_n}(\otimes_{i=1}^k F^2(H_{n_i}))^-$. Therefore, item $(iv)$ holds.

It remains to prove that $(iv)\implies (i)$. Assume that $\tilde\Lambda({\bf S}):=\{\tilde\Lambda({\bf S}_{i,j})\}\in {\bf B_n}(\otimes_{i=1}^k F^2(H_{n_i}))^-$.
Due to the noncommutative von Neumann type inequality \cite{Po-Berezin-poly}, we have
$$
\left\|\left[\tilde\Lambda(p_{i,j}({\bf S}))\right]_{m\times m}\right\|
=
\left\|\left[p_{i,j}(\tilde\Lambda({\bf S}))\right]_{m\times m}\right\|
\leq
\left\|\left[p_{i,j}({\bf S})\right]_{m\times m}\right\|
$$
for any operator matrix $\left[p_{i,j}({\bf S})\right]_{m\times m}\in \boldsymbol\cA_{\bf n}\otimes M_{m\times m}(\CC)$. Since $\tilde\Lambda $ is continuous  and $\boldsymbol\cA_{\bf n}$ is the norm closed self-adjoint  algebra generated by $\{{\bf S}_{i,j}\}$ and the identity, we deduce that
$\tilde\Lambda:\boldsymbol\cA_{\bf n}\to \boldsymbol\cA_{\bf n}$ is a completely contractive homomorphism. Similarly, using the fact that
$\{\tilde\Lambda^{-1}({\bf S}_{i,j})\}$ is  in the polyball ${\bf B_n}(\otimes_{i=1}^k F^2(H_{n_i}))^-$, one can prove that $\tilde\Lambda^{-1}:\boldsymbol\cA_{\bf n}\to \boldsymbol\cA_{\bf n}$ is  also a completely contractive homomorphism.
Now, as in the proof of Theorem \ref{auto-C*}, one can show that $\tilde \Lambda$ is  a unitarily implemented automorphism of  $\boldsymbol\cA_{\bf n}$.
This completes the proof.
\end{proof}

We remark that if $\Lambda:A({\bf B_n})\to A({\bf B_n})$ is a unital algebraic homomorphism and at least one of $n_1,\ldots, n_k$ is greater than  or equal 2, then $\tilde \Lambda$ is automatically continuous. Indeed, assume that there is $i_0\in \{1,\ldots, k\}$ such that $n_{i_0}\geq 2$ and $\tilde \Lambda$ is not continuous in the operator norm. Then there is a sequence  $\{g_p\}_{p=1}^\infty$ of elements in the polyball algebra $\boldsymbol\cA_{\bf n}$ such that
$\tilde\Lambda(g_p)\geq p$  and $\|g_p\|\leq \frac{1}{M^{p+2}}$ for any $p\in \NN$, for some constant $M>1$ with $M> \|\tilde\Lambda^{-1}({\bf S}_{i,j})\|$ for any $i\in \{1,\ldots, k\}$ and $j\in \{1,\ldots, n_i\}$. Note that $g:=\sum_{p=1}^\infty \tilde\Lambda^{-1}({\bf S}_{i_0,1})^p\tilde\Lambda^{-1}({\bf S}_{i_0,2})g_p$ is convergent in  norm and, consequently, it is in the polyball algebra $\boldsymbol\cA_{\bf n}$. For each $q\in \NN$, we have
$
\tilde\Lambda (g)=\sum_{p=1}^q {\bf S}_{i_0,1}^p {\bf S}_{i_0,2}\tilde \Lambda(g_p)+{\bf S}_{i_0,1}^{q+1} \tilde \Lambda (\xi_q)
$
for some $\xi_q\in \boldsymbol\cA_{\bf n}$. Since ${\bf S}_{i_0,1}$ and $ {\bf S}_{i_0,2}$ are isometries with orthogonal ranges, we have
${\bf S}_{i_0,2}^* ({\bf S}_{i_0,1}^*)^q\tilde\Lambda(g)=\tilde \Lambda(g_q)$ and, consequently,
$
\|\tilde \Lambda(g)\|\geq \|\tilde \Lambda(g_q)\|\geq q$ for $q\in \NN,
$
which is a contradiction. Therefore, $\tilde \Lambda$ is continuous.

\begin{theorem} \label{auto-F} Any unitarily implemented  automorphism of  the noncommutative Hardy algebra     ${\bf F}_{\bf n}^\infty $ is the Berezin transform $\boldsymbol\cB_{\hat \Psi} $ of a boundary function   $\hat\Psi$,  where  $\Psi\in \text{\rm Aut}({\bf B_n})$. Moreover, we have
$$
\text{\rm Aut}_u({\bf F}_{\bf n}^\infty)\simeq  \text{\rm Aut}({\bf B_n}).
$$
\end{theorem}
\begin{proof} Let $\boldsymbol\Psi=(\Psi_1,\ldots, \Psi_k)\in \text{\rm Aut}({\bf B_n})$. According to Theorem \ref{structure2},
 each $\hat{\Psi}_i=(\hat{\Psi}_{i,1},\ldots, \hat{\Psi}_{i,n_i})$  is a pure row isometry with entries in the noncommutative disk algebra generated by ${\bf S}_{i,1},\ldots, {\bf S}_{i,n_i}$ and the identity. Consider the Berezin transform $\boldsymbol{\cB}_{\hat {\Psi}}: {\bf F_n^\infty}\to B(F^2(H_{n_1})\otimes \cdots \otimes F^2(H_{n_k}))$ defined by

 $$
 \boldsymbol{\cB}_{\hat {\Psi}}[f]:={\bf K}_{\hat {\Psi}}[f\otimes I_{\cD_{\hat {\Psi}}}]{\bf K}_{\hat {\Psi}}^*,\qquad
f\in {\bf F_n^\infty}.
$$
Due to Theorem \ref{structure2} and Corollary \ref{unit-implement}, the noncommutative Berezin kernel ${\bf K}_{\hat {\Psi}}$ is a unitary operator. We recall that if $f\in {\bf F_n^\infty}$, then $f_r\in \boldsymbol\cA$, $\|f_r\|\leq \|f\|$ and $\text{\rm SOT-}\lim_{r\to 1} f_r=f$. Since $\boldsymbol{\cB}_{\hat {\Psi}}[{\bf S}_{(\alpha)}]=\hat\Psi_{(\alpha)}$ is in $\boldsymbol\cA_{\bf n}$ for any
$(\alpha)\in \FF_{n_1}^+\times \cdots \times \FF_{n_k}^+$, and ${\bf F_n^\infty}$ is the WOT-closed non-selfadjoint algebra generate by $\{{\bf S}_{(\alpha)}\}_{(\alpha)\in \FF_{n_1}^+\times \cdots \times \FF_{n_k}^+}$, we deduce that
$\boldsymbol{\cB}_{\hat {\Psi}}[{\bf F_n^\infty}]\subseteq {\bf F_n^\infty}$.
On the other hand, $\boldsymbol{\cB}_{\widehat {\Psi^{-1}}}$ has similar properties and, due to Proposition \ref{Berezin-comp}, we have
$(\boldsymbol{\cB}_{\hat {\Psi}}\boldsymbol{\cB}_{\widehat {\Psi^{-1}}})[f]=f$ for any
$f\in {\bf F_n^\infty}$.  Therefore $\boldsymbol{\cB}_{\hat {\Psi}}({\bf F_n^\infty})={\bf F_n^\infty}$ and $\boldsymbol{\cB}_{\hat {\Psi}}$ is a unitarily implemented automorphism of ${\bf F_n^\infty}$.

Now, we assume that $\Gamma$ is a unitarily implemented automorphism of ${\bf F}_{\bf n}^\infty $, i.e.,
there exists a unitary operator $U\in B(\otimes_{i=1}^k F^2(H_{n_i}))$ such that
$\Gamma(Y)=U^* YU$ for any $Y\in {\bf F}_{\bf n}^\infty$. For each $i\in\{1,\ldots,k\}$ and  $j\in\{1,\ldots,n_i\}$, set $\tilde\varphi_{i,j}:=\Gamma({\bf S}_{i,j})\in {\bf F}_{\bf n}^\infty$.
Since
$$
(id-\Phi_{\tilde\varphi_1})^{p_1}\circ \cdots \circ (id-\Phi_{\tilde\varphi_k})^{p_k}
=U^*\left[(id-\Phi_{{\bf S}_i})^{p_1}\circ \cdots \circ (id-\Phi_{{\bf S}_i})^{p_k} (I)\right]U\geq 0
$$
for any $p_i\in \{0,1\}$, and
$$
\tilde\varphi_{s,j}\tilde \varphi_{t,p}=U^*({\bf S}_{s,j}{\bf S}_{t,p})U=U({\bf S}_{t,p}{\bf S}_{s,j})U=\tilde \varphi_{t,p}\tilde\varphi_{s,j}
$$
for $s,t\in \{1,\ldots, k\}$, $s\neq t$, and  any $j\in \{1,\ldots, n_s\}$, $p\in \{1,\ldots, n_t\}$, we deduce that the $k$-tuple $\tilde\varphi:=(\tilde\varphi_1, \ldots, \tilde \varphi_k)$ is in the   closed  regular polyball
${\bf B_n}(\otimes_{i=1}^k F^2(H_{n_i}))^-$. On the other hand,
  each $n_i$-tuple $\tilde\varphi_i:=(\tilde\varphi_{i,1},\ldots, \tilde\varphi_{i,n_i})$ is a row isometry with entries in the  Hardy algebra ${\bf F}_{\bf n}^\infty$, and
$$
\Phi^p_{\tilde \varphi_i}(I)=\sum_{\alpha\in \FF_{n_i}^+, |\alpha|=p} \tilde \varphi_{i, \alpha}\tilde \varphi_{i,\alpha}^*=U^*\left(\sum_{\alpha\in \FF_{n_i}^+, |\alpha|=p} {\bf S}_{i,\alpha} {\bf S}_{i,\alpha}^*\right)U
$$
for any $p\in \NN$. Consequently, $\Phi^p_{\tilde \varphi_i}(I)\to 0$  strongly as $p\to \infty$. Setting
$$
\varphi_{i,j}({\bf X}):=\boldsymbol\cB_{\bf X}[\tilde \varphi_{i,j}],\qquad {\bf X}\in {\bf B_n}(\cH),
$$
 we deduce that the map $\varphi $ defined by
$\varphi({\bf X}):=(\varphi_1({\bf X}), \ldots, \varphi_k({\bf X}))$  is a free holomorphic function on ${\bf B_n}(\cH)$ with values in ${\bf B_n}(\cH)^-$.
If ${\bf X}\in {\bf B_n}(\cH)$, we can use the Berezin transform at ${\bf X}$ and obtain
$$
\sum_{\alpha\in \FF_{n_i}^+, |\alpha|=p}  \varphi_{i, \alpha}({\bf X}) \varphi_{i,\alpha}({\bf X})^*
={\bf K^*_X}\left(\sum_{\alpha\in \FF_{n_i}^+, |\alpha|=p} \tilde \varphi_{i, \alpha}\tilde \varphi_{i,\alpha}^*\otimes I_\cH\right){\bf K_X}.
$$
Since $\sum_{\alpha\in \FF_{n_i}^+, |\alpha|=p} \tilde \varphi_{i, \alpha}\tilde \varphi_{i,\alpha}^*\leq I$ for any $p\in \NN$ and $\sum_{\alpha\in \FF_{n_i}^+, |\alpha|=p} \tilde \varphi_{i, \alpha}\tilde \varphi_{i,\alpha}^*\to0$  strongly as $p\to\infty$, we deduce that
$\sum_{\alpha\in \FF_{n_i}^+, |\alpha|=p}  \varphi_{i, \alpha}({\bf X}) \varphi_{i,\alpha}({\bf X})^*\to 0$  strongly as $p\to\infty$. Therefore, each $\varphi_i({\bf X})$ is a pure row contraction for any ${\bf X}\in {\bf B_n}(\cH)$.
In particular, $\varphi_i(0)=\lambda_i=(\lambda_{i,1},\ldots, \lambda_{i, n_i})\in (\CC^{n_i})_1^-$.
Hence,  we deduce that
$$
\left(\sum_{j=1}^{n_i} |\lambda_{i,j}|^2\right)^p=\sum_{\alpha\in \FF_{n_i}^+, |\alpha|=p}  \varphi_{i, \alpha}(0) \varphi_{i,\alpha}(0)^*\to 0,\qquad \text{ as }p\to\infty.
$$
This implies $\|\lambda_i\|_2<1$ and $\varphi(0)=(\varphi_1(0),\ldots, \varphi_k(0))\in {\bf B_n}(\CC)$.
Therefore, for each $i\in \{1,\ldots, k\}$, $\varphi_i:{\bf B_n}(\cH)\to B(\cH)^{n_i}$  is a free holomorphic function with the properties: $\|\varphi_i\|_\infty=1$ and $\|\varphi_i(0)\|<1$. Applying Theorem \ref{max-mod}, we conclude that $\|\varphi_i({\bf X})\|<1$ for any ${\bf X}\in {\bf B_n}(\cH)$. Hence, and using Proposition 1.3 from \cite{Po-Berezin-poly}, we deduce that $\varphi({\bf X})\in {\bf B_n}(\cH)$.

Now, note that $\Gamma^{-1}(Y)=UYU^*$ for any $Y\in {\bf F}_{\bf n}^\infty$.
For each $i\in\{1,\ldots,k\}$ and  $j\in\{1,\ldots,n_i\}$, let $\tilde\xi_{i,j}:=\Gamma^{-1}({\bf S}_{i,j})\in {\bf F}_{\bf n}^\infty$. As in the first part of the proof, one can show that
the $k$-tuple $\tilde\xi:=(\tilde\xi_1, \ldots, \tilde \xi_k)$, with $\tilde\xi_i:=(\tilde\xi_{i,1},\ldots, \tilde\xi_{i,n_i})$,  is in the   closed  regular polyball
${\bf B_n}(\otimes_{i=1}^k F^2(H_{n_i}))^-$.
Using the noncommutative  Berezin transform, we define
$$
\xi_{i,j}({\bf X}):=\boldsymbol\cB_{\bf X}[\tilde \xi_{i,j}],\qquad {\bf X}\in {\bf B_n}(\cH),
$$
and using again  Proposition \ref{Berezin-comp} we deduce that the map $\xi$ defined by
$\xi({\bf X}):=(\xi_1({\bf X}), \ldots, \xi_k({\bf X}))$ and
$\xi_i({\bf X}):=(\xi_{i,1}({\bf X}), \ldots, \xi_{i,n_i}({\bf X}))$  is a free holomorphic function on ${\bf B_n}(\cH)$.
As above, one can  prove that
$\xi({\bf X})\in {\bf B_n}(\cH)$ for any ${\bf X}\in {\bf B_n}(\cH)$.
As in the proof of Theorem \ref{auto-C*}, we have  $(\xi\circ \varphi)({\bf X})=(\varphi\circ \xi)({\bf X})={\bf X}$ for any ${\bf X}\in {\bf B_n}(\cH)$, which shows $\varphi \in \text{\rm Aut}({\bf B_n})$.
 Moreover, one can show that $\Gamma|_{\boldsymbol \cA_{\bf n}}=\boldsymbol{\cB}_{\hat {\varphi}}|_{\boldsymbol \cA_{\bf n}}$. Since
$\boldsymbol \cA_{\bf n}$ is $w^*$-dense in ${\bf F_n^\infty}$ and $\Gamma$ and $\boldsymbol{\cB}_{\hat {\varphi}}$ are unitarily implemented (therefore $w^*$-continuous), we deduce that $\Gamma=\boldsymbol{\cB}_{\hat {\varphi}}$.
The fact that
$
 \text{\rm Aut}({\bf B_n}) \simeq  \text{\rm Aut}_u({\bf F_n^\infty})
$ can be proved as in Theorem \ref{auto-C*}.
The proof is complete.

\end{proof}

If $\Lambda: H^\infty({\bf B_n})\to H^\infty({\bf B_n})$
is an algebraic  homomorphism, it induces a unique homomorphism $\tilde
\Lambda:{\bf F^\infty_n}\to {\bf F^\infty_n}$ such that the  diagram
\begin{equation*}
\begin{matrix}
{\bf F^\infty_n}& \mapright{\tilde\Lambda}&
{\bf F^\infty_n}\\
 \mapdown{\boldsymbol\cB}& & \mapdown{\boldsymbol\cB}\\
H^\infty({\bf B_n})&  \mapright{\Lambda}& H^\infty({\bf B_n})
\end{matrix}
\end{equation*}
is commutative, i.e., $\Lambda \boldsymbol\cB=\boldsymbol\cB\tilde \Lambda$. The homomorphisms $\Lambda$
and $\tilde \Lambda$ uniquely determine each other by the formulas:
\begin{equation*}
\begin{split}
(\Lambda f)(X)&=\boldsymbol\cB_X[\tilde \Lambda(\hat f)], \qquad  f\in H^\infty({\bf B_n}),\
X\in {\bf B_n}(\cH), \quad \text{ and }\\
\tilde \Lambda(\hat f)&=\widehat{\Lambda(f)}, \qquad \hat f\in {\bf F^\infty_n}.
\end{split}
\end{equation*}

\begin{theorem} Let $\Lambda:H^\infty({\bf B_n})\to H^\infty({\bf B_n})$ be a unital algebraic automorphism. Then the following statements are equivalent.
\begin{enumerate}
 \item[(i)]
    $\tilde \Lambda$ is a unitarily implemented automorphism of ${\bf F}^\infty_{\bf n}$.
 \item[(ii)] There is $\varphi\in \text{\rm Aut}({\bf B_n})$ such that
    $$\Lambda(f)=f\circ \varphi,\qquad f\in H^\infty({\bf B_n}).$$
   \item[(iii)] $\tilde\Lambda$ is a $WOT$-continuous, completely contractive automorphism   of ${\bf F_n^\infty}$ with completely contractive hereditary linear extension.
    \item[(iv)] $\tilde\Lambda$ is norm-continuous and $WOT$-continuous such that $\{\tilde\Lambda({\bf S}_{i,j})\}$ and $\{\tilde\Lambda^{-1}({\bf S}_{i,j})\}$ are  in the polyball ${\bf B_n}(\otimes_{i=1}^k F^2(H_{n_i}))^-$, where ${\bf S}=\{{\bf S}_{i,j}\}$ is the universal model of the regular polyball ${\bf B_n}$.
\end{enumerate}

\end{theorem}
\begin{proof} The  implications $(i)\implies (ii)\implies (iii)$  follow from
Theorem \ref{auto-F} and Proposition \ref{Berezin-comp}. Now, assume that item (iii) holds. As in the proof of Theorem \ref{auto-A} (implication $(iii)\implies (iv)$), one can prove that $\{\tilde\Lambda({\bf S}_{i,j})\}$ and $\{\tilde\Lambda^{-1}({\bf S}_{i,j})\}$ are  in the polyball ${\bf B_n}(\otimes_{i=1}^k F^2(H_{n_i}))^-$, hence,  item $(iv)$ holds. If we assume that $(iv)$ holds, then, due to the continuity in norm of $\tilde\Lambda$, we deduce, according to Theorem \ref{auto-A}, that $\varphi\in \text{\rm Aut}({\bf B_n})$ and $\tilde\Lambda|_{\boldsymbol\cA_{\bf n}}=\boldsymbol\cB_{\hat\varphi}|_{\boldsymbol\cA_{\bf n}}$.
 Recall that $\hat \varphi$  is pure (see Theorem \ref{structure2}) and  $\boldsymbol\cB_{\hat\varphi}$ is a unitarily implemented automorphism of ${\bf F_n^\infty}$.
 Since
$\boldsymbol \cA_{\bf n}$ is $WOT$-dense in ${\bf F_n^\infty}$ and $\tilde \Lambda$ and $\boldsymbol{\cB}_{\hat {\varphi}}$ are    $WOT$-continuous on ${\bf F_n^\infty}$, we deduce that $\tilde\Lambda=\boldsymbol{\cB}_{\hat {\varphi}}$. Therefore, item $(i)$ holds. The proof is complete.
\end{proof}

\smallskip

\section{The automorphism group $Aut({\bf B_n})$ and  unitary projective representations}
In this section, we prove that, under a natural topology, the free holomorphic automorphism group
    $\text{\rm Aut}( {\bf B_n}))$ is  a metrizable, $\sigma$-compact,  locally compact group,
     and provide  a  concrete unitary projective representation of it in terms
      of noncommutative Berezin kernels associated with regular polyballs.

According to Section 3, any
 $\Phi\in \text{\rm Aut}({\bf B_n})$, it is uniformly continuous on
${\bf B_n}(\cH)^-$.  Using standard arguments, one can easily  prove the following result.

\begin{lemma} \label{prod} Let $\Phi_m, \Phi, \Gamma_p$, and  $\Gamma$ be
in the automorphism group $\text{\rm Aut}({\bf B_n})$, where $m,p\in \NN$.  If $\Phi_m\to
\Phi$ and $\Gamma_p\to \Gamma$ uniformly on ${\bf B_n}(\cH)^-$, then
$\Phi_m\circ \Gamma_p\to \Phi\circ \Gamma$ uniformly on
${\bf B_n}(\cH)^-$, as $m,p\to\infty$.
\end{lemma}

Let $\phi,\psi\in  \text{\rm Aut}({\bf B_n})$ and define
$$
d_{\bf B_n}(\phi,\psi):=\|\phi -\psi\|_\infty +
\|\phi^{-1}(0)-\psi^{-1}(0)\|.
$$
It is clear that  $d_{\bf B_n}$ is a metric on $\text{\rm Aut}({\bf B_n})$.

\begin{lemma} \label{conv-equi} Let $\boldsymbol\Psi_m=p_{\sigma^{(m)}}\circ\boldsymbol\Phi_{{\bf U}^{(m)}}\circ \boldsymbol\Psi_{\boldsymbol\lambda^{(m)}}$, $m\in \NN$,
and $\boldsymbol\Psi=p_\sigma\circ \boldsymbol\Phi_{\bf U}\circ \boldsymbol\Psi_{\boldsymbol\lambda}$ be free holomorphic automorphisms of the
 noncommutative polyball ${\bf B_n}(\cH)$ in standard form, where  $\sigma^{(m)}, \sigma\in \cS_k$, with $n_{\sigma^{(m)}(i)}=n_{\sigma(i)}=n_i$ for $i\in \{1,\ldots, k\}$,
  $${\bf U}^{(m)}=U_1^{(m)}\oplus\cdots \oplus U_k^{(m)}\ \text{ and } \  {\bf U}=U_1\oplus\cdots\oplus U_k\quad \text{ with }\ U_i^{(m)},U_i\in \cU(\CC^{n_i}),
  $$
  and
  $$\boldsymbol\lambda^{(k)}=(\lambda^{(k)}_1,\ldots, \lambda^{(k)}_k)\ \text{ and }\ \boldsymbol\lambda=(\lambda_1,\ldots, \lambda_k) \ \text{  with } \ \lambda^{(k)}_i, \lambda_i\in (\CC^{n_i})_1.
  $$
Then the following statements are equivalent:
\begin{enumerate}

\item[(i)] for each $i\in \{1,\ldots, k\}$, $U_i^{(m)}\to U_i$ in $B(\CC^{n_i})$ and $\lambda_i^{(m)}\to \lambda_i$
in the Euclidean norm of $\CC^{n_i}$, and  there is $N\in \NN$ such that $\sigma^{(m)}=\sigma$ for  any $m\geq N$.
    \item[(ii)]$p_{\sigma^{(m)}}\to p_{\sigma}$,  $\boldsymbol\Phi_{{\bf U}^{(m)}}\to \boldsymbol\Phi_{\bf U}$, and $\boldsymbol\Psi_{\boldsymbol\lambda^{(m)}}\to \boldsymbol\Psi_{\boldsymbol\lambda}$
     uniformly on  ${\bf B_n}(\cH)^-$.
     \item[(iii)] $\boldsymbol\Psi_m\to \boldsymbol\Psi$ in the metric $d_{\bf B_n}$;
\end{enumerate}
\end{lemma}
\begin{proof}
First, we prove that (i) is equivalent to (ii).
Assume that $U_i^{(m)}=[u_{st}^{(m)}]_{n_i\times n_i}$, $m\in \NN$, and $U_i=[u_{st}]_{n_i\times n_i}$ are unitary matrices with scalar entries, and  $\boldsymbol\Phi_{{\bf U}^{(m)}}\to \boldsymbol\Phi_{U}$ uniformly on  ${\bf B_n}(\cH)^-$, as $m\to\infty$.
For each $j=1,\ldots,n_i$, denote ${\bf E}_{ij}:=[0,\ldots,E_{ij},\ldots, 0]$, where $E_{ij}$ is on the $i$-position, and $E_{ij}=[0,\ldots, I,\ldots 0]$, where the identity is on the $j$-position.  Note that
$$
\|\boldsymbol\Phi_{{\bf U}^{(m)}}({\bf E}_{i,j})-\boldsymbol\Phi_{\bf U}({\bf E}_{ij})\|=\left(\sum_{j=1}^{n_i} |u_{ij}^{(m)}- u_{ij}|^2\right)^{1/2}.
$$
Consequently, if $\boldsymbol\Phi_{{\bf U}^{(m)}}\to \boldsymbol\Phi_{\bf U}$, then, for each $i\in \{1,\ldots,k\}$, we have  $u_{ij}^{(m)}\to u_{ij}$ as $m\to\infty$. Hence, $U_i^{(m)}\to U_i$ in $B(\CC^{n_i})$.
Conversely, assume that the latter condition holds. Since
$$\|\boldsymbol\Phi_{{\bf U}^{(m)}}({\bf X})-\boldsymbol\Phi_{\bf U}({\bf X})\|\leq k\|{\bf X}\|\max_{i\in \{1,\ldots,k\}}\|U_i^{(m)}-U_i\|
$$
 for any ${\bf X}\in {\bf B_n}(\cH)^-$, we deduce that  $\boldsymbol\Phi_{{\bf U}^{(m)}}\to \boldsymbol\Phi_{{\bf U}}$ uniformly on
${\bf B_n}(\cH)^-$.

Now we prove that $\lambda_i^{(m)}\to \lambda_i$
in the Euclidean norm of $\CC^{n-i}$ if and only if $\boldsymbol\Psi_{\boldsymbol\lambda^{(m)}}\to \boldsymbol\Psi_{\boldsymbol\lambda}$
     uniformly on  ${\bf B_n}(\cH)^-$. Since $\boldsymbol\Psi_{\boldsymbol\lambda^{(m)}}(0)=\boldsymbol\lambda^{(m)}$ and $\boldsymbol\Psi_{\boldsymbol\lambda}(0)=\boldsymbol\lambda$,  one implication is clear. To prove the converse, assume that, for each $i\in \{1,\ldots, k\}$, $\lambda_i^{(m)}\to \lambda_i$
in the Euclidean norm of $\CC^{n_i}$.
Since the left creation operators ${\bf S}_{i,1},\ldots, {\bf S}_{i,n_i}$ are isometries with orthogonal ranges on the full Fock space  $F^2(H_{n_i})$, we have
$$
\left\|\sum_{j=1}^{n_i} \overline{\lambda}_{i,j} {\bf S}_{i,j}\right\|=\left(\sum_{j=1}^{n_i} |\lambda_{i,j}|^2\right)^{1/2}<1.
$$
Consequently, $\left(\sum_{j=1}^{n_i} \overline{\lambda}^{(m)}_{i,j} {\bf S}_{i,j}\right)^{-1}$ converges to $\left(\sum_{j=1}^{n_i} \overline{\lambda}_{i,j} {\bf S}_{i,j}\right)^{-1}$, as $m\to \infty$, in the operator norm.
Taking into account that
$$\widehat{\Psi}_{\lambda_i}=\lambda_i -\Delta_{\lambda_i} \left(I-\sum_{j=1}^{n_i} \overline{\lambda}_{i,j} {\bf S}_{i,j}\right)^{-1}[{\bf S}_{i,1},\ldots,{\bf S}_{i,n_i}]\Delta_{\lambda_i^*}
$$
and a similar relation holds for $\widehat{\Psi}_{\lambda_i^{(m)}}$, we deduce that
$\widehat{\Psi}_{\lambda_i^{(m)}}\to \widehat{\Psi}_{\lambda_i}$ in the operator norm.
Due to the noncommutative von Neumann inequality \cite{Po-poisson}, we have
$$
\|\boldsymbol\Psi_{\boldsymbol\lambda^{(m)}}(X)-\boldsymbol\Psi_{\boldsymbol\lambda}(X)\|\leq k\max_{i\in \{1,\ldots, k\}}\|\widehat{\Psi}_{\lambda_i^{(m)}} - \widehat{\Psi}_{\lambda_i}\|
$$
for any   ${\bf X}\in {\bf B_n}(\cH)^-$. Hence,  $\boldsymbol\Psi_{\boldsymbol\lambda^{(m)}}\to \boldsymbol\Psi_{\boldsymbol\lambda}$
     uniformly on  ${\bf B_n}(\cH)^-$, which proves our assertion.

 If $\sigma^{(m)}\neq \sigma$, then there is $i_0\in \{1,\ldots, k\}$ such that   $\sigma^{(m)}(i_0)\neq \sigma(i_0)$. Hence
 $$
 \|p_{\sigma^{(m)}}-p_\sigma\|_\infty\geq \sup_{{\bf X}=(X_1,\ldots, X_k)\in {\bf B_n}(\cH)}
 \| X_{\sigma^{(m)}(i_0)} -X_{ \sigma(i_0)}\|\geq \sup_{X_{ \sigma(i_0)}\in [B(\cH)^{n_{i_0}}]_1}
 \|X_{ \sigma(i_0)}\| =1.
 $$
 Therefore, $p_{\sigma^{(m)}}\to p_\sigma$ as $m\to \infty$ if and only if there is $N\in \NN$ such that $\sigma^{(m)}=\sigma$ for  any $m\geq N$.
     In conclusion, (i) is equivalent to (ii).

     Now, we prove that (iii)$\implies$(i). Assume that $d_{\bf B_n}(\boldsymbol\Psi_m,\boldsymbol\Psi)\to 0$ as $k\to\infty$. Hence, $\boldsymbol\Psi_m\to \boldsymbol\Psi$ uniformly on ${\bf B_n}(\cH)^-$ and
     $\boldsymbol\lambda^{(m)}=\boldsymbol\Psi_m^{-1}(0)\to \boldsymbol\lambda=\boldsymbol\Phi^{-1}(0)$ in ${\bf P_n}(\CC)$. Consequently, as proved above, we have that $\boldsymbol\Psi_{\boldsymbol\lambda^{(m)}}\to \boldsymbol\Psi_{\boldsymbol\lambda}$
     uniformly on  ${\bf B_n}(\cH)^-$. Using Lemma \ref{prod} and the fact that
     $\boldsymbol\Psi_m=p_{\sigma^{(m)}}\circ\boldsymbol\Phi_{{\bf U}^{(m)}}\circ \boldsymbol\Psi_{\boldsymbol\lambda^{(m)}}$, $m\in \NN$,
and $\boldsymbol\Psi=p_\sigma\circ \boldsymbol\Phi_{\bf U}\circ \boldsymbol\Psi_{\boldsymbol\lambda}$, we deduce that
\begin{equation}
\label{uconv}
p_{\sigma^{(m)}}\circ\boldsymbol\Phi_{{\bf U}^{(m)}}\to p_\sigma\circ \boldsymbol\Phi_{\bf U}, \quad \text{ as } \ m\to \infty,
\end{equation}
uniformly on ${\bf B_n}(\cH)^-$.
Note that $\Phi_{{\bf U}^{(m)}}({\bf X})=(X_1U_1^{(m)},\ldots, X_kU^{(m)}_k)$ and
$\Phi_{{\bf U}}({\bf X})=(X_1U_1,\ldots, X_kU_k)$
for any ${\bf X}=(X_1,\ldots, X_k)\in {\bf B_n}(\cH)$. If $\sigma^{(m)}\neq \sigma$, then there is $i_0\in \{1,\ldots, k\}$ such that   $\sigma^{(m)}(i_0)\neq \sigma(i_0)$. Consequently, we have
\begin{equation*}
\begin{split}
 \|p_{\sigma^{(m)}}\circ\boldsymbol\Psi_{{\bf U}^{(m)}}-p_\sigma\circ \boldsymbol\Psi_{{\bf U}} \|_\infty
 &\geq \sup_{{\bf X}=(X_1,\ldots, X_k)\in {\bf B_n}(\cH)}
 \| X_{\sigma^{(m)}(i_0)}U_{\sigma^{(m)}(i_0)} -X_{ \sigma(i_0)}U_{ \sigma(i_0)}\|\\
 &\geq \sup_{X_{ \sigma(i_0)}\in [B(\cH)^{n_{i_0}}]_1}
 \|X_{ \sigma(i_0)}U_{ \sigma(i_0)}\| =1.
 \end{split}
 \end{equation*}
Hence, we deduce that relation \eqref{uconv} holds if and only if  there is $N\in \NN$ such that $\sigma^{(m)}=\sigma$ for  any $m\geq N$, and $\boldsymbol\Phi_{{\bf U}^{(m)}}\to \boldsymbol\Phi_{\bf U}$
     uniformly on  ${\bf B_n}(\cH)^-$. Due to the equivalence of (i) with (ii), the latter convergence is  equivalent to $U_i^{(m)}\to U_i$ in $B(\CC^{n_i})$ for each $i\in \{1,\ldots, k\}$.

 The implication (i)$\implies$(iii) is straightforward if one uses the equivalence of (i) with (ii) and
   Lemma \ref{prod}.  The proof is complete.
\end{proof}

After these preliminaries, we can prove the following

\begin{theorem} \label{Aut}
The free holomorphic automorphism group \
 $\text{\rm Aut}({\bf B_n})$ is a   $\sigma$-compact, locally
compact topological group with respect to the topology induced by the metric $d_{\bf B_n}$.
\end{theorem}
\begin{proof} First, we prove that the map
$$
\text{\rm Aut}({\bf B_n})\times \text{\rm Aut}({\bf B_n})\ni
(\boldsymbol\Psi,\boldsymbol\Gamma)\mapsto \boldsymbol\Psi\circ\boldsymbol\Gamma\in \text{\rm Aut}({\bf B_n}) $$
is continuous when $\text{\rm Aut}({\bf B_n})$ has the topology
induced by the metric $d_{\bf B_n}$. For $m,p\in \NN$,  let

\begin{equation*}
\begin{split}
\boldsymbol\Psi_m&=p_{\sigma^{(m)}}\circ\Phi_{{\bf U}^{(m)}}\circ \boldsymbol\Psi_{\boldsymbol\lambda^{(m)}},\quad \boldsymbol\Psi=p_\sigma\circ\Phi_{\bf U}\circ \boldsymbol\Psi_{\boldsymbol\lambda},\\
\boldsymbol\Gamma_p&=p_{\omega^{(p)}}\circ\Phi_{\boldsymbol W^{(p)}}\circ \boldsymbol\Psi_{\boldsymbol\mu^{(p)}},\quad   \boldsymbol\Gamma=p_\omega\circ\Phi_{\bf W}\circ
\boldsymbol\Psi_{\boldsymbol\mu},
\end{split}
\end{equation*}
be free holomorphic automorphisms of  ${\bf B_n}$, in
standard decomposition. Then
\begin{equation*}
\begin{split}
{\bf U}^{(m)}=U^{(m)}_1\oplus\cdots\oplus U^{(m)}_k,&\quad  {\bf U}=U_1\oplus\cdots\oplus U_k,\\
{\bf W}^{(p)}=W_1^{(p)}\oplus\cdots\oplus W_k^{(p)},&\quad  {\bf W}=W_1\oplus\cdots\oplus W_k,
\end{split}
\end{equation*}
 where $U_i^{(m)}, W_i^{(p)}, U_i, W_i$ are unitary
operators on $\CC^{n_i}$ and $\boldsymbol\lambda^{(m)}, \boldsymbol\mu^{(p)}, \boldsymbol\lambda,\boldsymbol\mu$ are
in ${\bf P_n}(\CC)$ satisfying relations
$$
\boldsymbol\lambda^{(m)}=\boldsymbol\Psi^{-1}_k(0),\ \boldsymbol\mu^{(p)}=\boldsymbol\Gamma_p^{-1}(0), \
\boldsymbol\lambda=\boldsymbol\Psi^{-1}(0),  \text{ and } \boldsymbol\mu=\boldsymbol\Gamma^{-1}(0).
$$
Assume that
$d_{\bf B_n}({\bf \Psi}_m, {\bf \Psi})\to 0$ as $m\to\infty$ and
$d_{\bf B_n}(\boldsymbol\Gamma_p,\boldsymbol\Gamma)\to 0$ as $p\to \infty$.
Using  Lemma
\ref{conv-equi}, we deduce that $\boldsymbol\Psi_m\circ \boldsymbol\Gamma_p\to \boldsymbol\Psi\circ
\boldsymbol\Gamma$ uniformly  on ${\bf B_n}(\cH)$.
  Note that
\begin{equation*}
\begin{split}
(\boldsymbol\Psi_m\circ \boldsymbol\Gamma_p)^{-1}(0)&= (\boldsymbol\Psi_{\mu^{(p)}}^{-1}\circ
\Phi_{{\bf W}^{(p)}}^{-1}\circ
p_{\omega^{(p)}}^{-1}\circ {\bf \Psi}_m^{-1})(0)\\
&=\left(\boldsymbol\Psi_{\mu^{(p)}}\circ \Phi_{({\bf W}^{(p)})^*}\circ p_{(\omega^{(p)})^{-1}}\right)(\boldsymbol \lambda^{(m)}).
\end{split}
\end{equation*}
Similarly, we have
\begin{equation*}
\begin{split}
(\boldsymbol\Psi\circ \boldsymbol\Gamma)^{-1}(0)&= (\boldsymbol\Psi_{\mu}^{-1}\circ
\Phi_{{\bf W}}^{-1}\circ
p_{\omega}^{-1}\circ {\bf \Psi}^{-1})(0)\\
&=\left(\boldsymbol\Psi_{\mu}\circ \Phi_{{\bf W}^*}\circ p_{\omega^{-1}}\right)(\boldsymbol \lambda).
\end{split}
\end{equation*}
According to Lemma
\ref{conv-equi}, $\boldsymbol\lambda^{(m)}\to \boldsymbol\lambda$ in ${\bf P_n}(\CC)$,
$W_i^{(p)}\to W_i$ in $B(\CC^n)$, $p_{\omega^{(p)}}^{-1}\to p_{\omega^{-1}}$ and  $\boldsymbol\Psi_{\mu^{(p)}}\to \boldsymbol\Psi_\mu$ uniformly  on
${\bf B_n}(\cH)^-$. Consequently,
$
(\boldsymbol\Psi_m\circ \boldsymbol\Gamma_p)^{-1}(0)\to (\boldsymbol\Psi\circ \boldsymbol\Gamma)^{-1}(0)
$
as  $m,p\to \infty$. Therefore,  $\boldsymbol\Psi_m\circ \boldsymbol\Gamma_p\to \boldsymbol\Psi\circ
\boldsymbol\Gamma$  in the topology induced by the metric $d_{\bf B_n}$.

  In what follows, we show that the map ${\bf \Psi}\mapsto {\bf \Psi}^{-1}$ is
  continuous on $\text{\rm Aut}({\bf B_n})$ with the topology
  induced by  $d_{\bf B_n}$.
Assume that   $d_{\bf B_n}({\bf \Psi}_m, {\bf \Psi})\to 0$  as $k\to\infty$.
Using the same notations as above, we have
$$
{\bf \Psi}_m^{-1}={\bf \Psi}_{{\boldsymbol \lambda}^{(m)}}\circ \Phi_{({\bf U}^{(m)})^*}\circ p_{(\sigma^{(m)})^{-1}} \quad \text{ and } \quad {\bf \Psi}_m^{-1}={\bf \Psi}_{{\boldsymbol \lambda}^{(m)}}\circ \Phi_{({\bf U}^{(m)})^*}\circ p_{(\sigma^{(m)})^{-1}}.
 $$
Using Lemma \ref{prod} and Lemma \ref{conv-equi}, one can easily see that $d_{\bf B_n}({\bf \Psi}_m^{-1}, {\bf \Psi}^{-1})\to 0$ as $m\to\infty$. Therefore, $\text{\rm Aut}({\bf B_n})$ is a   topological group with respect to the topology induced by the metric $d_{\bf B_n}$.

On the other hand, each free holomorphic automorphism ${\bf \Psi} \in \text{\rm Aut}({\bf B_n})$ has a unique representation
$\boldsymbol\Psi=p_\sigma\circ\Phi_{\bf U}\circ \boldsymbol\Psi_{\boldsymbol\lambda}$, where $\boldsymbol\lambda:={\bf \Phi}^{-1}(0)$ and  ${\bf U}=U_1\oplus \cdots \oplus U_k$ with  $U_i\in \cU(\CC^{n_i})$, the unitary group on $\CC^{n_i}$.
This generates  a bijection
$$\chi: \text{\rm Aut}({\bf B_n})\to \Sigma \times \cU(\CC^{n_1})\times\cdots \times \cU(\CC^{n_k}) \times {\bf P_n}(\CC),
$$
 by setting
$\chi({\bf \Psi}):=(\sigma, U_1,\cdots, U_k,\boldsymbol\lambda)$, where $\Sigma$ is the discrete subgroup
$$
 \Sigma:=\{\sigma\in \cS_k: \ (n_{\sigma(1)},\ldots, n_{\sigma(k)})=(n_1,\ldots, n_k)\}.
 $$
According to Lemma \ref{conv-equi}, the map $\chi$ is a homeomorphism of topological spaces, where $\text{\rm Aut}({\bf B_n})$ has the topology induced by the metric $d_{\bf B_n}$ and  $\cU(\CC^{n_i})$ and $ {\bf P_n}(\CC)$ have the natural topology. Consequently, since
$\Sigma \times \cU(\CC^{n_1})\times\cdots \times \cU(\CC^{n_k}) \times {\bf P_n}(\CC)$ is a $\sigma$-compact, locally compact topological space, so is  the automorphism  group $\text{\rm Aut}({\bf B_n})$. The proof is complete.
\end{proof}

\begin{corollary}  Let ${\bf n}=(n_1,\ldots, n_k)\in \NN^k$ and
 $$
 \Sigma:=\{\sigma\in \cS_k: \ (n_{\sigma(1)},\ldots, n_{\sigma(k)})=(n_1,\ldots, n_k)\}.
 $$
 The free holomorphic automorphism group \
 $\text{\rm Aut}({\bf B_n})$  has $\text{\rm card} (\Sigma)$  path connected components.
\end{corollary}
\begin{proof} We saw in the proof of Theorem \ref{Aut} that the map
$$\chi: \text{\rm Aut}({\bf B_n})\to \Sigma \times \cU(\CC^{n_1})\times\cdots \times \cU(\CC^{n_k}) \times (\CC^{n_1})_1\times \cdots\times (\CC^{n_k})_1
$$
is a homeomorphism.
Since $\cU(\CC^{n_i})$ and $(\CC^{n_i})_1$ are path connected and $\Sigma$ has $\text{\rm card} (\Sigma)$  path connected components, the result follows.
\end{proof}

Let $\text{\rm Aut}({\bf B_n})$ be the free holomorphic
automorphism group of the noncommutative polyball ${\bf B_n}$ and let
$\cU(\cK)$ be the unitary group on the Hilbert space $\cK$.
According to Theorem \ref{Aut}, $\text{\rm Aut}({\bf B_n})$ is a topological group with respect to the metric $d_{\bf B_n}$.
A  map
$\pi: \text{\rm Aut}({\bf B_n})\to \cU(\cK)$  is called (unitary)
projective representation if the following conditions are satisfied:
\begin{enumerate}
\item[(i)] $\pi(id)=I$, where $id$ is the identity on ${\bf B_n}(\cH)$;

\item[(ii)] $
 \pi({\bf \Phi}) \pi({\bf \Psi})=c_{({\bf \Phi},{\bf \Psi})} \pi({{\bf \Phi}\circ {\bf \Psi}})$,
  for any  ${\bf \Phi}, {\bf \Psi}\in \text{\rm Aut}({\bf B_n})$, where $c_{({\bf \Phi},{\bf \Psi})}$
  is a complex number  with
 $|c_{({\bf \Phi},{\bf \Psi})}|=1$;
\item[(iii)]
 the map $\text{\rm Aut}({\bf B_n})\ni {\bf \Phi}\mapsto \left<\pi({\bf\Phi})\xi,\eta\right> \in \CC$
is continuous for each $\xi,\eta\in \cK$.
\end{enumerate}

\begin{theorem}\label{projective}   For each
${\bf \Psi}=(\Psi_1,\ldots, \Psi_k)\in \text{\rm Aut}({\bf B_n})$ with $\Psi_i=(\Psi_{i,1},\ldots, \Psi_{i, n_i})$, there is a unitary operator $U_{\bf \Psi}\in B(F^2(H_{n_1})\otimes \cdots\otimes F^2(H_{n_k}))$ satisfying the relations
$$
\Psi_{i,j}({\bf S})=U_{\bf \Psi}^* {\bf S}_{i,j} U_{\bf \Psi},\qquad i\in \{1,\ldots, k\}, j\in \{1,\ldots, n_i\},$$
   and
$$U_{\bf \Psi}  U_{\bf \Phi}=c_{({\bf \Psi},{\bf \Phi})} U_{\bf \Psi\circ \Phi},\qquad {\bf \Phi,\Psi}\in \text{\rm Aut}({\bf B_n})
$$
 for some complex number  $c_{({\bf \Phi},{\bf \Psi})}\in \TT$.
Moreover, the map ${\bf \Psi}\to U_{\bf \Psi}^*$ is continuous from the uniform topology to the strong operator topology, and
    the map
    $$\pi: \text{\rm Aut}({\bf B_n})\to B(F^2(H_{n_1})\otimes \cdots\otimes F^2(H_{n_k}))\ \text{ defined by } \ \pi({\bf \Psi}):=U_{\bf \Psi}
    $$ is
a
projective representation  of the automorphism group $\text{\rm Aut}({\bf B_n})$.
\end{theorem}
\begin{proof}

Let ${\bf \Psi}=(\Psi_1,\ldots, \Psi_n)\in \text{\rm Aut}({\bf B_n})$ and let $\widehat {\bf \Psi}=(\widehat\Psi_1,\ldots, \widehat\Psi_n)$ be its  boundary
function with respect to the universal model ${\bf S}$. According to Theorem \ref{structure2},
$\widehat {\bf \Psi}$ is a pure element in the polyball $ {\bf B_n}(\otimes_{i=1}^k F^2(H_{n_i}))^-$
  and
$\widehat{\Psi}_i=(\widehat{\Psi}_{i,1},\ldots, \widehat{\Psi}_{i,n_i})$  is an  isometry with entries in the noncommutative disk algebra generated by ${\bf S}_{i,1},\ldots, {\bf S}_{i,n_i}$ and the identity.
Moreover, $\text{\rm rank}\ {\bf \Delta_{\widehat\Psi}}=1$ and $\hat{\bf \Psi}$ is unitarily equivalent to the universal model ${\bf S}$. Combining these results with  Theorem \ref{unitar-shift} and Corollary \ref{unit-implement}, we deduce that the noncommutative Berezin kernel ${\bf K}_{\widehat {\bf \Psi}}$ is a unitary operator. Moreover, in this case,   we have
\begin{equation}
\label{BeW}
\widehat\Psi_{i,j}={\bf K}^*_{\widehat {\bf \Psi}} ( {\bf S}_{i,j}\otimes I_{\cD_{\widehat {\bf \Psi}}}){\bf K}_{\widehat {\bf \Psi}}={\bf K}^*_{\widehat {\bf \Psi}}\tilde W {\bf S}_{i,j} \tilde W^* {\bf K}_{\widehat {\bf \Psi}}, \qquad i\in \{1,\ldots, k\}, j\in \{1,\ldots, n_i\},
\end{equation}
where $\tilde W_{\bf \Psi}:\otimes_{i=1}^k F^2(H_{n_i})\to (\otimes_{i=1}^k F^2(H_{n_i}))\otimes \CC v_0$ is the unitary operator defined by
$$
\tilde W_{\bf \Psi}g:= g\otimes v_0,\qquad g\in \otimes_{i=1}^k F^2(H_{n_i}),
$$
 and the   defect space $\cD_{\widehat {\bf \Psi}}=\CC v_0$ for some vector $v_0\in \otimes_{i=1}^k F^2(H_{n_i})$ with $\|v_0\|=1$.

 According to Theorem \ref{structure}, if  $\boldsymbol\Psi\in Aut({\bf B_n}(\cH))$ and $\boldsymbol\lambda=(\lambda_1,\ldots, \lambda_k)=\boldsymbol\Psi^{-1}(0)$, then there are  unique unitary operators $U_i\in B(\CC^{n_i})$, $i\in \{1,\ldots, k\}$, and a unique permutation $\sigma\in \cS_k$ with $n_{\sigma(i)}=n_i$  such that
 $$
 \boldsymbol\Psi=p_\sigma\circ \boldsymbol\Phi_{\bf U}\circ \boldsymbol\Psi_{\boldsymbol\lambda},
 $$
 where ${\bf U}:=U_1\oplus\cdots \oplus U_k$ and $\boldsymbol\Psi_{\boldsymbol\lambda}:=(\Psi_{\lambda_1},\ldots, \Psi_{\lambda_k})$.
Moreover, we have
$$
 {\bf \Delta}_{\widehat {\bf \Psi}}(I)=
 {\bf \Delta}_{\widehat {\bf \Psi}_{\boldsymbol \lambda}}(I)
 =\boldsymbol\Delta_{\boldsymbol\lambda}\left[\prod_{i=1}^k \left(I_\cH-\sum_{j=1}^k \bar{\lambda}_{i,j} {\bf S}_{i,j}\right)^{-1}\right]P_\CC
 \left[\prod_{i=1}^k \left(I_\cH-\sum_{j=1}^k {\lambda}_{i,j} {\bf S}_{i,j}^*\right)^{-1}\right],
 $$
 where  $\boldsymbol\Delta_{\boldsymbol\lambda}=\prod_{i=1}^k(1-\|\lambda_i\|^2_2)$.
Hence we deduce that
$$\|{\bf \Delta}_{\widehat {\bf \Psi}}(I)^{1/2}(1)\|^2=\left\|\boldsymbol\Delta_{\boldsymbol\lambda}^{1/2}P_\CC
 \left[\prod_{i=1}^k \left(I_\cH-\sum_{j=1}^k {\lambda}_{i,j} {\bf S}_{i,j}^*\right)^{-1}\right](1)\right\|^2=\boldsymbol\Delta_{\boldsymbol\lambda}.
$$
Let $v_0:=\boldsymbol\Delta_{\boldsymbol\lambda}^{-1/2} {\bf \Delta}_{\widehat {\bf \Psi}}(I)^{1/2}(1)\in \otimes_{i=1}^k F^2(H_{n_i})$  and note that $\|v_0\|=1$.
Now, relation \eqref{BeW} becomes
$$
\widehat \Psi_{i,j}=\Psi_{i,j}({\bf S})=U_{\bf \Psi}^* {\bf S}_{i,j} U_{\bf \Psi},\qquad i\in \{1,\ldots, k\}, j\in \{1,\ldots, n_i\},
$$
where $U_{\bf \Psi}:=\tilde W^* {\bf K}_{\widehat {\bf \Psi}}$.
  If ${\bf \Phi}\in
\text{\rm Aut}({\bf B_n})$ with ${\bf \Phi}=(\Phi_1,\ldots, \Phi_k)$ and $\Phi_{i}=(\Phi_{i,1},\ldots, \Phi_{i,n_i})$, then the relation above written for ${\bf \Psi}\circ {\bf \Phi}$  shows that
\begin{equation}
\label{comp}
({\bf \Psi}_{i,j}\circ {\bf \Phi})({\bf S})=(\widehat{{\bf \Psi}\circ {\bf \Phi}})_{i,j}
=U_{{\bf \Psi}\circ {\bf \Phi}}{\bf S}_{i,j} U_{{\bf \Psi}\circ {\bf \Phi}}.
\end{equation}

On the other hand, due to Corollary \ref{Berezin-comp2},
$$
\boldsymbol\cB_{\widehat{\Psi\circ \Phi}}[g]=(\boldsymbol{\cB}_{\hat \Phi} \boldsymbol {\cB}_{\hat \Psi})[g]
$$
for any $g$ in the Cuntz-Toeplitz algebra $ C^*({\bf S})$. In
particular, when $g={\bf S}_{i,j}$,  we obtain
$${\bf K}_{\widehat{{\bf \Psi}\circ
{\bf \Phi}}}^*({\bf S}_{i,j}\otimes I_{\cD_{{\widehat{{\bf \Psi}\circ
{\Phi}}}}}){\bf K}_{\widehat{{\bf \Psi}\circ {\bf \Phi}}} ={\bf K}_{\widehat{ {\bf \Phi}}}^*\left\{[
{\bf K}_{\widehat{{\bf \Psi} }}^*({\bf S}_{i,j}\otimes I_{\cD_{\widehat{{\bf \Psi}
}}}){\bf K}_{\widehat{\bf \Psi}}]\otimes I_{\cD_{\widehat{\bf \Phi }}}\right\}
{\bf K}_{\widehat{ \bf \Phi}}.
$$
Hence, and using relation \eqref{comp}, we deduce that
$$
 ({\bf \Psi}_{i,j}\circ {\bf \Phi})({\bf S})=(\widehat{{\bf \Psi}\circ {\bf \Phi}})_{i,j}
 = U_{\bf \Phi}^* U_{\bf \Psi}^* {\bf S}_{i,j} U_{\bf \Psi} U_{\bf \Phi},\qquad i\in \{1,\ldots, k\}, j\in \{1,\ldots, n_i\}.
 $$
Combining this relation with \eqref{comp}, we deduce that
$$
U_{{\bf \Psi}\circ {\bf \Phi}}{\bf S}_{i,j} U_{{\bf \Psi}\circ {\bf \Phi}}=U_{\bf \Phi}^* U_{\bf \Psi}^* {\bf S}_{i,j} U_{\bf \Psi} U_{\bf \Phi}
$$
which is equivalent to
$$
 U_{\bf \Psi}U_{\bf \Phi}U_{{\bf \Psi}\circ {\bf \Phi}}^*{\bf S}_{i,j}={\bf S}_{i,j} U_{\bf \Psi}U_{\bf \Phi}U_{{\bf \Psi}\circ {\bf \Phi}}^*.
$$
Since $C^*({\bf S})$ is irreducible and $U_{\bf \Psi}U_{\bf \Phi}U_{{\bf \Psi}\circ {\bf \Phi}}^*$ is a unitary operator, we have
$U_{\bf \Psi}U_{\bf \Phi}U_{{\bf \Psi}\circ {\bf \Phi}}^*={c_{({\bf \Psi},{\bf \Phi})}} I$ for some complex number with
$|c_{({\bf \Psi},{\bf \Phi})}|=1$. Hence, we deduce that $
U_{\bf \Psi}  U_{\bf \Phi}=c_{({\bf \Psi},{\bf \Phi})} U_{\bf \Psi\circ \Phi}
$ for any ${\bf \Phi,\Psi}\in \text{\rm Aut}({\bf B_n})$.

 Let $\boldsymbol\Psi^{(m)}= (\Psi^{(m)}_1,\ldots, \Psi^{(m)}_k)$, $m\in \NN$, with $\Psi^{(m)}_i=(\Psi^{(m)}_{i,1},\ldots, \Psi^{(m)}_{i,n_i})$
and $\boldsymbol\Psi= (\Psi_1,\ldots, \Psi_k)$, $m\in \NN$, with $\Psi_i=(\Psi_{i,1},\ldots, \Psi_{i,n_i})$ be free holomorphic automorphisms of the
 noncommutative polyball ${\bf B_n}(\cH)$. Assume that $\boldsymbol\Psi^{(m)}\to \boldsymbol\Psi$ in the uniform norm. Then, for each $i\in \{1,\ldots, k\}$ and  $j\in \{1,\ldots, n_i\}$,
 $\widehat\Psi^{(m)}_{i,j}\to \widehat\Psi_{i,j}$ in the operator norm topology.

   Now consider the standard representations $\boldsymbol\Psi^{(m)}=p_{\sigma^{(m)}}\circ\boldsymbol\Phi_{{\bf U}^{(m)}}\circ \boldsymbol\Psi_{\boldsymbol\lambda^{(m)}}$
and $\boldsymbol\Psi=p_\sigma\circ \boldsymbol\Phi_{\bf U}\circ \boldsymbol\Psi_{\boldsymbol\lambda}$. Since  $\boldsymbol\Psi^{(m)}(0)={\boldsymbol \lambda}^{(m)}$ and $\boldsymbol\Psi(0)={\boldsymbol \lambda}$, we deduce that
$\|\boldsymbol\lambda^{(m)}\|_2\to \|\boldsymbol \lambda\|_2$ as $m\to\infty$.
Given $\epsilon>0$ and $x=\sum_{(\alpha)\in \FF_{n_1}^+\times\cdots \times \FF_{n_k}^+} a_{(\alpha)} e_{(\alpha)}\in \otimes _{i=1}^k F^2(H_{n_i})$,  let $q\in \NN$ be such that
\begin{equation}
\label{xe}
\|x-\sum_{{(\alpha)\in \FF_{n_1}^+\times\cdots \times \FF_{n_k}^+}\atop{|\alpha_1|+\cdots |\alpha_k|\leq q}} a_{(\alpha)} e_{(\alpha)}\|<\frac{\epsilon}{4}.
\end{equation}
Since $U_{{\bf \Psi}^{(m)}}:=\tilde W_{{\bf \Psi}^{(m)}}^* {\bf K}_{\widehat{\bf \Psi}^{(m)}}$ and
$\tilde W_{{\bf \Psi}^{(m)}}:\otimes_{i=1}^k F^2(H_{n_i})\to (\otimes_{i=1}^k F^2(H_{n_i}))\otimes \cD_{\widehat{\bf \Psi}^{(m)}}$ is the unitary operator defined by
$$
\tilde W_{{\bf \Psi}^{(m)}}g:= g\otimes \boldsymbol\Delta_{\boldsymbol\lambda^{(m)}}^{-1/2} {\bf \Delta}_{\widehat {\bf \Psi}^{(m)}}(I)^{1/2}(1),\qquad g\in \otimes_{i=1}^k F^2(H_{n_i}),
$$
we can use the properties of the noncommutative Berezin kernel  to deduce that
\begin{equation*}
\begin{split}
\sum_{{(\alpha)\in \FF_{n_1}^+\times\cdots \times \FF_{n_k}^+}\atop{|\alpha_1|+\cdots |\alpha_k|\leq q}} a_{(\alpha)} U^*_{{\bf \Psi}^{(m)}} e_{(\alpha)}
&=\sum_{{(\alpha)\in \FF_{n_1}^+\times\cdots \times \FF_{n_k}^+}\atop{|\alpha_1|+\cdots |\alpha_k|\leq q}} a_{(\alpha)}{\bf K}_{\widehat{\bf \Psi}^{(m)}}^* \tilde W_{{\bf \Psi}^{(m)}}e_{(\alpha)}\\
&=\sum_{{(\alpha)\in \FF_{n_1}^+\times\cdots \times \FF_{n_k}^+}\atop{|\alpha_1|+\cdots |\alpha_k|\leq q}} a_{(\alpha)}{\bf K}_{\widehat{\bf \Psi}^{(m)}}^*
\left(e_{(\alpha)}\otimes
\boldsymbol\Delta_{\boldsymbol\lambda^{(m)}}^{-1/2} {\bf \Delta}_{\widehat {\bf \Psi}^{(m)}}(I)^{1/2}(1)\right)\\
&=
\sum_{{(\alpha)\in \FF_{n_1}^+\times\cdots \times \FF_{n_k}^+}\atop{|\alpha_1|+\cdots |\alpha_k|\leq q}} a_{(\alpha)}[{\widehat{\bf \Psi}^{(m)}}]_{(\alpha)}
\boldsymbol\Delta_{\boldsymbol\lambda^{(m)}}^{-1/2} {\bf \Delta}_{\widehat {\bf \Psi}^{(m)}}(I)(1).
\end{split}
\end{equation*}
A similar relation holds if  we replace  $\boldsymbol\Psi^{(m)}$ with $\boldsymbol\Psi$.
Since,  for each $i\in \{1,\ldots, k\}$ and  $j\in \{1,\ldots, n_i\}$,
 $\widehat\Psi^{(m)}_{i,j}\to \widehat\Psi_{i,j}$ in the operator norm topology, and
$\|\boldsymbol\lambda^{(m)}\|_2\to \|\boldsymbol \lambda\|_2$ as $m\to\infty$,  there is $N\in \NN$ such that
\begin{equation}
\label{aUaU}
\left\|\sum_{{(\alpha)\in \FF_{n_1}^+\times\cdots \times \FF_{n_k}^+}\atop{|\alpha_1|+\cdots |\alpha_k|\leq q}} a_{(\alpha)} U^*_{{\bf \Psi}^{(m)}} e_{(\alpha)} -\sum_{{(\alpha)\in \FF_{n_1}^+\times\cdots \times \FF_{n_k}^+}\atop{|\alpha_1|+\cdots |\alpha_k|\leq q}} a_{(\alpha)} U^*_{{\bf \Psi}} e_{(\alpha)}\right\|<\frac{\epsilon}{2}
\end{equation}
for all $q\geq N$. Using  relations \eqref{xe}, \eqref{aUaU} and the fact that  $U_{{\bf \Psi}^{(m)}}$ and $U_{{\bf \Psi}}$ are unitary operators, one can easily deduce that
\begin{equation*}
\begin{split}
\|U_{{\bf \Psi}^{(m)}}^* x-U_{{\bf \Psi}}^*x\|
 <\epsilon
\end{split}
\end{equation*}
for any $q\geq N$. Therefore  the map ${\bf \Psi}\to U_{\bf \Psi}^*$ is continuous from the uniform topology to the strong operator topology.

To prove the last part of this theorem,  note  that if $\boldsymbol\Psi^{(m)}\to \boldsymbol\Psi$ in  the metric $d_{\bf B_n}$, then  $\boldsymbol\Psi^{(m)}\to \boldsymbol\Psi$ in the uniform norm and,  using  the first part of the theorem, we can complete the proof.
\end{proof}

\bigskip

       %


\begin{thebibliography}{99}



\bibitem{Arv-book}  {\sc  W.B.~Arveson},
{\it An invitation to $C^*$-algebras}, Graduate Texts in Math., {\bf
39}. Springer-Verlag, New-York-Heidelberg, 1976.




\bibitem{BeTi2}
{\sc C.~Benhida and D.~Timotin},
 Some automorphism invariance
properties for multicontractions, {\it  Indiana Univ. Math. J.} {\bf
56} (2007), no. 1, 481--499.


 \bibitem{BeTi}
{\sc C.~Benhida and D.~Timotin},
Automorphism invariance properties for certain families of multioperators, {\it Operator theory 
live}, {\bf 5–-15}, Theta Ser. Adv. Math., {\bf 12}, Theta, Bucharest, 2010.


\bibitem{Ca} {\sc H.~Cartan},
Les fonctions de deux variables complexes et le probl\` eme de la
repr\'esentation analytique, {\it J. de Math. Pures et Appl.} {\bf
96} (1931), 1--114.

  \bibitem{Cu} {\sc  J.~Cuntz},
 Simple $C^*$--algebras generated by isometries, {\it Commun.Math.Phys.}
  {\bf 57} (1977), 173--185.








\bibitem{DP2} {\sc K.~R.~Davidson and D.~Pitts},
The algebraic structure of non-commutative analytic Toeplitz algebras,
{\it  Math. Ann.}
   {\bf 311} (1998),  275--303.











\bibitem{HKMS} {\sc J.W.~Helton, I.~Klep, S.~McCullough, and
N.~Slingled}, Noncommutative ball maps, {\it J. Funct. Anal.} {\bf
257} (2009),  47--87.



\bibitem{Kr} {\sc S.G.~Krantz}, {\em Function theory of several complex variables}. Reprint of the 1992 edition. AMS Chelsea Publishing, Providence, RI, 2001. xvi+564 pp.








\bibitem{L} {\sc E.~Ligocka}, On proper holomorphic  and biholomorphic mappings between product domains,
    {\it Bull.Acad. Polon.Sci., Ser. Sci. Math.} {\bf 28} (1980), 319--323.






\bibitem{MuSo3}  {\sc P.S.~Muhly and  B.~Solel},
Schur Class Operator Functions and Automorphisms of Hardy Algebras,
{\it Documenta Math.} {\bf 13} (2008),  365--411.




\bibitem{Pa-book} {\sc V.I.~Paulsen},
 {\it Completely Bounded Maps and Dilations},
Pitman Research Notes in Mathematics, Vol.146, New York, 1986.





\bibitem{Pi-book} {\sc G.~Pisier}, {\it Similarity problems and completely bounded
maps},
 Second, expanded edition. Includes the solution to ``The Halmos problem''.
Lecture Notes in Mathematics, 1618. Springer-Verlag, Berlin, 2001. viii+198 pp.









\bibitem{Po-charact} {\sc G.~Popescu}, Characteristic functions for infinite
sequences of noncommuting operators, {\it J. Operator Theory}
{\bf 22} (1989), 51--71.



      \bibitem{Po-von} {\sc G.~Popescu},
{Von Neumann inequality for $(B(H)^n)_1$,}
      {\it Math.  Scand.} {\bf 68} (1991), 292--304.









\bibitem{Po-disc} {\sc G.~Popescu}, Noncommutative disc algebras and their
  representations, {\it Proc. Amer. Math.} {\bf 124} (1996), 2137--2148.




      \bibitem{Po-poisson} {\sc G.~Popescu},
     {Poisson transforms on some $C^*$-algebras generated by isometries,}
       {\it J. Funct. Anal.} {\bf 161} (1999),  27--61.






     \bibitem{Po-holomorphic} {\sc G.~Popescu},
     {Free holomorphic functions on the unit ball of $B(\cH)^n$},
      {\it J. Funct. Anal.}  {\bf 241} (2006), 268--333.













\bibitem{Po-automorphism} {\sc G.~Popescu},
{ Free holomorphic automorphisms of the unit ball of $B(\cH)^n$},
{\it J. Reine Angew. Math.} {\bf 638} (2010), 119--168.


\bibitem{Po-classification}  {\sc G.~Popescu},
 Free biholomorphic classification of noncommutative domains {\it  Int. Math. Res. Not.} IMRN  {\bf 2011}, no. 4, 784--850.



\bibitem{Po-Berezin3}  {\sc G.~Popescu}, {Berezin transforms on
noncommutative  varieties in polydomains}, {\it  J. Funct. Anal.} {\bf 265} (2013), no. 10, 2500--2552.



\bibitem{Po-Berezin-poly} {\sc G.~Popescu}, {Berezin transforms on noncommutative polydomains}, {\it Trans. Amer. Math. Soc.}, to appear.




\bibitem{Po-curvature-polyball} {\sc G.~Popescu}, {Curvature invariant  on noncommutative polyballs},  submitted for publication.

\bibitem{Po-Euler-charact} {\sc G.~Popescu}, {Euler characteristic on noncommutative polyballs}, {\it J. Reine Angew. Math.}, to appear.



\bibitem{PS} {\sc S.C.~Power  and   B.~Solel},
{Operator algebras associated with unitary commutation relations},
 {\it J. Funct. Anal.} {\bf 260} (2011), no. 6, 1583--1614.


\bibitem{Ru1} {\sc W.~Rudin},
{\em Function theory in polydiscs}, W. A. Benjamin, Inc., New York-Amsterdam 1969 vii+188 pp.


\bibitem{Ru2} {\sc W.~Rudin},
{\em Function theory in the unit ball of \,$\CC^n$}, {
Springer-verlag, New-York/Berlin}, 1980.


\bibitem{T} {\sc Sh.~Tsyganov}, Biholomorphic maps of the direct product of domains,
    {\it Mat. Zametki.} {\bf 41} (1987), no.6, 824--828. (Russian and English translation)








\bibitem{Vo} {\sc D.~ Voiculescu},  Symmetries of some reduced free product $C\sp
*$-algebras,
 {\it Lecture Notes in Math.} {\bf 1132}, 556--588, Springer Verlag, New York,
 1985.











       \end{thebibliography}
      \end{document}